\documentclass[12pt]{article}
\pdfoutput=1

\usepackage{amsmath, amsthm, amssymb}
\usepackage{ifpdf}
\ifpdf
\usepackage[pdftex]{graphicx}
\else
\usepackage[dvips]{graphicx}
\fi
\usepackage{tikz}
 	 \usetikzlibrary{arrows,backgrounds} 
\usepackage[all]{xy}

\usepackage{enumerate}
\usepackage{fullpage}

\input xy
\xyoption{all}

\usepackage[pdftex,plainpages=false,hypertexnames=false,pdfpagelabels]{hyperref}
\newcommand{\arXiv}[1]{\href{http://arxiv.org/abs/#1}{\tt arXiv:\nolinkurl{#1}}}
\newcommand{\arxiv}[1]{\href{http://arxiv.org/abs/#1}{\tt arXiv:\nolinkurl{#1}}}
\newcommand{\doi}[1]{\href{http://dx.doi.org/#1}{{\tt DOI:#1}}}

\newcommand{\googlebooks}[1]{(preview at \href{http://books.google.com/books?id=#1}{google books})}

\usepackage{xcolor}
\definecolor{dark-red}{rgb}{0.7,0.25,0.25}
\definecolor{dark-blue}{rgb}{0.15,0.15,0.55}
\definecolor{medium-blue}{rgb}{0,0,.8}
\hypersetup{
   colorlinks, linkcolor={purple},
   citecolor={medium-blue}, urlcolor={medium-blue}
}

\theoremstyle{plain}
\newtheorem{thm}{Theorem}[section]
\newtheorem*{thm*}{Theorem}
\newtheorem{cor}[thm]{Corollary}
\newtheorem*{claim}{Claim}
\newtheorem{lem}[thm]{Lemma}
\newtheorem{prop}[thm]{Proposition}
\newtheorem{quest}[thm]{Question}
\newtheorem{quests}[thm]{Questions}
\theoremstyle{definition}
\newtheorem{defn}[thm]{Definition}
\newtheorem{nota}[thm]{Notation}
\newtheorem{assume}[thm]{Assumption}

\newtheorem{ex}[thm]{Example}
\newtheorem{rem}[thm]{Remark}
\newtheorem{rems}[thm]{Remarks}
\newtheorem{fact}[thm]{Fact}
\newtheorem{facts}[thm]{Facts}

\DeclareMathOperator{\Ad}{Ad}
\DeclareMathOperator{\CAR}{CAR}

\DeclareMathOperator{\GICAR}{GICAR}
\DeclareMathOperator{\End}{End}

\DeclareMathOperator{\id}{id}

\DeclareMathOperator{\OP}{op}
\DeclareMathOperator{\Out}{Out}

\DeclareMathOperator{\sign}{sign}
\DeclareMathOperator{\spann}{span}
\DeclareMathOperator{\Stab}{Stab}
\DeclareMathOperator{\Tr}{Tr}
\DeclareMathOperator{\tr}{tr}
\DeclareMathOperator{\wt}{wt}

\newcommand{\D}{\displaystyle}
\newcommand{\comment}[1]{}
\newcommand{\noshow}[1]{}
\newcommand{\MR}[1]{}

\newcommand{\be}{\begin{enumerate}[(1)]}
\newcommand{\ee}{\end{enumerate}}
\newcommand{\itm}[1]{\item[\underline{\ensuremath{#1}:}]}
\newcommand{\itt}[1]{\item[\underline{\text{#1}:}]}
\newcommand{\N}{\mathbb{N}}

\newcommand{\Z}{\mathbb{Z}}

\newcommand{\C}{\mathbb{C}}

\newcommand{\I}{\infty}
\newcommand{\set}[2]{\left\{#1 \middle| #2\right\}}

\renewcommand{\a}{\mathfrak{a}}

\newcommand{\n}{\mathfrak{n}}
\newcommand{\m}{\mathfrak{m}}
\newcommand{\jw}[1]{f^{(#1)}}
\newcommand{\op}{^{\OP}}

\def\semicolon{;}
\def\applytolist#1{
    \expandafter\def\csname multi#1\endcsname##1{
        \def\multiack{##1}\ifx\multiack\semicolon
            \def\next{\relax}
        \else
            \csname #1\endcsname{##1}
            \def\next{\csname multi#1\endcsname}
        \fi
        \next}
    \csname multi#1\endcsname}

\def\calc#1{\expandafter\def\csname c#1\endcsname{{\mathcal #1}}}
\applytolist{calc}QWERTYUIOPLKJHGFDSAZXCVBNM;
\def\bbc#1{\expandafter\def\csname bb#1\endcsname{{\mathbb #1}}}
\applytolist{bbc}QWERTYUIOPLKJHGFDSAZXCVBNM;
\def\bfc#1{\expandafter\def\csname bf#1\endcsname{{\mathbf #1}}}
\applytolist{bfc}QWERTYUIOPLKJHGFDSAZXCVBNM;
\def\sfc#1{\expandafter\def\csname s#1\endcsname{{\sf #1}}}
\applytolist{sfc}QWERTYUIOPLKJHGFDSAZXCVBNM;

\usetikzlibrary{calc}
\tikzstyle{shaded}=[fill=red!10!blue!20!gray!30!white]
\tikzstyle{unshaded}=[fill=white]
\tikzstyle{empty box}=[circle, draw, thick, fill=white, opaque, inner sep=2mm]
\tikzstyle{annular}=[scale=.65, inner sep=1mm, baseline=-.1cm]
\tikzstyle{rectangular}=[scale=.75, inner sep=1mm, baseline=-.1cm]

\newcommand{\nbox}[6]{
	\draw[very thick, #1] ($#2+(-#3,-#3)+(-#4,0)$) rectangle ($#2+(#3,#3)+(#5,0)$);
	\coordinate (a) at ($#2+(-#4,0)$);
	\coordinate (b) at ($#2+(#5,0)$);
	\node at ($1/2*(a)+1/2*(b)$) {#6};
}

\newcommand{\String}[1]{
	\draw ($#1+(0,-.4)$) -- ($#1+(0,.4)$);
}
\newcommand{\BrokenString}[1]{
	\draw ($#1+(0,-.4)$) -- ($#1+(0,-.15)$);
	\draw ($#1+(0,.4)$) -- ($#1+(0,.15)$);
	\filldraw ($#1+(0,-.15)$) circle (.05cm);
	\filldraw ($#1+(0,.15)$) circle (.05cm);
}
\newcommand{\DottedString}[1]{
	\draw ($#1+(0,-.4)$) -- ($#1+(0,.4)$);
	\filldraw #1 circle (.05cm);
}
\newcommand{\BrokenTopString}[1]{
	\draw ($#1+(0,.4)$) -- ($#1+(0,0)$);
	\filldraw ($#1+(0,0)$) circle (.05cm);
}
\newcommand{\BrokenBottomString}[1]{
	\draw ($#1+(0,-.4)$) -- ($#1+(0,0)$);
	\filldraw ($#1+(0,0)$) circle (.05cm);
}
\newcommand{\PartialIso}[1]{
	\draw ($#1+(0,.15)$) -- ($#1+(0,.4)$);
	\draw ($#1+(0,-.4)$) -- ($#1+(.2,.4)$);
	\draw ($#1+(.2,-.4)$) -- ($#1+(.2,-.15)$);
	\filldraw ($#1+(0,.15)$) circle (.05cm);
	\filldraw ($#1+(.2,-.15)$) circle (.05cm);
}
\newcommand{\PartialIsoStar}[1]{
	\draw ($#1+(0,-.15)$) -- ($#1+(0,-.4)$);
	\draw ($#1+(0,.4)$) -- ($#1+(.2,-.4)$);
	\draw ($#1+(.2,.4)$) -- ($#1+(.2,.15)$);
	\filldraw ($#1+(0,-.15)$) circle (.05cm);
	\filldraw ($#1+(.2,.15)$) circle (.05cm);
}
\newcommand{\PartialDottedIso}[1]{
	\draw ($#1+(0,.15)$) -- ($#1+(0,.4)$);
	\draw ($#1+(0,-.4)$) -- ($#1+(.2,.4)$);
	\draw ($#1+(.2,-.4)$) -- ($#1+(.2,-.15)$);
	\filldraw ($#1+(0,.15)$) circle (.05cm);
	\filldraw ($#1+(.2,-.15)$) circle (.05cm);
	\filldraw ($#1+(.1,0)$) circle (.05cm);
}
\newcommand{\PartialDottedIsoStar}[1]{
	\draw ($#1+(0,-.15)$) -- ($#1+(0,-.4)$);
	\draw ($#1+(0,.4)$) -- ($#1+(.2,-.4)$);
	\draw ($#1+(.2,.4)$) -- ($#1+(.2,.15)$);
	\filldraw ($#1+(0,-.15)$) circle (.05cm);
	\filldraw ($#1+(.2,.15)$) circle (.05cm);
	\filldraw ($#1+(.1,0)$) circle (.05cm);
}

\newcommand{\AnnularString}[1]{
	\draw ($(#1:.5cm)$) -- ($(#1:1.5cm)$);
}
\newcommand{\AnnularBrokenTopString}[1]{
	\draw ($(#1:1.15cm)$) -- ($(#1:1.5cm)$);
	\filldraw ($(#1:1.15cm)$) circle (.05cm);
}

\newcommand{\identity}{
\begin{tikzpicture}[baseline=-.1cm]
	\nbox{}{(0,0)}{.4}{-.2}{-.2}{}
	\String{(0,0)}	
\end{tikzpicture}}

\newcommand{\confetti}{
\begin{tikzpicture}[baseline=-.1cm]
	\nbox{}{(0,0)}{.4}{-.2}{-.2}{}
	\draw (0,-.15)--(0,.15);
	\filldraw (0,-.15) circle (.05cm);
	\filldraw (0,.15) circle (.05cm);
\end{tikzpicture}}

\newcommand{\eOne}{
\begin{tikzpicture}[baseline=-.1cm]
	\nbox{}{(0,0)}{.4}{-.2}{-.2}{}
	\BrokenString{(0,0)}	
\end{tikzpicture}}

\newcommand{\eOnePerp}{
\begin{tikzpicture}[baseline=-.1cm]
	\nbox{}{(0,0)}{.4}{-.2}{-.2}{}
	\DottedString{(0,0)}	
\end{tikzpicture}}

\newcommand{\doubleDotStrand}{
\begin{tikzpicture}[baseline=-.1cm]
	\nbox{}{(0,0)}{.4}{-.2}{-.2}{}
	\String{(0,0)}	
	\filldraw (0,.15) circle (.05cm);
	\filldraw (0,-.15) circle (.05cm);
\end{tikzpicture}}

\newcommand{\doubleDotEnd}{
\begin{tikzpicture}[baseline=-.1cm]
	\nbox{}{(0,0)}{.4}{-.2}{-.2}{}
	\draw (0,-.4)--(0,.15);
	\filldraw (0,.15) circle (.05cm);
	\filldraw (0,-.15) circle (.05cm);
\end{tikzpicture}}

\newcommand{\cwBrokenLoop}{
\begin{tikzpicture}[rectangular]
	\clip (.5,1.1) --(-.9,1.1) -- (-.9,-1.1) -- (.5,-1.1);
	\filldraw (-.4,-.15) circle (.05cm);
	\filldraw (-.4,.15) circle (.05cm);
	\draw (-.4,-.4)--(-.4,-.15);
	\draw (-.4,.4)--(-.4,.15);
	\draw [very thick] (-.6,-.4) --(-.6,.4) -- (-.2,.4) -- (-.2,-.4)--(-.6,-.4);
	\draw (-.4,.4) arc (180:0:.3cm) -- (.2,-.4) arc (0:-180:.3cm);
	\draw (.1,.1) -- (.2,0) -- (.3,.1);
	\draw [very thick] (.4,1) --(-.8,1) -- (-.8,-1) -- (.4,-1)--(.4,1);
\end{tikzpicture}}

\newcommand{\ccwBrokenLoop}{
\begin{tikzpicture}[rectangular,xscale=-1]
	\clip (.5,1.1) --(-.9,1.1) -- (-.9,-1.1) -- (.5,-1.1);
	\filldraw (-.4,-.15) circle (.05cm);
	\filldraw (-.4,.15) circle (.05cm);
	\draw (-.4,-.4)--(-.4,-.15);
	\draw (-.4,.4)--(-.4,.15);
	\draw [very thick] (-.6,-.4) --(-.6,.4) -- (-.2,.4) -- (-.2,-.4)--(-.6,-.4);
	\draw (-.4,.4) arc (180:0:.3cm) -- (.2,-.4) arc (0:-180:.3cm);
	\draw (.1,.1) -- (.2,0) -- (.3,.1);
	\draw [very thick] (.4,1) --(-.8,1) -- (-.8,-1) -- (.4,-1)--(.4,1);
\end{tikzpicture}}

\newcommand{\cwDottedLoop}{
\begin{tikzpicture}[rectangular]
	\clip (.9,.9) --(-.9,.9) -- (-.9,-.9) -- (.9,-.9);
	\draw (0,0) circle (.5cm);
	\filldraw (-.5,0) circle (.05cm);
	\draw (.6,.1) -- (.5,0) -- (.4,.1);
	\draw [very thick] (.8,.8) --(-.8,.8) -- (-.8,-.8) -- (.8,-.8)--(.8,.8);
\end{tikzpicture}}

\newcommand{\ccwDottedLoop}{
\begin{tikzpicture}[rectangular]
	\clip (.9,.9) --(-.9,.9) -- (-.9,-.9) -- (.9,-.9);
	\draw (0,0) circle (.5cm);
	\filldraw (.5,0) circle (.05cm);
	\draw (-.6,.1) -- (-.5,0) -- (-.4,.1);
	\draw [very thick] (.8,.8) --(-.8,.8) -- (-.8,-.8) -- (.8,-.8)--(.8,.8);
\end{tikzpicture}}

\begin{document}
\title{Infinite index subfactors and the GICAR categories}
\author{Vaughan F. R. Jones and  David Penneys }
\date{\today}
\maketitle
\begin{abstract}
Given a II$_1$-subfactor $A\subset B$ of arbitrary index, we show that the rectangular GICAR category, also called the rectangular planar rook category, faithfully embeds as $A-A$ bimodule maps among the bimodules $\bigotimes_A^n L^2(B)$. 
As a corollary, we get a lower bound on the dimension of the centralizer algebras $A_0'\cap A_{2n}$ for infinite index subfactors, and we also get that $A_0'\cap A_{2n}$ is nonabelian for $n\geq 2$, where $(A_n)_{n\geq 0}$ is the Jones tower for $A_0=A\subset B=A_1$.
We also show that the annular GICAR/planar rook category acts as maps amongst the $A$-central vectors in $\bigotimes_A^n L^2(B)$, although this action may be degenerate.
We prove these results in more generality using bimodules.

The embedding of the GICAR category builds on work of Connes and Evans who originally found GICAR algebras inside Temperley-Lieb algebras with finite modulus.
%This is the published version of \arxiv{...}.
\end{abstract}
%\tableofcontents
%%%%%%%%%%%%%%%%%%%%%%%%%%%%%%%%%%%%%%%%%%%%%%%%%%

\section{Introduction}

\subsection{Finite vs. infinite index}

In \cite{MR696688}, Jones pioneered the modern theory of subfactors.
Starting with a finite index {\rm II}$_1$-subfactor $A_0\subseteq A_1$, he used his basic construction to construct the Jones tower $(A_n)_{n\geq 0}$ iteratively by adding the Jones projections $(e_n)_{n\geq 1}$, which satisfy the Temperley-Lieb relations. 
Jones used these Temperley-Lieb algebras to show that the index lies in the range $\set{4\cos^2(\pi/n)}{n\geq 3}\cup [4,\I)$, and he found a hyperfinite subfactor for each allowed index value.

A finite index subfactor is studied by analyzing its standard invariant, the two towers of finite dimensional centralizer algebras $(A_i'\cap A_j)_{i=0,1;j\geq 0}$. 
The standard invariant has been axiomatized in three different ways: Ocneanu's paragroups \cite{MR996454}, Popa's $\lambda$-lattices \cite{MR1334479}, and Jones' planar algebras \cite{math/9909027}. 

Some finite index results generalize to infinite index subfactors.
Discrete, irreducible, ``depth $2$" subfactors correspond to outer (cocycle) actions of Kac algebras \cite{MR1055223,MR1387518}.
The classical Galois correspondence also holds for outer actions of infinite discrete groups and minimal actions of compact groups \cite{MR1622812}. 

Burns, in his Ph.D. thesis \cite{1111.1362}, studied extremality and rotations for infinite index subfactors, as the key ingredient in proving isotopy invariance for Jones' planar algebras in \cite{math/9909027} is the rotation operator (also known to Ocneanu \cite{MR1317353}). 
Essentially, Burns observed that for infinite index subfactors, the centralizer algebras $A_0'\cap A_{n}$ and the central $L^2$-vectors $A_0'\cap L^2(A_n)$ do not coincide.

Using this observation, the second author generalized the work of Burns in \cite{MR3040370}, where he gave a planar calculus for an arbitrary index {\rm II}$_1$-factor bimodule $\sb{A}H_A$. 
Setting $H^n=\bigotimes_A^n H$, he found two planar operads acting on the centralizer algebras $\cQ_n=A'\cap (A\op)'\cap B(H^n)$ and the central $L^2$-vectors $\cP_n=A'\cap H^n$ respectively whose actions are compatible.
We recover the subfactor case when $A=A_0$ and $H=L^2(A_1)$.
Interestingly, this planar structure was discovered without the use of Jones' basic construction and without the resulting Jones projections.

Hence we have one possible definition for the standard invariant of an infinite index subfactor, or a {\rm II}$_1$-factor bimodule: the centralizer algebras $\cQ_\bullet=(\cQ_n)_{n\geq 0}$ and the central $L^2$-vectors $\cP_\bullet=(\cP_n)_{n\geq 0}$, together with their compatible planar calculi.

\subsection{The simplest possible standard invariant}

The Jones subfactors with index at most 4 discovered in \cite{MR696688} have the simplest possible standard invariants; they consist entirely of the Temperley-Lieb algebras generated by the Jones projections.
Since these projections are always contained in the centralizer algebras, the Temperley-Lieb standard invariant is always contained within the standard invariant of a finite index subfactor.
Hence each subfactor planar algebra has a canonical Temperley-Lieb planar subalgebra. 

In \cite{MR1198815}, for every index greater than $4$, Popa found a (non-hyperfinite) subfactor whose standard invariant is only Temperley-Lieb, and his methods led to his famous subfactor reconstruction theorem \cite{MR1334479}.
An important open question is to determine for which indices greater than 4 there is a hyperfinite subfactor whose standard invariant is Temperley-Lieb.

The main motivation for this article is the following question.
\begin{quest}
For infinite index subfactors, what is the simplest possible standard invariant?
\end{quest}

When the index is infinite, one still has a Jones tower $(A_n)_{n\geq 0}$ of type {\rm II} factors, but $A_n$ is type {\rm II}$_\I$ for $n\geq 2$ (see Section \ref{sec:SubfactorBackground}).
In this case, Burns showed in \cite{1111.1362} that the odd canonical trace-preserving operator-valued weight $T_{2n+1} : A_{2n+1}\to A_{2n}$ is a conditional expectation, which results in an odd Jones projection $e_{2n+1}\in \cQ_{n+1}$. 
We immediately see that $\dim(\cQ_n)\geq n$, since the abelian algebra generated by the odd Jones projections is contained in $\cQ_n$.
However, the odd Jones projections actually give us non-abelian structure as well.

\begin{thm}
The odd Jones projections are equivalent in $\cQ_n$. 
Hence $\cQ_n$ is not abelian for $n\geq 2$.
\end{thm}

We prove this result in more generality for the case of a {\rm II}$_1$-factor bimodule $\sb{A}H_A$ containing a distinguished central vector $\zeta$, so $\cP_n\neq (0)$.
This is the natural analog of the bimodule $H=L^2(A_1)$ with distinguished $A_0$-central vector $\widehat{1}$.
We give the odd Jones projections for such bimodules in Section \ref{sec:BimoduleRepresentations}.

\subsection{GICAR and planar rook algebras and categories}

The Temperley-Lieb algebras appear implicitly in Lieb's ice-type model in statistical mechanics \cite{LiebIceModel,MR0498284}, \cite[Section 2.5]{JonesPANotes}.
The canonical algebra generated by the odd Jones projections together with the partial isometries witnessing the equivalences is actually another well-studied canonical operator algebra which arises in the study of fermions.

\begin{thm}\label{thm:Main2}
The gauge-invariant canonical anticommutation relations algebra $\GICAR(\cH_n)$ (also known as the fermion algebra) where $\dim(\cH_n)=n$ is represented faithfully in $\cQ_n$ as the odd Jones projections and the partial isometries between them.
\end{thm}

For finite index subfactors, this map was constructed by Connes and Evans in their work on representations of the Virasoro algebra \cite{MR990778}.
Our map is the bimodule analog, which is independent of von Neumann dimension.

In fact, there is a simple proof of the existence of such an injection, although further analysis is needed to show the image is correct.
Our distinguished $A$-central vector $\zeta\in H$ yields an $A-A$ bimodule isomorphism $H\cong L^2(A)\oplus K$ where $L^2(A)\cong \overline{A\zeta}^{\|\cdot\|_2}$ and $K\cong\{\zeta\}^\perp$. 
By the binomial theorem, 
$$
\bigotimes_{A}^n H 
\cong
(L^2(A)\oplus K)^{\otimes_A n}
\cong 
\bigoplus_{j=0}^n {n\choose j} K^j
$$
where $K^0=L^2(A)$, and $K^j=\bigotimes_A^j K$.
The obvious intertwiners amongst the $n\choose j$ copies of $K^j$ give a canonical inclusion 
$$
\bigoplus_{j=0}^n M_{n\choose j}(\C) 
\hookrightarrow
\End_{A-A}\left( \bigotimes^n_A H\right).
$$
The left hand side above is isomorphic to $\GICAR(\cH_n)$.

We compute our map explicitly in Section \ref{sec:RectangularGICARReps}, and we show it is compatible with the towers $\GICAR(\cH_\bullet)=(\GICAR(\cH_n))_{n\geq 0}$ and $\cQ_\bullet$, along with their standard representations. 
The GICAR tower arises from choosing an orthonormal basis $(\xi_n)$ of an infinite dimensional separable $\cH$, and setting $\cH_n=\spann\{\xi_1,\dots, \xi_n\}$.
Again, we do so in more generality:

\begin{thm}
The tower $\GICAR(\cH_\bullet)$ fits naturally into a ``rectangular" $*,\otimes$-category $\sR\sG$ which acts faithfully as $A-A$ bimodule maps amongst the $H^n$'s. This action extends the faithful representation from Theorem \ref{thm:Main2}.
\end{thm}

For a finite index subfactor, the image of Connes and Evans' map, which is also the map from Theorem \ref{thm:Main2}, consists of the Kauffman diagrams in the Temperley-Lieb algebra with only shaded caps and cups \cite[Lemma 4.2, Theorem 4.3 ]{MR990778}. 
Contracting shaded regions, we obtain a diagrammatic algebra which also appears in the literature as the planar rook algebra (see Section \ref{sec:PlanarRookAlgebras}).
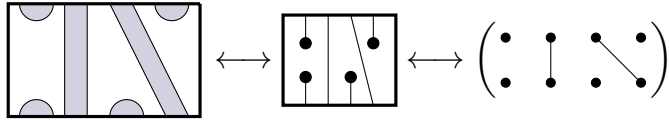
$$
\begin{tikzpicture}[rectangular]
	\clip (2.5,1.1) --(-1.1,1.1) -- (-1.1,-1.1) -- (2.5,-1.1);
	\filldraw[shaded] (1.8,-1)--(.8,1) -- (1.2,1) -- (2.2,-1);
	\filldraw[shaded] (0,1)--(0,-1) -- (.4,-1) -- (.4,1);
	\filldraw[shaded] (.8,-1) arc (180:0:.3cm);
	\filldraw[shaded] (-.8,-1) arc (180:0:.3cm);
	\filldraw[shaded] (-.8,1) arc (-180:0:.3cm);
	\filldraw[shaded] (1.6,1) arc (-180:0:.3cm);
	\draw [very thick] (2.4,1) --(-1,1) -- (-1,-1) -- (2.4,-1)--(2.4,1);
\end{tikzpicture}
\longleftrightarrow
\begin{tikzpicture}[baseline=-.1cm, scale=1.5]
	\nbox{}{(0,0)}{.4}{.2}{0}{}
	\BrokenString{(-.4,0)}
	\String{(-.2,0)}
	\PartialIsoStar{(0,0)}
\end{tikzpicture}
\longleftrightarrow
\left(
\begin{tikzpicture}[baseline=.2cm, scale=.6]
	\clip (-.2,-.2) -- (3.2,-.2) -- (3.2,1.2) -- (-.2,1.2);
	\draw (1,0) -- (1,1);
	\draw (2,1) -- (3,0);
	\filldraw (0,0) circle (.1cm);
	\filldraw (1,0) circle (.1cm);
	\filldraw (2,0) circle (.1cm);
	\filldraw (3,0) circle (.1cm);
	\filldraw (0,1) circle (.1cm);
	\filldraw (1,1) circle (.1cm);
	\filldraw (2,1) circle (.1cm);
	\filldraw (3,1) circle (.1cm);
\end{tikzpicture}
\right)$$

The representation theory of these diagrammatic algebras was studied in \cite{MR2541502}, where they showed the Bratteli diagram for the tower of algebras resulting from the right inclusion is Pascal's Triangle. 
Of course this also follows from the isomorphism with the tower of GICAR algebras (see Sections \ref{sec:GICAR}, \ref{sec:PlanarRookCategories}, and \ref{sec:DigrammaticFermions}), and we get a diagrammatic representation of the infinite dimensional GICAR algebra in Section \ref{sec:DigrammaticFermions}.
We remark that Bigelow-Ramos-Yi showed that the Jones and Alexander polynomials can be recovered via traces on the planar rook algebras \cite{MR2978881}.

Just as there is an annular version of the Temperley-Lieb category, there is an annular GICAR category $\sA\sG$, which contains the rectangular GICAR category $\sR\sG$.
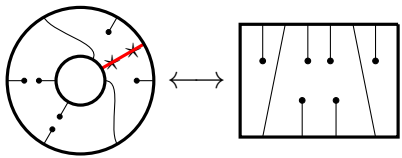
$$
\begin{tikzpicture}[annular]
	\clip (0,0) circle (1.6cm);
	\draw (0:1.15cm) -- (0:1.5cm);
	\filldraw (0:1.15cm) circle (.05cm);
	\draw (60:1.15cm) -- (60:1.5cm);
	\filldraw (60:1.15cm) circle (.05cm);
	\draw (60:0.5cm) .. controls ++(60:.4cm) and ++(-60:.4cm) .. (120:1.5cm);
	\draw (180:0.5cm) -- (180:.85cm);
	\filldraw (180:.85cm) circle (.05cm);
	\draw (180:1.15cm) -- (180:1.5cm);
	\filldraw (180:1.15cm) circle (.05cm);
	\draw (0:0.5cm) .. controls ++(0:.4cm) and ++(120:.4cm) .. (300:1.5cm);
	\draw (240:0.5cm) -- (240:.85cm);
	\filldraw (240:.85cm) circle (.05cm);
	\draw (240:1.15cm) -- (240:1.5cm);
	\filldraw (240:1.15cm) circle (.05cm);
	\node at (30:.75cm) {$\star$};
	\node at (30:1.25cm) {$\star$};
	\draw[very thick, red] (30:.5cm) -- (30:1.5cm);
	\draw [very thick] (0,0) circle (.5cm);
	\draw [very thick] (0,0) circle (1.5cm);
\end{tikzpicture}
\longleftrightarrow
\begin{tikzpicture}[rectangular]
	\clip (1.9,1.1) --(-1.1,1.1) -- (-1.1,-1.1) -- (1.9,-1.1);
	\draw (-.6,1)--(-.6,.35);
	\filldraw (-.6,.35) circle (.05cm);
	\draw (.2,1)--(.2,.35);
	\filldraw (.2,.35) circle (.05cm);
	\draw (.6,1)--(.6,.35);
	\filldraw (.6,.35) circle (.05cm);
	\draw (1.4,1)--(1.4,.35);
	\filldraw (1.4,.35) circle (.05cm);
	\draw (-.2,1)--(-.6,-1);
	\draw (1,1)--(1.4,-1);
	\draw (.1,-1)--(.1,-.35);
	\filldraw (.1,-.35) circle (.05cm);	
	\draw (.7,-1)--(.7,-.35);
	\filldraw (.7,-.35) circle (.05cm);	
	\draw [very thick] (1.8,1) --(-1,1) -- (-1,-1) -- (1.8,-1)--(1.8,1);
\end{tikzpicture}
$$
Diagrams in $\sA\sG$ are obtained from diagrams in $\sR\sG$ by tensoring the morphisms with themselves around the outside, i.e., gluing the rectangles into annuli, and then allowing for rotation.
This has two consequences:
\be
\item
$\sA\sG$ is no longer a tensor category, and
\item 
$\sA\sG$ must act on the spaces obtained from the $H^n$'s by tensoring themselves on the outside, i.e., the invariant vectors of the bimodules.
\ee
Using the Burns rotations studied in \cite{MR3040370}, we get the following theorem:

\begin{thm}
There is an action of the annular GICAR category $\sA\sG$ as maps amongst the sequence of central $L^2$-vectors $\cP_\bullet$.
\end{thm}

However, this action is not necessarily faithful, and there are subfactor examples where it is completely degenerate. 
This is in stark contrast to the finite index case, where the action of the annular Temperley-Lieb category is never degenerate.
In Section \ref{sec:Representations}, we compute the representation theory of $\sR\sG$ and $\sA\sG$ in the spirit of Graham and Lehrer's cellular algebras \cite{MR1376244}, as was done for the affine and annular Temperley-Lieb categories in \cite{MR1659204} and \cite{MR1929335} respectively.

\subsection{Examples}

By Theorem \ref{thm:Main2}, we see that $\cQ_n$ must contain $\GICAR(\cH_n)$.
In Examples \ref{ex:OuterZAction} and \ref{ex:OuterZAction2}, we give an example of a {\rm II}$_1$-factor bimodule with a distinguished central vector such that $\cQ_n$ is exactly the image of $\GICAR(\cH_n)$ and $\dim(\cP_n)=1$ for all $n\geq 0$.
However, this example does not come from a subfactor, and at this point, we do not have such an example.

We note that in the subfactor case, $\cP_{2n} \cong L^2(\cQ_n,\Tr_n)$ \cite[Remark 4.27]{MR3040370}, and the only Hilbert-Schmidt element in $\cQ_n$ in the image of $\GICAR(\cH_n)$ is the product of all the odd Jones projections, which can be identified with the element $\widehat{1}\otimes \cdots \otimes \widehat{1}\in H^n$ (see Example \ref{ex:SInfinity}).
In \cite{MR3040370} it was shown that when $H=L^2(A_1)$ for the subgroup-subfactor $A_0=R\rtimes \Stab(1)\subset R\rtimes S_\I=A_1$, we have $\dim(\cQ_n)<\I$ and $\dim(\cP_n)=1$ for all $n\geq 0$. 

%%%%%%%%%%%%%%%%%%%%%%%%%%%%%%%%%%%%%%%%%%%%%%%%%%
\subsection{Outline}

In Section \ref{sec:Fermions}, we give a background on fermionic Fock space and the CAR and GICAR algebras along with planar rook algebras.
In Section \ref{sec:GICARCategories}, we define the diagrammatic annular and rectangular planar rook categories and the abstract annular and rectangular GICAR categories, and we show they are respectively equivalent.
We then give the classification of the finite dimensional Hilbert space representations of the annular and rectangular categories in Section \ref{sec:Representations}.
We give the background necessary for our {\rm II}$_1$-factor bimodule and subfactor representations of these categories in Section \ref{sec:InfiniteBackground}, and we construct these representations in Section \ref{sec:BimoduleRepresentations}.

%%%%%%%%%%%%%%%%%%%%%%%%%%%%%%%%%%%%%%%%%%%%%%%%%%
\subsection{Future research}

We will continue to search for an example of an infinite index subfactor with the simplest possible standard invariant, or to attempt to show no such example exists.

In the recent article \cite{1110.5671}, the authors clarify the connection between bifinite Hilbert bimodules and two-sided dualizability.
Given an infinite index {\rm II}$_1$-subfactor $A\subset B$, the standard bimodule $\sb{A}L^2(B)_B$ is finite on only one side, so we have only one-sided duals.
In future work, we will clarify the connection between one-sided finite Hilbert bimodules and one-sided dualizability.
We will work with an operator-valued index for bimodules over finite von Neumann algebras which may be infinite in several distinct ways.
It would be interesting if there were different types of one-sided duals associated to the different flavors of one-sided finite index bimodules.

%%%%%%%%%%%%%%%%%%%%%%%%%%%%%%%%%%%%%%%%%%%%%%%%%%
\subsection{Acknowledgements}
This work was completed in three installments: 
while David Penneys visited Vanderbilt University in Spring 2011 (thanks to Dietmar Bisch and Jesse Peterson);
while both authors visited Institut Henri Poincar\'{e} during the 2011 trimester on von Neumann algebras and ergodic theory of groups actions (thanks to the organizers Damien Gaboriau, Sorin Popa, and Stefaan Vaes); and
during the Summer of 2013.

The second author would like to thank Michael Hartglass and James Tener for helpful conversations.
David Penneys was supported in part by the Natural Sciences and Engineering Research Council of Canada.
Both authors were also supported by 
NSF DMS grant 0856316
and
DOD-DARPA grants HR0011-11-1-0001 and HR0011-12-1-0009.

%%%%%%%%%%%%%%%%%%%%%%%%%%%%%%%%%%%%%%%%%%%%%%%%%%
%%%%%%%%%%%%%%%%%%%%%%%%%%%%%%%%%%%%%%%%%%%%%%%%%%
%%%%%%%%%%%%%%%%%%%%%%%%%%%%%%%%%%%%%%%%%%%%%%%%%%
\section{Fermions and planar rook algebras}\label{sec:Fermions}

In this section, we give the background material on fermionic Fock space, the CAR and GICAR algebras, and planar rook algebras.

%%%%%%%%%%%%%%%%%%%%%%%%%%%%%%%%%%%%%%%%%%%%%%%%%%
\subsection{Fermions, CAR, and GICAR}\label{sec:GICAR}

We take the following definitions from \cite[Chapter 18]{JonesVNA}.
Suppose $\cH$ is a Hilbert space. 

\begin{defn}
The $n$-th exterior power of $\cH$ is $\Lambda^n\cH=p_n\bigotimes^n \cH$, where $p_n$ is the projection given by
$$
p_n(\xi_1\otimes\cdots \otimes \xi_n ) = \frac{1}{n!}\sum_{\sigma\in S_n} (-1)^{\sign(\sigma)} \xi_{\sigma(1)}\otimes\cdots\otimes \xi_{\sigma(n)}.
$$
The fermionic Fock space $\cF(\cH)$ is given by $\cF(H)=\bigoplus_{n\geq 0} \Lambda^n \cH$.
Given $\xi_1,\dots,\xi_n\in\cH$, we set
$$
\xi_1\wedge\cdots\wedge \xi_n = \sqrt{n!}\,\, p_n(\xi_1\otimes\cdots\otimes \xi_n).
$$
The inner product on $\cF(\cH)$ is given by
$$
\langle 
\eta_1\wedge\cdots \wedge \eta_n,
\xi_1\wedge\cdots\wedge \xi_n
\rangle
=
\det\big((\langle \eta_i,\xi_j\rangle)_{i,j}\big).
$$
For $f\in \cH$, the left creation operator $a(f)$ is given by the unique linear extension of
$$
a(f)(\xi_1\wedge\cdots \wedge \xi_n) = f\wedge \xi_1\wedge\cdots\wedge \xi_n),
$$
and its adjoint is given by
$$
a(f)^*(\xi_1\wedge\cdots \wedge \xi_n) = \sum_{i=1}^n (-1)^{i+1} \langle \xi_i,f\rangle(\xi_1\wedge\cdots \wedge\widehat{\xi_i}\wedge\cdots\wedge \xi_n).
$$
\end{defn}

\begin{rem}
The wave function of several fermions is antisymmetric, so the exterior power $\Lambda^n(\cH)$ describes $n$ identical fermions.
The fermionic Fock space $\cF(\cH)$ is used to treat a countably infinite family of fermions.
\end{rem}

\begin{defn}
If $\cH$ is a complex vector space, the canonical anticommutation relations algebra $\CAR(\cH)$ is the unital $*$-algebra with generators $a(f)$ for $f\in \cH$ subject to the following relations:
\begin{align}
&\text{The map $f\mapsto a(f)$ is linear.}
\tag{CAR1}
\label{rel:CAR1}
\\
&\text{$a(f)a(g)+a(g)a(f)=0$ for all $f,g\in\cH$.}
\tag{CAR2}
\label{rel:CAR2}
\\
&\text{$a(f)a(g)^*+a(g)^*a(f)=\langle f,g\rangle 1_{B(\cH)}$ for all $f,g\in\cH$.}
\tag{CAR3}
\label{rel:CAR3}
\end{align}
\end{defn}

\begin{fact}
There is a unique C* norm and normalized trace on $\CAR(\cH)$.
\end{fact}

\begin{defn}
Given a $u\in U(\cH)$, the Bogoliubov automorphism $\alpha_u$ of $\CAR(\cH)$ is given by $\alpha_u(a(f))=a(uf)$ for all $f\in\cH$.
\end{defn}

\begin{defn}
The gauge-invariant canonical anticommutation relations algebra $\GICAR(\cH)$ is $\CAR(\cH)^{U(1)}$, where $U(1)$ is the scalars acting by Bogoliubov automorphisms on $\CAR(\cH)$.
\end{defn}

\begin{fact}\label{fact:ChooseBasis}
Suppose $\cH$ is separable and infinite dimensional with a fixed choice of orthonormal basis $(\xi_i)_{i\geq 1}$. Let $\cH_n=\spann\{\xi_1,\dots, \xi_n\}$, and define $A_n = \CAR(\cH_n)$ and $G_n=A_n^{U(1)}=\GICAR(\cH_n)$.
We use the abbreviation $a_i=a(\xi_i)$ and $a_i^*=a(\xi_i)^*$ for all $i\geq 1$.

The inclusion $\cH_n\hookrightarrow \cH_{n+1}$ induces inclusions of algebras $A_n\hookrightarrow A_{n+1}$ and $G_n\hookrightarrow G_{n+1}$.
A straightforward calculation (e.g., see \cite[Examples III.5.4-5]{MR1402012}) shows
\begin{align*}
A_n &=C^*\{a_1,\dots, a_n,a_1^*,\dots,a_n^*\}
\cong \bigotimes^n_{k=1} M_2(\C)\cong M_{2^n}(\C) \text{ and}\\
G_n &=\spann\set{a_{i_k}\cdots a_{i_1}a_{j_1}^*\cdots a_{j_k}^*}{i_1<\cdots <i_k,\,\, j_1<\cdots <j_k,\,\,k\in\N}
\cong \bigoplus_{k=0}^n M_{n\choose k}(\C),
\end{align*}
where the inclusion $A_n\hookrightarrow A_{n+1}$ is given by $x\mapsto x\otimes 1$, and the Bratteli diagram for the tower $(G_n)_{n\geq 1}$ is Pascal's triangle.
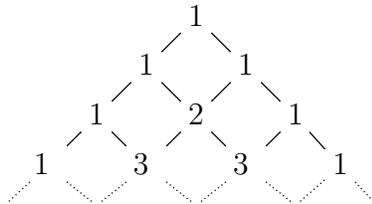
\[
\xymatrix@R=5pt@C=5pt{
&&&&1\ar@{-}[dr]\ar@{-}[dl] \\
&&&1\ar@{-}[dr]\ar@{-}[dl]  & & 1\ar@{-}[dr]\ar@{-}[dl] \\
&&1\ar@{-}[dr]\ar@{-}[dl]  && 2\ar@{-}[dr]\ar@{-}[dl]  && 1\ar@{-}[dr]\ar@{-}[dl] \\
&1\ar@{..}[dr]\ar@{..}[dl]\ &&3\ar@{..}[dr]\ar@{..}[dl]\ && 3\ar@{..}[dr]\ar@{..}[dl]\ && 1\ar@{..}[dr]\ar@{..}[dl]\\
&&&&&&&&&&
}
\]
\end{fact}

The following facts are well known about the GICAR algebras.
We provide a short proof for the convenience of the reader.
Let $\cH$, $(\xi_i)_{i\geq 1}$, $\cH_n$, and $G_n$ be as in Fact \ref{fact:ChooseBasis}.

\begin{thm}\label{thm:GICARRepresentations}
\mbox{}
\be
\item
The representation of $G_n$ on $\Lambda^k \cH_n$ is irreducible.
\item
The left regular representation of $G_n$ breaks up as 
$$
G_n \cong \bigoplus_{k=0}^n {n\choose k} \Lambda^k\cH_n.
$$
Thus the complete list of irreducible representations of $G_n$ is $\set{\Lambda^k\cH_n}{k=0,\dots, n}$.
\item
When restricted to the image of $G_{n-1}$ in $G_{n}$, 
$$
\Lambda^{k}\cH_{n}\cong \Lambda^{k-1}\cH_{n-1}\oplus \Lambda^{k}\cH_{n-1},
$$
where $\Lambda^{k-1}\cH_{n-1}=(0)$ if $k=0$ and $\Lambda^{k+1}\cH_n=(0)$ if $k=n$.
\ee
\end{thm}
\begin{proof}
\mbox{}
\be
\item
This is straightforward. 
One can use that any vector of the form $\xi_{i_1}\wedge \cdots \wedge \xi_{i_k}$ with $i_1<\cdots <i_k$ and $k\leq n$ generates $\Lambda^k \cH_n$ as a $G_n$-module.
\item
By Relations \eqref{rel:CAR2}-\eqref{rel:CAR3}, for each $i=1,\dots, n$, the operators $a_ia_i^*$ are commuting projections.
The words in $a_ia_i^*$ and $a_j^*a_j=1-a_ja_j^*$ for which all subscripts $1,\dots, n$ appear give the $2^n$ minimal projections in $G_n$.
Thus there are exactly $n\choose k$ minimal projections in $G_n$ with exactly $k$ projections  $a_ja_j^*$ appearing in the word.
For each of these minimal projections $p$, the left $G_n$-module $G_np$ is isomorphic to $\Lambda^k\cH_n$ via the map
$$
p=\prod_{\ell=1}^k a_{i_\ell}a_{i_\ell}^* \prod_{\ell=1}^{n-k} a_{j_\ell}^*a_{j_\ell}\longmapsto \xi_{i_1}\wedge \cdots \wedge \xi_{i_k}.
$$
Since minimal projections in a multi-matrix algebra generate equivalent representations if and only if the projections are equivalent, the result follows from Fact \ref{fact:ChooseBasis}. The last statement follows from the Artin-Wedderburn Theorem.
\item
We have two invariant subspaces of $\Lambda^k\cH_n$ under the action of $G_{n-1}$, namely
$$
\hspace{-.5cm}
\spann\set{\xi_{i_1}\wedge \cdots\wedge \xi_{i_k}}{i_1<\cdots <i_k<n} 
\text{ and }
\spann\set{\xi_{i_1}\wedge \cdots\wedge \xi_{i_k}}{i_1<\cdots <i_k=n}.
$$
Both subspaces are irreducible under the action of $G_{n-1}$ as in (1), the latter because $G_{n-1}$ never moves $\xi_n$. 
The first is isomorphic to $\Lambda^{k}\cH_n$ since $\xi_n$ never appears, and the second is isomorphic to $\Lambda^{k-1}\cH_n$ since we may ignore the $\xi_n$ which never moves.
If $k=n$, the first subspace is $(0)$, and if $k=1$, $G_{n-1}$ acts as the zero algebra on the second subspace.
\qedhere
\ee
\end{proof}

\begin{rems}
\mbox{}
\be
\item
Theorem \ref{thm:GICARRepresentations} part (3) gives another proof that the Bratteli diagram for the tower $(G_n)_{n\geq 0}$ where $G_0=\C$ is Pascal's Triangle.
\item
Remark \ref{rem:DiagrammaticGICARRep} gives a nice diagrammatic description of the representations in Theorem \ref{thm:GICARRepresentations} part (2).
\ee
\end{rems}

%%%%%%%%%%%%%%%%%%%%%%%%%%%%%%%%%%%%%%%%%%%%%%%%%%
\subsection{Rook monoids and planar rook algebras}\label{sec:PlanarRookAlgebras}

\begin{defn}
Let $R_n$ be the set of all $n\times n$ zero-one matrices with at most one entry equal to one in each row and column. 
Then $R_n$ is a monoid under matrix multiplication.
In \cite{MR1939108}, the author named $R_n$ the \underline{rook monoid}, since the matrices are in one-to-one correspondence with placements of non-attacking rooks on an $n\times n$ chessboard.
\end{defn}

\begin{ex}
The rook monoid $R_2$ consists of the following matrices
$$
R_2=
\left\{
\begin{pmatrix}
0 & 0\\
0 & 0
\end{pmatrix}
\,,\,
\begin{pmatrix}
1 & 0\\
0 & 0
\end{pmatrix}
\,,\,
\begin{pmatrix}
0 & 1\\
0 & 0
\end{pmatrix}
\,,\,
\begin{pmatrix}
0 & 0\\
1 & 0
\end{pmatrix}
\,,\,
\begin{pmatrix}
0 & 0\\
0 & 1
\end{pmatrix}
\,,\,
\begin{pmatrix}
0 & 1\\
1 & 0
\end{pmatrix}
\,,\,
\begin{pmatrix}
1 & 0\\
0 & 1
\end{pmatrix}
\right\}.
$$
\end{ex}

In \cite{MR2541502}, a diagrammatic description of the rook monoid was given as follows. Since each matrix in $R_n$ has at most one 1 in each row and column, we can identify it with a bipartite graph on two rows of $n$ vertices such that each node has degree 0 or 1.
If the $(i,j)$-th entry of $x\in R_n$ is 1, then we connect the $i$-th node on the top row to the $j$-th node on the bottom row. 
For example, the matrices in $R_2$ are identified with the following diagrams:
$$
\left(
\begin{tikzpicture}[baseline=.2cm, scale=.6]
	\clip (-.2,-.2) -- (1.2,-.2) -- (1.2,1.2) -- (-.2,1.2);
	\filldraw (0,0) circle (.1cm);
	\filldraw (1,0) circle (.1cm);
	\filldraw (0,1) circle (.1cm);
	\filldraw (1,1) circle (.1cm);
\end{tikzpicture}
\right)
\,,\,
\left(
\begin{tikzpicture}[baseline=.2cm, scale=.6]
	\clip (-.2,-.2) -- (1.2,-.2) -- (1.2,1.2) -- (-.2,1.2);
	\draw (0,0) -- (0,1);
	\filldraw (0,0) circle (.1cm);
	\filldraw (1,0) circle (.1cm);
	\filldraw (0,1) circle (.1cm);
	\filldraw (1,1) circle (.1cm);
\end{tikzpicture}
\right)
\,,\,
\left(
\begin{tikzpicture}[baseline=.2cm, scale=.6]
	\clip (-.2,-.2) -- (1.2,-.2) -- (1.2,1.2) -- (-.2,1.2);
	\draw (1,0) -- (0,1);
	\filldraw (0,0) circle (.1cm);
	\filldraw (1,0) circle (.1cm);
	\filldraw (0,1) circle (.1cm);
	\filldraw (1,1) circle (.1cm);
\end{tikzpicture}
\right)
\,,\,
\left(
\begin{tikzpicture}[baseline=.2cm, scale=.6]
	\clip (-.2,-.2) -- (1.2,-.2) -- (1.2,1.2) -- (-.2,1.2);
	\draw (0,0) -- (1,1);
	\filldraw (0,0) circle (.1cm);
	\filldraw (1,0) circle (.1cm);
	\filldraw (0,1) circle (.1cm);
	\filldraw (1,1) circle (.1cm);
\end{tikzpicture}
\right)
\,,\,
\left(
\begin{tikzpicture}[baseline=.2cm, scale=.6]
	\clip (-.2,-.2) -- (1.2,-.2) -- (1.2,1.2) -- (-.2,1.2);
	\draw (1,0) -- (1,1);
	\filldraw (0,0) circle (.1cm);
	\filldraw (1,0) circle (.1cm);
	\filldraw (0,1) circle (.1cm);
	\filldraw (1,1) circle (.1cm);
\end{tikzpicture}
\right)
\,,\,
\left(
\begin{tikzpicture}[baseline=.2cm, scale=.6]
	\clip (-.2,-.2) -- (1.2,-.2) -- (1.2,1.2) -- (-.2,1.2);
	\draw (0,0) -- (1,1);
	\draw (1,0) -- (0,1);
	\filldraw (0,0) circle (.1cm);
	\filldraw (1,0) circle (.1cm);
	\filldraw (0,1) circle (.1cm);
	\filldraw (1,1) circle (.1cm);
\end{tikzpicture}
\right)
\,,\,
\left(
\begin{tikzpicture}[baseline=.2cm, scale=.6]
	\clip (-.2,-.2) -- (1.2,-.2) -- (1.2,1.2) -- (-.2,1.2);
	\draw (0,0) -- (0,1);
	\draw (1,0) -- (1,1);
	\filldraw (0,0) circle (.1cm);
	\filldraw (1,0) circle (.1cm);
	\filldraw (0,1) circle (.1cm);
	\filldraw (1,1) circle (.1cm);
\end{tikzpicture}
\right)
$$
Multiplicaiton then corresponds to vertical concatenation of diagrams up to isotopy, where we contract any edge which does not reach the other side, and we delete the middle nodes, e.g.,
\begin{align*}
\begin{pmatrix}
0 & 1 & 0\\
0 & 0 & 0\\
0 & 0 & 1
\end{pmatrix}
\begin{pmatrix}
0 & 0 & 0\\
1 & 0 & 0\\
0 & 0 & 0
\end{pmatrix}
&=
\begin{pmatrix}
1 & 0 & 0\\
0 & 0 & 0\\
0 & 0 & 0
\end{pmatrix}
\\&\updownarrow\\
\left(
\begin{tikzpicture}[baseline=.2cm, scale=.6]
	\clip (-.2,-.2) -- (2.2,-.2) -- (2.2,1.2) -- (-.2,1.2);
	\draw (0,1) -- (1,0);
	\draw (2,0) -- (2,1);
	\filldraw (0,0) circle (.1cm);
	\filldraw (1,0) circle (.1cm);
	\filldraw (2,0) circle (.1cm);
	\filldraw (0,1) circle (.1cm);
	\filldraw (1,1) circle (.1cm);
	\filldraw (2,1) circle (.1cm);
\end{tikzpicture}
\right)
\left(
\begin{tikzpicture}[baseline=.2cm, scale=.6]
	\clip (-.2,-.2) -- (2.2,-.2) -- (2.2,1.2) -- (-.2,1.2);
	\draw (0,0) -- (1,1);
	\filldraw (0,0) circle (.1cm);
	\filldraw (1,0) circle (.1cm);
	\filldraw (2,0) circle (.1cm);
	\filldraw (0,1) circle (.1cm);
	\filldraw (1,1) circle (.1cm);
	\filldraw (2,1) circle (.1cm);
\end{tikzpicture}
\right)
&=
\left(
\begin{tikzpicture}[baseline=.5cm, scale=.6]
	\clip (-.2,-.2) -- (2.2,-.2) -- (2.2,2.2) -- (-.2,2.2);
	\draw (0,0) -- (1,1);
	\draw (0,2) -- (1,1);
	\draw (2,1) -- (2,2);
	\filldraw (0,0) circle (.1cm);
	\filldraw (1,0) circle (.1cm);
	\filldraw (2,0) circle (.1cm);
	\filldraw (0,1) circle (.1cm);
	\filldraw (1,1) circle (.1cm);
	\filldraw (2,1) circle (.1cm);
	\filldraw (0,2) circle (.1cm);
	\filldraw (1,2) circle (.1cm);
	\filldraw (2,2) circle (.1cm);
\end{tikzpicture}
\right)
=
\left(
\begin{tikzpicture}[baseline=.2cm, scale=.6]
	\clip (-.2,-.2) -- (2.2,-.2) -- (2.2,1.2) -- (-.2,1.2);
	\draw (0,0) -- (0,1);
	\filldraw (0,0) circle (.1cm);
	\filldraw (1,0) circle (.1cm);
	\filldraw (2,0) circle (.1cm);
	\filldraw (0,1) circle (.1cm);
	\filldraw (1,1) circle (.1cm);
	\filldraw (2,1) circle (.1cm);
\end{tikzpicture}
\right),
\end{align*}
and the adjoint corresponds to vertical reflection
$$
\begin{pmatrix}
0 & 1\\
0 & 0
\end{pmatrix}
^*
=
\begin{pmatrix}
0 & 0\\
1 & 0
\end{pmatrix}
\longleftrightarrow
\left(
\begin{tikzpicture}[baseline=.2cm, scale=.6]
	\clip (-.2,-.2) -- (1.2,-.2) -- (1.2,1.2) -- (-.2,1.2);
	\draw (0,1) -- (1,0);
	\filldraw (0,0) circle (.1cm);
	\filldraw (1,0) circle (.1cm);
	\filldraw (0,1) circle (.1cm);
	\filldraw (1,1) circle (.1cm);
\end{tikzpicture}
\right)
^*
=
\left(
\begin{tikzpicture}[baseline=.2cm, scale=.6]
	\clip (-.2,-.2) -- (1.2,-.2) -- (1.2,1.2) -- (-.2,1.2);
	\draw (0,0) -- (1,1);
	\filldraw (0,0) circle (.1cm);
	\filldraw (1,0) circle (.1cm);
	\filldraw (0,1) circle (.1cm);
	\filldraw (1,1) circle (.1cm);
\end{tikzpicture}
\right).
$$

\begin{defn}
The \underline{planar rook monoid} \cite{MR2541502} $P_n$ consists of the subset of $R_n$ for which the corresponding graphs are planar.
For example, 
$$
P_2 = 
\left\{
\left(
\begin{tikzpicture}[baseline=.2cm, scale=.6]
	\clip (-.2,-.2) -- (1.2,-.2) -- (1.2,1.2) -- (-.2,1.2);
	\filldraw (0,0) circle (.1cm);
	\filldraw (1,0) circle (.1cm);
	\filldraw (0,1) circle (.1cm);
	\filldraw (1,1) circle (.1cm);
\end{tikzpicture}
\right)
\,,\,
\left(
\begin{tikzpicture}[baseline=.2cm, scale=.6]
	\clip (-.2,-.2) -- (1.2,-.2) -- (1.2,1.2) -- (-.2,1.2);
	\draw (0,0) -- (0,1);
	\filldraw (0,0) circle (.1cm);
	\filldraw (1,0) circle (.1cm);
	\filldraw (0,1) circle (.1cm);
	\filldraw (1,1) circle (.1cm);
\end{tikzpicture}
\right)
\,,\,
\left(
\begin{tikzpicture}[baseline=.2cm, scale=.6]
	\clip (-.2,-.2) -- (1.2,-.2) -- (1.2,1.2) -- (-.2,1.2);
	\draw (1,0) -- (0,1);
	\filldraw (0,0) circle (.1cm);
	\filldraw (1,0) circle (.1cm);
	\filldraw (0,1) circle (.1cm);
	\filldraw (1,1) circle (.1cm);
\end{tikzpicture}
\right)
\,,\,
\left(
\begin{tikzpicture}[baseline=.2cm, scale=.6]
	\clip (-.2,-.2) -- (1.2,-.2) -- (1.2,1.2) -- (-.2,1.2);
	\draw (0,0) -- (1,1);
	\filldraw (0,0) circle (.1cm);
	\filldraw (1,0) circle (.1cm);
	\filldraw (0,1) circle (.1cm);
	\filldraw (1,1) circle (.1cm);
\end{tikzpicture}
\right)
\,,\,
\left(
\begin{tikzpicture}[baseline=.2cm, scale=.6]
	\clip (-.2,-.2) -- (1.2,-.2) -- (1.2,1.2) -- (-.2,1.2);
	\draw (1,0) -- (1,1);
	\filldraw (0,0) circle (.1cm);
	\filldraw (1,0) circle (.1cm);
	\filldraw (0,1) circle (.1cm);
	\filldraw (1,1) circle (.1cm);
\end{tikzpicture}
\right)
\,,\,
\left(
\begin{tikzpicture}[baseline=.2cm, scale=.6]
	\clip (-.2,-.2) -- (1.2,-.2) -- (1.2,1.2) -- (-.2,1.2);
	\draw (0,0) -- (0,1);
	\draw (1,0) -- (1,1);
	\filldraw (0,0) circle (.1cm);
	\filldraw (1,0) circle (.1cm);
	\filldraw (0,1) circle (.1cm);
	\filldraw (1,1) circle (.1cm);
\end{tikzpicture}
\right)
\right\}
$$
The \underline{planar rook algebra} $\C P_n$ is the complex $*$-algebra spanned by $P_n$.
\end{defn}

\begin{fact}
The representation theory of $\C P_n$ was classified in \cite{MR2541502}. 
Moreover, it was shown that $\C P_n \cong \bigoplus_{k=0}^n M_{n\choose k}(\C)$, and the Bratteli diagram for the tower of algebras $(\C P_n)_{n\geq 0}$ is Pascals' Triangle, where the unital inclusion $\C P_n \hookrightarrow \C P_{n+1}$ is given by adding a through string on the right:
$$
\left(
\begin{tikzpicture}[baseline=.2cm, scale=.6]
	\clip (-.2,-.2) -- (1.2,-.2) -- (1.2,1.2) -- (-.2,1.2);
	\draw (0,0) -- (1,1);
	\filldraw (0,0) circle (.1cm);
	\filldraw (1,0) circle (.1cm);
	\filldraw (0,1) circle (.1cm);
	\filldraw (1,1) circle (.1cm);
\end{tikzpicture}
\right)
\longmapsto
\left(
\begin{tikzpicture}[baseline=.2cm, scale=.6]
	\clip (-.2,-.2) -- (2.2,-.2) -- (2.2,1.2) -- (-.2,1.2);
	\draw (0,0) -- (1,1);
	\draw (2,0) -- (2,1);
	\filldraw (0,0) circle (.1cm);
	\filldraw (1,0) circle (.1cm);
	\filldraw (2,0) circle (.1cm);
	\filldraw (0,1) circle (.1cm);
	\filldraw (1,1) circle (.1cm);
	\filldraw (2,1) circle (.1cm);
\end{tikzpicture}
\right)
$$
Hence the tower $(G_n)_{n\geq 0}$ is isomorphic to the tower $(\C P_n)_{n\geq 0}$. 
\end{fact}

We give an independently found short proof of the isomorphism of towers in Proposition \ref{prop:Pascal} using a notational trick due to Bigelow.
After we establish that the towers are isomorphic, we immediately get the representation theory of the $\C P_n$ from the well-known representation theory of the GICAR algebras given in Theorem \ref{thm:GICARRepresentations}.
In Remark \ref{rem:DiagrammaticGICARRep}, we give a diagrammatic description of these representations in the spirit of Graham and Lehrer's cellular algebras \cite{MR1376244}.

\begin{rem}
We discovered these diagrams in a completely different way. 
The Temperley-Lieb diagrams in $TL_{2n}(\delta)$ with only shaded caps and cups are in one-to-one correspondence with the diagrams in $P_n$. One sees this by contracting the cups and caps to nodes and contracting shaded regions to lines as in Figure \ref{fig:ShadedTLDiagrams}. To make the multiplication agree on the nose, one must include a factor of $\delta$ for each maxima in the Temperley-Lieb diagram. (Note that the number of maxima must equal the number of minima).

\begin{figure}[!ht]
$$
\begin{tikzpicture}[rectangular]
	\clip (2.5,1.1) --(-1.1,1.1) -- (-1.1,-1.1) -- (2.5,-1.1);
	\filldraw[shaded] (1.8,-1)--(.8,1) -- (1.2,1) -- (2.2,-1);
	\filldraw[shaded] (0,1)--(0,-1) -- (.4,-1) -- (.4,1);
	\filldraw[shaded] (.8,-1) arc (180:0:.3cm);
	\filldraw[shaded] (-.8,-1) arc (180:0:.3cm);
	\filldraw[shaded] (-.8,1) arc (-180:0:.3cm);
	\filldraw[shaded] (1.6,1) arc (-180:0:.3cm);
	\draw [very thick] (2.4,1) --(-1,1) -- (-1,-1) -- (2.4,-1)--(2.4,1);
\end{tikzpicture}
\longleftrightarrow
\left(
\begin{tikzpicture}[baseline=.2cm, scale=.6]
	\clip (-.2,-.2) -- (3.2,-.2) -- (3.2,1.2) -- (-.2,1.2);
	\draw (1,0) -- (1,1);
	\draw (2,1) -- (3,0);
	\filldraw (0,0) circle (.1cm);
	\filldraw (1,0) circle (.1cm);
	\filldraw (2,0) circle (.1cm);
	\filldraw (3,0) circle (.1cm);
	\filldraw (0,1) circle (.1cm);
	\filldraw (1,1) circle (.1cm);
	\filldraw (2,1) circle (.1cm);
	\filldraw (3,1) circle (.1cm);
\end{tikzpicture}
\right)$$
\caption{TL diagrams with only shaded cups/caps and planar rook diagrams}
\label{fig:ShadedTLDiagrams}
\end{figure}
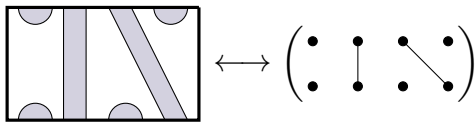

Note that this is the same map $G_n \to TL_{2n}(\delta)$ found by Evans and Connes \cite[Theorem 4.3]{MR990778} without the use of Kauffman diagrams \cite{MR899057}.
They showed this map is injective regardless of $\delta$ by verifying the minimal projections in $G_n$ (see Proposition \ref{prop:Pascal}) map to nonzero orthogonal projections in $TL_{2n}(\delta)$.
We will show a modification of this map works for infinite index subfactors (see Theorem \ref{thm:GICARinQ}).
\end{rem}

%%%%%%%%%%%%%%%%%%%%%%%%%%%%%%%%%%%%%%%%%%%%%%%%%%
%%%%%%%%%%%%%%%%%%%%%%%%%%%%%%%%%%%%%%%%%%%%%%%%%%
%%%%%%%%%%%%%%%%%%%%%%%%%%%%%%%%%%%%%%%%%%%%%%%%%%
\section{Annular and rectangular GICAR categories}\label{sec:GICARCategories}

Just as the Temperley-Lieb algebras can be thought of as a category, so can the planar rook algebras.
We discuss two realizations of this category, which we show are equivalent: 
a diagrammatic category, which we call the rectangular planar rook category, 
and an abstract category via generators and relations, which we call the rectangular GICAR category.
We also have the notion of the annular planar rook and GICAR categories, which we show are equivalent.

Along the way, we will take a brief detour to discuss a diagrammatic representation of the GICAR algebra.

\begin{nota}
We denote categories using the sans-serif font $\sA\sB\sC\dots$
Given a category $\sC$, we write $X,Y\in\sC$ to denote $X,Y$ are objects in $\sC$, and we write $\sC(X,Y)$ for the space of morphisms from $X$ to $Y$. 
If the objects in $\sC$ are symbols of the form $[n]$ for $n\geq 0$, we simply write $\sC(m,n)$ for $\sC([m],[n])$.
We further simplify notation by writing $\sC_n$ for $\sC(n,n)$.
\end{nota}

%%%%%%%%%%%%%%%%%%%%%%%%%%%%%%%%%%%%%%%%%%%%%%%%%%
\subsection{Annular and rectangular planar rook categories}\label{sec:PlanarRookCategories}

\begin{defn}
The \underline{annular planar rook category} $\sA\sP$ is the following small involutive category:
\itt{Objects} $[n]$ for $n\geq 0$, and
\itt{Morphisms} 
$\sA\sP(m,n)$ is all $\C$-linear combinations of isotopy classes of tangles on annuli with decoration as follows.
\begin{itemize}
\item There are $m$ marked points on the inner boundary, called the inner points, and $n$ marked points on the external boundary, called the outer points.
\item Each marked point is connected to exactly one string.
Each string is connected to at least one and at most two marked boundary points.
Strings do not intersect. 
No string may connect two inner points or two outer points.
Hence there are three possibilities for strings:
\be
\item
A \underline{through string} connects an inner and an outer boundary point.
\item
A \underline{cap} is a string that only connects to an inner boundary point.
\item
A \underline{cup} is a string that only connects to an outer boundary point.
\ee
We draw a dark circle on the end of a non-through string to denote that that end does not attach to another boundary point.
\item There is a distinguished interval on each boundary disk, marked by a $\star$.
\end{itemize}
\itt{Composition} Composition is the $\C$-linear extension of insertion of annuli, making sure the boundary points line up, as do the distinguished intervals.
When we get a floating string (a string connected to no boundary points), we just remove it.
$$
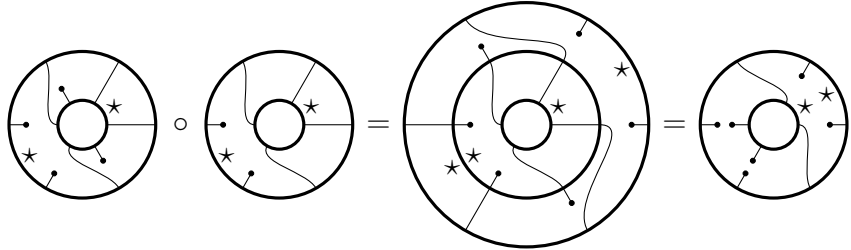
\begin{tikzpicture}[annular]
	\clip (0,0) circle (1.6cm);
	\draw (0:0.5cm) -- (0:1.5cm);
	\draw (60:0.5cm) -- (60:1.5cm);
	\draw (180:0.5cm) .. controls ++(180:.4cm) and ++(-60:.4cm) .. (120:1.5cm);
	\draw (120:0.5cm) -- (120:.85cm);
	\filldraw (120:.85cm) circle (.05cm);
	\draw (180:1.15cm) -- (180:1.5cm);
	\filldraw (180:1.15cm) circle (.05cm);
	\draw (240:0.5cm) .. controls ++(240:.4cm) and ++(120:.4cm) .. (300:1.5cm);
	\draw (300:0.5cm) -- (300:.85cm);
	\filldraw (300:.85cm) circle (.05cm);
	\draw (240:1.15cm) -- (240:1.5cm);
	\filldraw (240:1.15cm) circle (.05cm);
	\node at (30:.75cm) {$\star$};
	\node at (210:1.25cm) {$\star$};
	\draw [very thick] (0,0) circle (.5cm);
	\draw [very thick] (0,0) circle (1.5cm);
\end{tikzpicture}
\,\circ\,
\begin{tikzpicture}[annular]
	\clip (0,0) circle (1.6cm);
	\draw (0:0.5cm) -- (0:1.5cm);
	\draw (60:0.5cm) -- (60:1.5cm);
	\draw (180:0.5cm) .. controls ++(180:.4cm) and ++(-60:.4cm) .. (120:1.5cm);
	\draw (180:1.15cm) -- (180:1.5cm);
	\filldraw (180:1.15cm) circle (.05cm);
	\draw (240:0.5cm) .. controls ++(240:.4cm) and ++(120:.4cm) .. (300:1.5cm);
	\draw (240:1.15cm) -- (240:1.5cm);
	\filldraw (240:1.15cm) circle (.05cm);
	\node at (30:.75cm) {$\star$};
	\node at (210:1.25cm) {$\star$};
	\draw [very thick] (0,0) circle (.5cm);
	\draw [very thick] (0,0) circle (1.5cm);
\end{tikzpicture}
=
\begin{tikzpicture}[annular]
	\clip (0,0) circle (2.6cm);
	\draw (0:0.5cm) -- (0:1.5cm) .. controls ++(0:.8cm) and ++(120:.8cm) .. (300:2.5cm);
	\draw (0:2.15cm) -- (0:2.5cm);
	\filldraw (0:2.15cm) circle (.05cm);	
	\draw (60:0.5cm) -- (60:1.5cm) .. controls ++(60:.8cm) and ++(-60:.8cm) .. (120:2.5cm);
	\draw (60:2.15cm) -- (60:2.5cm);
	\filldraw (60:2.15cm) circle (.05cm);	
	\draw (180:0.5cm) .. controls ++(180:.4cm) and ++(-60:.4cm) .. (120:1.5cm) -- (120:1.85cm);
	\filldraw (120:1.85cm) circle (.05cm);
	\draw (180:1.15cm) -- (180:2.5cm);
	\filldraw (180:1.15cm) circle (.05cm);
	\draw (240:0.5cm) .. controls ++(240:.4cm) and ++(120:.4cm) .. (300:1.5cm) -- (300:1.85cm);
	\filldraw (300:1.85cm) circle (.05cm);
	\draw (240:1.15cm) -- (240:2.5cm);
	\filldraw (240:1.15cm) circle (.05cm);
	\node at (30:.75cm) {$\star$};
	\node at (210:1.25cm) {$\star$};
	\node at (210:1.75cm) {$\star$};
	\node at (30:2.25cm) {$\star$};
	\draw [very thick] (0,0) circle (.5cm);
	\draw [very thick] (0,0) circle (1.5cm);
	\draw [very thick] (0,0) circle (2.5cm);
\end{tikzpicture}
=
\begin{tikzpicture}[annular]
	\clip (0,0) circle (1.6cm);
	\draw (0:1.15cm) -- (0:1.5cm);
	\filldraw (0:1.15cm) circle (.05cm);
	\draw (60:1.15cm) -- (60:1.5cm);
	\filldraw (60:1.15cm) circle (.05cm);
	\draw (60:0.5cm) .. controls ++(60:.4cm) and ++(-60:.4cm) .. (120:1.5cm);
	\draw (180:0.5cm) -- (180:.85cm);
	\filldraw (180:.85cm) circle (.05cm);
	\draw (180:1.15cm) -- (180:1.5cm);
	\filldraw (180:1.15cm) circle (.05cm);
	\draw (0:0.5cm) .. controls ++(0:.4cm) and ++(120:.4cm) .. (300:1.5cm);
	\draw (240:0.5cm) -- (240:.85cm);
	\filldraw (240:.85cm) circle (.05cm);
	\draw (240:1.15cm) -- (240:1.5cm);
	\filldraw (240:1.15cm) circle (.05cm);
	\node at (30:.75cm) {$\star$};
	\node at (30:1.25cm) {$\star$};
	\draw [very thick] (0,0) circle (.5cm);
	\draw [very thick] (0,0) circle (1.5cm);
\end{tikzpicture}
$$
\itt{Adjoint} The adjoint is the conjugate-linear extension of flipping the tangle inside out.
$$
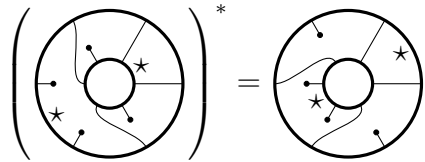
\left(
\begin{tikzpicture}[annular]
	\clip (0,0) circle (1.6cm);
	\draw (0:0.5cm) -- (0:1.5cm);
	\draw (60:0.5cm) -- (60:1.5cm);
	\draw (180:0.5cm) .. controls ++(180:.4cm) and ++(-60:.4cm) .. (120:1.5cm);
	\draw (120:0.5cm) -- (120:.85cm);
	\filldraw (120:.85cm) circle (.05cm);
	\draw (180:1.15cm) -- (180:1.5cm);
	\filldraw (180:1.15cm) circle (.05cm);
	\draw (240:0.5cm) .. controls ++(240:.4cm) and ++(120:.4cm) .. (300:1.5cm);
	\draw (300:0.5cm) -- (300:.85cm);
	\filldraw (300:.85cm) circle (.05cm);
	\draw (240:1.15cm) -- (240:1.5cm);
	\filldraw (240:1.15cm) circle (.05cm);
	\node at (30:.75cm) {$\star$};
	\node at (210:1.25cm) {$\star$};
	\draw [very thick] (0,0) circle (.5cm);
	\draw [very thick] (0,0) circle (1.5cm);
\end{tikzpicture}
\right)^*
=
\begin{tikzpicture}[annular]
	\clip (0,0) circle (1.6cm);
	\draw (0:0.5cm) -- (0:1.5cm);
	\draw (60:0.5cm) -- (60:1.5cm);
	\draw (120:0.5cm) .. controls ++(120:.4cm) and ++(0:.4cm) .. (180:1.5cm);
	\draw (120:1.15cm) -- (120:1.5cm);
	\filldraw (120:1.15cm) circle (.05cm);
	\draw (180:.5cm) -- (180:.85cm);
	\filldraw (180:.85cm) circle (.05cm);
	\draw (300:0.5cm) .. controls ++(300:.4cm) and ++(60:.4cm) .. (240:1.5cm);
	\draw (300:1.1cm) -- (300:1.5cm);
	\filldraw (300:1.15cm) circle (.05cm);
	\draw (240:.5cm) -- (240:.85cm);
	\filldraw (240:.85cm) circle (.05cm);
	\node at (30:1.25cm) {$\star$};
	\node at (210:0.75cm) {$\star$};
	\draw [very thick] (0,0) circle (.5cm);
	\draw [very thick] (0,0) circle (1.5cm);
\end{tikzpicture}
$$
\end{defn}

\begin{rem}
Unlike the annular Temperley-Lieb category, $\sA\sP_n$ is finite dimensional for all $n\geq 0$ due to the absence of non-contractible closed loops.
\end{rem}

We now count the number of annular tangles in $\sA\sP(m,n)$.

\begin{defn}
Let $N(m,n;k)$ be the number of annular tangles in $\sA\sP_n$ with $m$ inner points, $n$ outer points, and $k$ through strings.
Let $N(m,n)=\sum_{k=0}^{\min\{m,n\}} N(m,n;k)$.
\end{defn}

\begin{rem}\label{rem:CountTangles}
Note that
\begin{itemize}
\item
$N(m,n;k)=N(n,m;k)$ for all $m,n,k$, so we only need to count when $k\leq m\leq n$,
\item
$N(0,n)=1$ for al $n\geq 0$, and
\item
$N(m,n;0)=1$ for all $m,n\geq 0$.
\end{itemize}
\end{rem}

\begin{lem}\label{lem:CountTangles}
If $1\leq k\leq m\leq n$, then $\D N(m,n;k)=m{n\choose k}{m-1\choose k-1}$.
\end{lem}
\begin{proof}
Draw an annulus with $m$ inner points and $n$ outer points.
Fix the outer $\star$. 
There are exactly $n\choose k$ ways to connect $k$ through strings to the $n$ outer points.
Equivalently, there are exactly $n\choose k$ choices for the cup positions.

Let us examine one of these choices more closely. 
Look at the first through string connected to an outer point counting clockwise from the outer $\star$.
Follow the through string inward, and put the inner star on the interval to the left of this inner point, so that the region meeting the outer $\star$ meets the inner $\star$.
We now see there are exactly $m-1 \choose k-1$ ways to connect the remaining through strings to the remaining inner points.
Equivalently, there are exactly $m-1\choose k-1$ choices of the cap positions.

Fix such a choice of cap position, which we will call the tangle's initial cap position.
Note that given an annular tangle in $\sA\sP(m,n)$, the cup positions and the initial cap positions only depend on the outer $\star$.
Hence we get $m$ distinct tangles as we shift the inner $\star$ clockwise.

In summary, for each of the $n\choose k$ cup positions and for each of the resulting $m-1 \choose k-1$ initial cap positions, there are $m$ distinct tangles.
The formula follows.
\end{proof}

Remark \ref{rem:CountTangles} and Lemma \ref{lem:CountTangles} now prove the following.

\begin{prop}\label{prop:CountTangles}
If $1\leq m\leq n$, then 
$\D
N(m,n)=1+\sum_{k=1}^m m{n\choose k}{m-1\choose k-1}.
$
\end{prop}

We will determine the algebra structure of $\sA\sP_n$ at the end of this subsection in Proposition \ref{prop:AnnularAlgebraStructure}. 
We first treat the rectangular planar rook category as a warmup.

\begin{defn}
The \underline{rectangular planar rook category} $\sR\sP$ is the subcategory of $\sA\sP$ such that $\sR\sP(m,n)$ is the $\C$-linear combinations of diagrams in $\sA\sP(m,n)$ such that the region meeting the internal $\star$ also meets the external $\star$.
For example,
$$
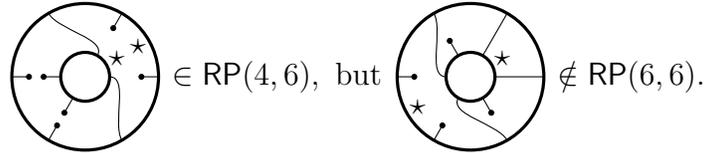
\begin{tikzpicture}[annular]
	\clip (0,0) circle (1.6cm);
	\draw (0:1.15cm) -- (0:1.5cm);
	\filldraw (0:1.15cm) circle (.05cm);
	\draw (60:1.15cm) -- (60:1.5cm);
	\filldraw (60:1.15cm) circle (.05cm);
	\draw (60:0.5cm) .. controls ++(60:.4cm) and ++(-60:.4cm) .. (120:1.5cm);
	\draw (180:0.5cm) -- (180:.85cm);
	\filldraw (180:.85cm) circle (.05cm);
	\draw (180:1.15cm) -- (180:1.5cm);
	\filldraw (180:1.15cm) circle (.05cm);
	\draw (0:0.5cm) .. controls ++(0:.4cm) and ++(120:.4cm) .. (300:1.5cm);
	\draw (240:0.5cm) -- (240:.85cm);
	\filldraw (240:.85cm) circle (.05cm);
	\draw (240:1.15cm) -- (240:1.5cm);
	\filldraw (240:1.15cm) circle (.05cm);
	\node at (30:.75cm) {$\star$};
	\node at (30:1.25cm) {$\star$};
	\draw [very thick] (0,0) circle (.5cm);
	\draw [very thick] (0,0) circle (1.5cm);
\end{tikzpicture}
\in\sR\sP(4,6),
\text{ but }
\begin{tikzpicture}[annular]
	\clip (0,0) circle (1.6cm);
	\draw (0:0.5cm) -- (0:1.5cm);
	\draw (60:0.5cm) -- (60:1.5cm);
	\draw (180:0.5cm) .. controls ++(180:.4cm) and ++(-60:.4cm) .. (120:1.5cm);
	\draw (120:0.5cm) -- (120:.85cm);
	\filldraw (120:.85cm) circle (.05cm);
	\draw (180:1.15cm) -- (180:1.5cm);
	\filldraw (180:1.15cm) circle (.05cm);
	\draw (240:0.5cm) .. controls ++(240:.4cm) and ++(120:.4cm) .. (300:1.5cm);
	\draw (300:0.5cm) -- (300:.85cm);
	\filldraw (300:.85cm) circle (.05cm);
	\draw (240:1.15cm) -- (240:1.5cm);
	\filldraw (240:1.15cm) circle (.05cm);
	\node at (30:.75cm) {$\star$};
	\node at (210:1.25cm) {$\star$};
	\draw [very thick] (0,0) circle (.5cm);
	\draw [very thick] (0,0) circle (1.5cm);
\end{tikzpicture}
\notin\sR\sP(6,6).
$$
Each such morphism can be represented by a rectangular tangle rather than an annular tangle as follows.
First, cut along a path from the internal $\star$ to the external $\star$ which does not meet any strings.
Second, isotope the resulting diagram into a rectangle with lower and upper boundary points so that the inner boundary points of the annulus are now the lower boundary points of the rectangle, and the outer boundary points of the annulus are now the upper boundary points of the rectangle.
%$$
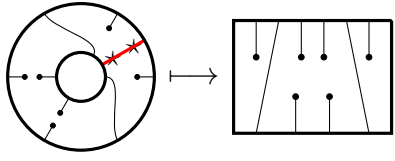
\begin{figure}[!ht]
$$
\begin{tikzpicture}[annular]
	\clip (0,0) circle (1.6cm);
	\draw (0:1.15cm) -- (0:1.5cm);
	\filldraw (0:1.15cm) circle (.05cm);
	\draw (60:1.15cm) -- (60:1.5cm);
	\filldraw (60:1.15cm) circle (.05cm);
	\draw (60:0.5cm) .. controls ++(60:.4cm) and ++(-60:.4cm) .. (120:1.5cm);
	\draw (180:0.5cm) -- (180:.85cm);
	\filldraw (180:.85cm) circle (.05cm);
	\draw (180:1.15cm) -- (180:1.5cm);
	\filldraw (180:1.15cm) circle (.05cm);
	\draw (0:0.5cm) .. controls ++(0:.4cm) and ++(120:.4cm) .. (300:1.5cm);
	\draw (240:0.5cm) -- (240:.85cm);
	\filldraw (240:.85cm) circle (.05cm);
	\draw (240:1.15cm) -- (240:1.5cm);
	\filldraw (240:1.15cm) circle (.05cm);
	\node at (30:.75cm) {$\star$};
	\node at (30:1.25cm) {$\star$};
	\draw[very thick, red] (30:.5cm) -- (30:1.5cm);
	\draw [very thick] (0,0) circle (.5cm);
	\draw [very thick] (0,0) circle (1.5cm);
\end{tikzpicture}
\longmapsto
\begin{tikzpicture}[rectangular]
	\clip (1.9,1.1) --(-1.1,1.1) -- (-1.1,-1.1) -- (1.9,-1.1);
	\draw (-.6,1)--(-.6,.35);
	\filldraw (-.6,.35) circle (.05cm);
	\draw (.2,1)--(.2,.35);
	\filldraw (.2,.35) circle (.05cm);
	\draw (.6,1)--(.6,.35);
	\filldraw (.6,.35) circle (.05cm);
	\draw (1.4,1)--(1.4,.35);
	\filldraw (1.4,.35) circle (.05cm);
	\draw (-.2,1)--(-.6,-1);
	\draw (1,1)--(1.4,-1);
	\draw (.1,-1)--(.1,-.35);
	\filldraw (.1,-.35) circle (.05cm);	
	\draw (.7,-1)--(.7,-.35);
	\filldraw (.7,-.35) circle (.05cm);	
	\draw [very thick] (1.8,1) --(-1,1) -- (-1,-1) -- (1.8,-1)--(1.8,1);
\end{tikzpicture}
$$
\caption{Cutting an annulus to get a rectangle}
\label{fig:CuttingAndGluing}
\end{figure}
%$$

Composition of annuli then corresponds to stacking rectangles,
$$
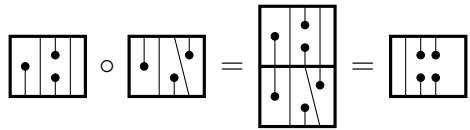
\begin{tikzpicture}[baseline=-.1cm]
	\nbox{}{(0,0)}{.4}{.2}{0}{}
	\BrokenBottomString{(-.4,0)}
	\String{(-.2,0)}
	\BrokenString{(0,0)}
	\String{(.2,0)}
\end{tikzpicture}
\,\circ\,
\begin{tikzpicture}[baseline=-.1cm]
	\nbox{}{(0,0)}{.4}{.2}{0}{}
	\BrokenTopString{(-.4,0)}
	\String{(-.2,0)}
	\PartialIsoStar{(0,0)}
\end{tikzpicture}
\,=\,
\begin{tikzpicture}[baseline=-.1cm]
	\nbox{}{(0,0)}{.8}{-.2}{-.4}{}
	\draw[very thick] (-.6,0) -- (.4,0);
	\BrokenBottomString{(-.4,.4)}
	\String{(-.2,.4)}
	\BrokenString{(0,.4)}
	\String{(.2,.4)}
	\BrokenTopString{(-.4,-.4)}
	\String{(-.2,-.4)}
	\PartialIsoStar{(0,-.4)}
\end{tikzpicture}
\,=\,
\begin{tikzpicture}[baseline=-.1cm]
	\nbox{}{(0,0)}{.4}{.2}{0}{}
	\String{(-.4,0)}
	\BrokenString{(-.2,0)}
	\BrokenString{(0,0)}
\end{tikzpicture}
$$
and the adjoint operation corresponds to vertical flipping of rectangles.
$$
\begin{tikzpicture}[baseline=-.1cm]
	\nbox{}{(0,0)}{.4}{.2}{0}{}
	\BrokenTopString{(-.4,0)}
	\String{(-.2,0)}
	\PartialIsoStar{(0,0)}
\end{tikzpicture}^{\,*}
\,=\,
\begin{tikzpicture}[baseline=-.1cm]
	\nbox{}{(0,0)}{.4}{.2}{0}{}
	\BrokenBottomString{(-.4,0)}
	\String{(-.2,0)}
	\PartialIso{(0,0)}
\end{tikzpicture}
$$
Viewing morphisms in $\sR\sP$ as rectangular tangles, we can endow $\sR\sP$ with a tensor structure.
The tensor product of objects is $[m]\otimes [n]=[m+n]$, and the tensor product of morphisms is the $\C$-linear extension of horizontal join.
$$
\begin{tikzpicture}[baseline=-.1cm]
	\nbox{}{(0,0)}{.4}{.2}{0}{}
	\BrokenBottomString{(-.4,0)}
	\String{(-.2,0)}
	\BrokenString{(0,0)}
	\String{(.2,0)}
\end{tikzpicture}
\,\otimes\,
\begin{tikzpicture}[baseline=-.1cm]
	\nbox{}{(0,0)}{.4}{.2}{0}{}
	\BrokenTopString{(-.4,0)}
	\String{(-.2,0)}
	\PartialIsoStar{(0,0)}
\end{tikzpicture}
\,=\,
\begin{tikzpicture}[baseline=-.1cm]
	\nbox{}{(0,0)}{.4}{.2}{.8}{}
	\BrokenBottomString{(-.4,0)}
	\String{(-.2,0)}
	\BrokenString{(0,0)}
	\String{(.2,0)}
	\BrokenTopString{(.4,0)}
	\String{(.6,0)}
	\PartialIsoStar{(.8,0)}
\end{tikzpicture}
$$
\end{defn}

\begin{rem}
Obviously $\sR\sP_n\cong \C P_n$ by contracting cups and caps and trading the external boundary for nodes at the marked boundary points.
$$
\begin{tikzpicture}[baseline=-.1cm]
	\nbox{}{(0,0)}{.4}{.2}{0}{}
	\BrokenTopString{(-.4,0)}
	\String{(-.2,0)}
	\PartialIsoStar{(0,0)}
\end{tikzpicture}
\longleftrightarrow
\left(
\begin{tikzpicture}[baseline=.2cm, scale=.6]
	\clip (-.2,-.2) -- (3.2,-.2) -- (3.2,1.2) -- (-.2,1.2);
	\draw (1,0) -- (1,1);
	\draw (2,1) -- (3,0);
%	\filldraw (0,0) circle (.1cm);
	\filldraw (1,0) circle (.1cm);
	\filldraw (2,0) circle (.1cm);
	\filldraw (3,0) circle (.1cm);
	\filldraw (0,1) circle (.1cm);
	\filldraw (1,1) circle (.1cm);
	\filldraw (2,1) circle (.1cm);
	\filldraw (3,1) circle (.1cm);
\end{tikzpicture}
\right)
$$
We use different diagrams for morphisms in $\sR\sP_n$ than the usual diagrams for $P_n$ to utilize a notational trick of Bigelow (see Definition \ref{defn:DottedStrand}).
\end{rem}

\begin{prop}\label{prop:Pascal}
As a complex $*$-algebra, $\sR\sP_n\cong \bigoplus_{k=0}^n M_{n\choose k}(\C)$. Moreover, the Bratteli diagram for the tower of finite dimensional algebras $(\sR\sP_n)_{n\geq 0}$ under the right inclusion (adding a through string to the right) is given by Pascal's Triangle.
\end{prop}
\begin{proof}
Let $0\leq k\leq n$. There are exactly $n\choose k$ diagrams with exactly $k$ through strings in $\sR\sP_n$. 
Hence $\dim_\C(\sR\sP_n)=\sum_{k=0}^n {n \choose k}=2^n$. 
However, it is important to note that diagrams with exactly $k$ through strings are not orthogonal to diagrams with exactly $j$ through strings for $j\neq k$.
To fix this problem, we make the following definition.
\begin{defn}\label{defn:DottedStrand}
As in \cite[Section 3]{MR2925434}, we let the dotted strand denote the following morphism in $\sR\sP_1$: 
$$
\eOnePerp=\identity-\eOne\,.
$$
We then have the following relations in $\sR\sP$:
\begin{align}
\confetti &= 1
\tag{$\sR\sP$1}\label{rel:RP1}\\
\doubleDotStrand &= \eOnePerp
\tag{$\sR\sP$2}\label{rel:RP2}\\
\doubleDotEnd &= 0.
\tag{$\sR\sP$3}\label{rel:RP3}
\end{align}
\end{defn}

\begin{rem}
Under the identification of these diagrams with those in the Temperley-Lieb category with only shaded cups and caps in Figure \ref{fig:ShadedTLDiagrams}, the broken strand corresponds to the Jones projection $e_1$, and the dotted strand corresponds to the Jones-Wenzl projection $\jw{2}=1-e_1$.
\end{rem}

With the use of the dotted strand, we find $2^n$ minimal orthogonal projections in $\sR\sP_n$ given by the simple tensors composed entirely of
$$
\eOne\, \text{ and }\,\eOnePerp\,.
$$
The diagrams with exactly $k$ through strings, all of which are dotted, span a full matrix algebra $M_{n\choose k}(\C)$.
Hence $\sR\sP_n$ is isomorphic to the orthogonal direct sum $ \bigoplus_{k=0}^n M_{n\choose k} (\C)$.

We now look at the right inclusion $\sR\sP_n\hookrightarrow \sR\sP_{n+1}$ given by adding a string to the right. Since 
$$
\identity = \eOne + \eOnePerp\,,
$$
we see that the right inclusion maps each minimal projection in $\sR\sP_n$ to the sum of exactly two minimal projections in $\sR\sP_{n+1}$. 
More precisely, each minimal projection in the simple summand corresponding to $M_{n\choose k}(\C)$ maps to the sum of two minimal projections, one in $M_{{n+1}\choose k}(\C)$, and one in $M_{{n+1} \choose {k+1}}(\C)$. Hence the Bratteli diagram is as claimed.
\end{proof}

\begin{rem}
We give an explicit formula for the resulting isomorphism of towers $(G_n)_{n\geq 0}\cong (\sR\sP_n)_{n\geq 0}$ in Theorem \ref{thm:IsoOfTowers}.
\end{rem}

\begin{cor}
$\D\dim(\sR\sP_n)=\sum_{k=0}^n {n\choose k}^2$.
\end{cor}

We now determine the algebra structure of $\sA\sP_n$.
The dotted strand will be of great use to us.

\begin{prop}\label{prop:AnnularAlgebraStructure}
As a complex $*$-algebra,
$\D\sA\sP_n \cong \C\oplus \bigoplus_{k=1}^n k M_{n\choose k}(\C)$.
\end{prop}

\begin{rem}\label{rem:BinomialIdentity}
Note that we have the identity
$$
n {n\choose k}{n-1\choose k-1}=k{n\choose k}^2,
$$
so the formula in Proposition \ref{prop:AnnularAlgebraStructure} is consistent with Proposition \ref{prop:CountTangles}.
\end{rem}

\begin{proof}[Proof of Proposition \ref{prop:AnnularAlgebraStructure}]
Using Bigelow's dotted strand, consider the annular tangles which give minimal projections in $\sR\sP_n$ under the cutting operation in Figure \ref{fig:CuttingAndGluing}.
These annular tangles are orthogonal projections in $\sA\sP_n$, but the only one that remains minimal is the one with only broken strings and no dotted through strings.
Now given a projection $p_k$ with $k$ dotted through strings, the $k$ powers of the 1-click rotation tangle $\rho$ (see Figure \ref{fig:rotation}) can be compressed by $p_k$ to give $k$ distinct tangles $p_k \rho^i p_k$ for $i=1\dots, k$. 
\begin{figure}[!ht]
$$
\rho = 
\begin{tikzpicture}[annular]
	\clip (0,0) circle (1.6cm);
	\draw (216:0.5cm) .. controls ++(216:.6cm) and ++(324:.4cm) .. (144:1.5cm);
	\draw (144:0.5cm) .. controls ++(144:.6cm) and ++(252:.4cm) .. (72:1.5cm);
	\draw (72:0.5cm) .. controls ++(72:.6cm) and ++(180:.4cm) .. (0:1.5cm);
	\draw (0:0.5cm) .. controls ++(0:.6cm) and ++(108:.4cm) .. (-72:1.5cm);
	\draw (-72:0.5cm) .. controls ++(-72:.6cm) and ++(36:.4cm) .. (-144:1.5cm);
	\node at (158:.75cm) {$\star$};
	\node at (180:1.25cm) {$\star$};
	\draw [very thick] (0,0) circle (.5cm);
	\draw [very thick] (0,0) circle (1.5cm);
\end{tikzpicture}
$$
\caption{The one click rotation $\rho\in \sA\sP_5$}
\label{fig:rotation}
\end{figure}

Now if $\omega_k$ is a $k$-th root of unity, we get a projection 
$$
p_k^\omega=\frac{1}{k}\sum_{i=0}^{k-1} \omega_k^{-i} p_k \rho^i p_k
$$ 
which lives under $p_k$. Distinct $\omega$ give distinct projections, since $\rho(p_k^\omega)=\omega p_k^\omega$, so $p_k$ splits into $k$ non-zero orthogonal projections.
Now using the usual partial isometries from $\sR\sP_n$ in annular form, we see that splitting each projection with $k\geq 2$ dotted through strings into $k$ orthogonal summands also splits the corresponding copy of $M_{n\choose k}(\C)$ in $\sR\sP_n$ into $k$ copies of $M_{n\choose k}(\C)$ in $\sA\sP_n$, which results in the claimed decomposition.
By Remark \ref{rem:BinomialIdentity}, we must have all the minimal projections, since the dimension count agrees with Proposition \ref{prop:CountTangles}.
\end{proof}

%%%%%%%%%%%%%%%%%%%%%%%%%%%%%%%%%%%%%%%%%%%%%%%%%%
\subsection{Annular and rectangular GICAR categories}

\begin{defn}
The \underline{annular GICAR category} $\sA\sG$ is the following small involutive category.
\itt{Objects} symbols $[n]$ for $n\geq 0$.
\itt{Morphisms} 
The morphisms of $\sA\sG$ are $\C$-linear combinations of the words $*$-generated by the maps
\begin{align*}
\alpha_i : [n] &\longrightarrow [n+1] \text{ for }i=1,\dots, n+1\text{ and }n\geq 0\\
\alpha_i^* : [n] &\longrightarrow [n-1] \text{ for }i=1,\dots, n\text{ and }n\geq 1\\
\tau :[n] & \longrightarrow [n] \text{ for }n\geq 0
\end{align*}
subject to the relations
\begin{align}
\alpha_i\alpha_{j-1}&=\alpha_j\alpha_i
\text{ and }
\alpha_i^*\alpha_j^* = \alpha_{j-1}^*\alpha_i^* 
\text{ for all }i<j
\label{rel:AG1}
\tag{$\sA\sG1$}
\\
\alpha_{i}^*\alpha_{j}
&=
\begin{cases}
\alpha_{j+1}\alpha_{i}^* &\text{if }i<j\\
\id_{[n]} & \text{if } i=j\\
\alpha_j \alpha_{i+1}^* & \text{if } i>j
\end{cases}
\label{rel:AG2}
\tag{$\sA\sG2$}
\\
\tau^*&=\tau^{-1}\text{ and }\tau^n=\id_{[n]}
\label{rel:AG3}
\tag{$\sA\sG3$}
\\
\alpha_i \tau &= \tau \alpha_{i-1}\text{ and }\alpha_i^*\tau = \tau\alpha_{i-1}^* \text{ for all }i=2,\dots, n.
\label{rel:AG4}
\tag{$\sA\sG4$}
\end{align}
\itt{Composition} The composition in $\sA\sG$ is the concatenation of words. 
\itt{Adjoint} The adjoint of a word $w=\ell_1\dots \ell_n$ where the letters $\ell_k\in\{\alpha_i,\alpha_j^*,\tau \}$ is given by $w^*=\ell_n^*\cdots \ell_1^*$.
\end{defn}

\begin{rem}
$\sA\sG$ is the full subcategory of ${\sf a\Delta}$ in \cite{MR2903179} generated by $\tau,\alpha_i,\alpha_j^*$, after replacing the $\alpha_i$'s by the $\beta_{2i}$'s appearing there.
\end{rem}

\begin{prop}\label{prop:ExtraRelation}
The additional relations 
\begin{equation}
\alpha_1=\tau\alpha_n
 \text{ and }
\alpha_1^*\tau = \alpha_n^*
\label{rel:AG5}
\tag{$\sA\sG5$}
\end{equation}
hold in $\sA\sG$.
\end{prop}
\begin{proof}
By Relations \eqref{rel:AG3} and \eqref{rel:AG4}, we have
$$
\alpha_1 = \tau^{n+1}\alpha_1 = \tau^n \alpha_2\tau=\cdots=\tau\alpha_n\tau^n=\tau\alpha_n. 
$$
Now take adjoints. (Note there is a typo in the proof of this relation in \cite[Proposition 3.6.(1)]{MR2903179}).
\end{proof}

\begin{rem}
By the results of \cite{MR2903179}, there is a $*$-equivalence of categories $\sA\sP\cong \sA\sG$. 
We provide a short proof of this fact for the convenience of the reader along the same line of reasoning as \cite{MR2903179}.
\end{rem}

\begin{prop}\label{prop:Standard}
Suppose $w\in \sA\sG(m,n)$ is a word in the $\alpha_i,\tau,\alpha_j^*$. 
Then $w$ can be written uniquely in the standard form
$$
w=\alpha_{i_k}\cdots \alpha_{i_1} \tau^r \alpha_{j_1}^*\cdots \alpha_{j_\ell}^*
$$
where $i_1< \cdots <i_k$, $0\leq r < m-k$, and $j_1<\cdots < j_\ell$. 

In particular, the words in standard form give bases for $\sG(m,n)$, and thus $\dim_\C(\sA\sG(m,n))<\I$ for all $m,n$.
\end{prop}
\begin{proof}
Using Relations \eqref{rel:AG1}-\eqref{rel:AG5}, first push all the $\alpha_i$'s all the way to the left and all the $\alpha_j^*$'s all the way the right, leaving the $\tau$'s in the middle. Then use Relation \eqref{rel:AG1} to put the $\alpha_i$'s in decreasing order and the $\alpha_j^*$'s in increasing order. Finally, use Relation \eqref{rel:AG3} to reduce the number of $\tau$'s in the middle.
\end{proof}

\begin{thm}\label{thm:Equivalence}
There is a $*$-equivalence of categories $\sA\sP\cong \sA\sG$.
\end{thm}
\begin{proof}
We construct a $*$-functor $\Psi:\sA\sG\to \sA\sP$.
First, define $\Psi([n])=[n]$.
Next, we define $\Psi$ on the morphisms $\alpha_i,\tau,\alpha_j^*$.
\begin{itemize}
\item
$\alpha_i\in \sA\sG(n,n+1)$ maps to the tangle in $\sR\sP$ with $n$ inner points, $n+1$ outer points, a cap attached to outer boundary point $i$, and all other boundary points are connected by through strings so that the region meeting the internal $\star$ also meets the external $\star$.
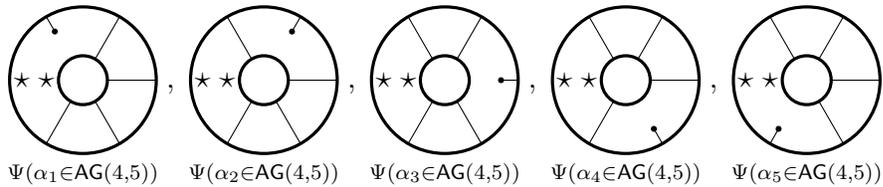
$$
\underset{\Psi(\alpha_1\in \sA\sG(4,5))}{
\begin{tikzpicture}[annular]
	\clip (0,0) circle (1.6cm);
	\draw (120:1.15cm) -- (120:1.5cm);
	\filldraw (120:1.15cm) circle (.05cm);
	\draw (60:0.5cm) -- (60:1.5cm);
	\draw (0:0.5cm) -- (0:1.5cm);
	\draw (-60:0.5cm) -- (-60:1.5cm);
	\draw (-120:0.5cm) -- (-120:1.5cm);
	\node at (180:.75cm) {$\star$};
	\node at (180:1.25cm) {$\star$};
	\draw [very thick] (0,0) circle (.5cm);
	\draw [very thick] (0,0) circle (1.5cm);
\end{tikzpicture}
}
\,,\,
\underset{\Psi(\alpha_2\in \sA\sG(4,5))}{
\begin{tikzpicture}[annular]
	\clip (0,0) circle (1.6cm);
	\draw (120:0.5cm) -- (120:1.5cm);
	\draw (60:1.15cm) -- (60:1.5cm);
	\filldraw (60:1.15cm) circle (.05cm);
	\draw (0:0.5cm) -- (0:1.5cm);
	\draw (-60:0.5cm) -- (-60:1.5cm);
	\draw (-120:0.5cm) -- (-120:1.5cm);
	\node at (180:.75cm) {$\star$};
	\node at (180:1.25cm) {$\star$};
	\draw [very thick] (0,0) circle (.5cm);
	\draw [very thick] (0,0) circle (1.5cm);
\end{tikzpicture}
}
\,,\,
\underset{\Psi(\alpha_3\in \sA\sG(4,5))}{
\begin{tikzpicture}[annular]
	\clip (0,0) circle (1.6cm);
	\draw (120:0.5cm) -- (120:1.5cm);
	\draw (60:0.5cm) -- (60:1.5cm);
	\draw (0:1.15cm) -- (0:1.5cm);
	\filldraw (0:1.15cm) circle (.05cm);
	\draw (-60:0.5cm) -- (-60:1.5cm);
	\draw (-120:0.5cm) -- (-120:1.5cm);
	\node at (180:.75cm) {$\star$};
	\node at (180:1.25cm) {$\star$};
	\draw [very thick] (0,0) circle (.5cm);
	\draw [very thick] (0,0) circle (1.5cm);
\end{tikzpicture}
}
\,,\,
\underset{\Psi(\alpha_4\in \sA\sG(4,5))}{
\begin{tikzpicture}[annular]
	\clip (0,0) circle (1.6cm);
	\draw (120:0.5cm) -- (120:1.5cm);
	\draw (60:0.5cm) -- (60:1.5cm);
	\draw (0:0.5cm) -- (0:1.5cm);
	\draw (-60:1.15cm) -- (-60:1.5cm);
	\filldraw (-60:1.15cm) circle (.05cm);
	\draw (-120:0.5cm) -- (-120:1.5cm);
	\node at (180:.75cm) {$\star$};
	\node at (180:1.25cm) {$\star$};
	\draw [very thick] (0,0) circle (.5cm);
	\draw [very thick] (0,0) circle (1.5cm);
\end{tikzpicture}
}
\,,\,
\underset{\Psi(\alpha_5\in \sA\sG(4,5))}{
\begin{tikzpicture}[annular]
	\clip (0,0) circle (1.6cm);
	\draw (120:0.5cm) -- (120:1.5cm);
	\draw (60:0.5cm) -- (60:1.5cm);
	\draw (0:0.5cm) -- (0:1.5cm);
	\draw (-60:0.5cm) -- (-60:1.5cm);
	\draw (-120:1.15cm) -- (-120:1.5cm);
	\filldraw (-120:1.15cm) circle (.05cm);
	\node at (180:.75cm) {$\star$};
	\node at (180:1.25cm) {$\star$};
	\draw [very thick] (0,0) circle (.5cm);
	\draw [very thick] (0,0) circle (1.5cm);
\end{tikzpicture}
}
$$
\item
$\alpha_j^*\in\sA\sG(n,n-1)$ maps to $\Psi(\alpha_j) ^*\in\sA\sP(n-1,n)$.
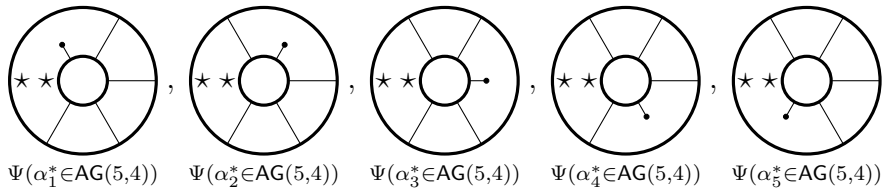
$$
\underset{\Psi(\alpha_1^*\in \sA\sG(5,4))}{
\begin{tikzpicture}[annular]
	\clip (0,0) circle (1.6cm);
	\draw (120:0.5cm) -- (120:.85cm);
	\filldraw (120:.85cm) circle (.05cm);
	\draw (60:0.5cm) -- (60:1.5cm);
	\draw (0:0.5cm) -- (0:1.5cm);
	\draw (-60:0.5cm) -- (-60:1.5cm);
	\draw (-120:0.5cm) -- (-120:1.5cm);
	\node at (180:.75cm) {$\star$};
	\node at (180:1.25cm) {$\star$};
	\draw [very thick] (0,0) circle (.5cm);
	\draw [very thick] (0,0) circle (1.5cm);
\end{tikzpicture}
}
\,,\,
\underset{\Psi(\alpha_2^*\in \sA\sG(5,4))}{
\begin{tikzpicture}[annular]
	\clip (0,0) circle (1.6cm);
	\draw (120:0.5cm) -- (120:1.5cm);
	\draw (60:0.5cm) -- (60:.85cm);
	\filldraw (60:.85cm) circle (.05cm);
	\draw (0:0.5cm) -- (0:1.5cm);
	\draw (-60:0.5cm) -- (-60:1.5cm);
	\draw (-120:0.5cm) -- (-120:1.5cm);
	\node at (180:.75cm) {$\star$};
	\node at (180:1.25cm) {$\star$};
	\draw [very thick] (0,0) circle (.5cm);
	\draw [very thick] (0,0) circle (1.5cm);
\end{tikzpicture}
}
\,,\,
\underset{\Psi(\alpha_3^*\in \sA\sG(5,4))}{
\begin{tikzpicture}[annular]
	\clip (0,0) circle (1.6cm);
	\draw (120:0.5cm) -- (120:1.5cm);
	\draw (60:0.5cm) -- (60:1.5cm);
	\draw (0:0.5cm) -- (0:.85cm);
	\filldraw (0:.85cm) circle (.05cm);
	\draw (-60:0.5cm) -- (-60:1.5cm);
	\draw (-120:0.5cm) -- (-120:1.5cm);
	\node at (180:.75cm) {$\star$};
	\node at (180:1.25cm) {$\star$};
	\draw [very thick] (0,0) circle (.5cm);
	\draw [very thick] (0,0) circle (1.5cm);
\end{tikzpicture}
}
\,,\,
\underset{\Psi(\alpha_4^*\in \sA\sG(5,4))}{
\begin{tikzpicture}[annular]
	\clip (0,0) circle (1.6cm);
	\draw (120:0.5cm) -- (120:1.5cm);
	\draw (60:0.5cm) -- (60:1.5cm);
	\draw (0:0.5cm) -- (0:1.5cm);
	\draw (-60:0.5cm) -- (-60:.85cm);
	\filldraw (-60:.85cm) circle (.05cm);
	\draw (-120:0.5cm) -- (-120:1.5cm);
	\node at (180:.75cm) {$\star$};
	\node at (180:1.25cm) {$\star$};
	\draw [very thick] (0,0) circle (.5cm);
	\draw [very thick] (0,0) circle (1.5cm);
\end{tikzpicture}
}
\,,\,
\underset{\Psi(\alpha_5^*\in \sA\sG(5,4))}{
\begin{tikzpicture}[annular]
	\clip (0,0) circle (1.6cm);
	\draw (120:0.5cm) -- (120:1.5cm);
	\draw (60:0.5cm) -- (60:1.5cm);
	\draw (0:0.5cm) -- (0:1.5cm);
	\draw (-60:0.5cm) -- (-60:1.5cm);
	\draw (-120:0.5cm) -- (-120:.85cm);
	\filldraw (-120:.85cm) circle (.05cm);
	\node at (180:.75cm) {$\star$};
	\node at (180:1.25cm) {$\star$};
	\draw [very thick] (0,0) circle (.5cm);
	\draw [very thick] (0,0) circle (1.5cm);
\end{tikzpicture}
}
$$
\item
$\tau\in \sA\sG_n$ maps to the counter-clockwise one click rotation in $\sA\sP_n$,
and $\tau^*=\tau^{-1}$ maps to the clockwise one click rotation.
$$
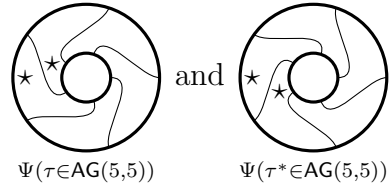
\underset{\Psi(\tau\in \sA\sG(5,5))}{
\begin{tikzpicture}[annular]
	\clip (0,0) circle (1.6cm);
	\draw (216:0.5cm) .. controls ++(216:.6cm) and ++(324:.4cm) .. (144:1.5cm);
	\draw (144:0.5cm) .. controls ++(144:.6cm) and ++(252:.4cm) .. (72:1.5cm);
	\draw (72:0.5cm) .. controls ++(72:.6cm) and ++(180:.4cm) .. (0:1.5cm);
	\draw (0:0.5cm) .. controls ++(0:.6cm) and ++(108:.4cm) .. (-72:1.5cm);
	\draw (-72:0.5cm) .. controls ++(-72:.6cm) and ++(36:.4cm) .. (-144:1.5cm);
	\node at (158:.75cm) {$\star$};
	\node at (180:1.25cm) {$\star$};
	\draw [very thick] (0,0) circle (.5cm);
	\draw [very thick] (0,0) circle (1.5cm);
\end{tikzpicture}
}
\text{ and }
\underset{\Psi(\tau^*\in \sA\sG(5,5))}{
\begin{tikzpicture}[annular]
	\clip (0,0) circle (1.6cm);
	\draw (72:0.5cm) .. controls ++(72:.6cm) and ++(324:.4cm) .. (144:1.5cm);
	\draw (0:0.5cm) .. controls ++(0:.6cm) and ++(252:.4cm) .. (72:1.5cm);
	\draw (-72:0.5cm) .. controls ++(-72:.6cm) and ++(180:.4cm) .. (0:1.5cm);
	\draw (-144:0.5cm) .. controls ++(-144:.6cm) and ++(108:.4cm) .. (-72:1.5cm);
	\draw (144:0.5cm) .. controls ++(144:.6cm) and ++(36:.4cm) .. (-144:1.5cm);
	\node at (-158:.75cm) {$\star$};
	\node at (180:1.25cm) {$\star$};
	\draw [very thick] (0,0) circle (.5cm);
	\draw [very thick] (0,0) circle (1.5cm);
\end{tikzpicture}
}
$$
\end{itemize}
One sees that these tangles satisfy the relations of $\sA\sG$ by drawing the appropriate diagrams.

We define $\Psi^{-1}$ by its $\C$-linear extension on tangles from $\sA\sP$. 
Given an annular tangle $T\in \sA\sP(m,n)$, there is a unique $r$ satisfying
$$
0\leq r< \#(\text{through strings of }T)
$$ 
which is the number of through strings to the left of the inner $\star$ that one must cross to get to the region which meets the outer $\star$.
We call this $r$ the relative star position of $T$.
Now, $\Psi^{-1}(T)\in \sA\sG(m,n)$ is the word in standard form
$$
\Psi^{-1}(T)=\alpha_{i_k}\cdots \alpha_{i_1}\tau^r\alpha_{j_1}^*\cdots \alpha_{j_\ell}^*
$$
where $j_1<\cdots < j_\ell$ are the positions of the caps of $T$, $r$ is the relative star position, and $i_1<\cdots < i_k$ are the positions of the cups of $T$.
That $\Psi^{-1}\circ \Psi= \id_{\sA\sG}$ and $\Psi\circ\Psi^{-1}=\id_{\sA\sP}$ follows immediately.
\end{proof}

\begin{defn}
The \underline{rectangular GICAR category} $\sR\sG$ is the subcategory of $\sA\sG$ such that $\sR\sG(m,n)$ consists of all $\C$-linear combinations of words $w$ on $\alpha_i,\tau,\alpha_j^*$ such that in the standard form of $w$ afforded by Proposition \ref{prop:Standard}, no $\tau$ appears, i.e., $r=0$.
\end{defn}

\begin{thm}\label{thm:RectangularEquivalence}
There is a $*$-equivalence of categories $\sR\sP\cong \sR\sG$.
\end{thm}
\begin{proof}
First, it is clear the functor $\Psi$ constructed in Theorem \ref{thm:Equivalence} restricts to a $*$-equivalence $\sR\sP\cong \sR\sG$.
In fact, 
\begin{itemize}
\item
$\alpha_i\in\sR\sG(n,n+1)$ maps to the diagram with $n$ lower boundary points, $n+1$ upper boundary points, a cup attached to lower boundary point $i$, and all other boundary points connected by undotted through strings.
$$
\hspace{-.5cm}
\underset{\Psi(\alpha_1\in\sR\sG(3,4))}{
\begin{tikzpicture}[baseline=-.1cm]
	\nbox{}{(0,0)}{.4}{.2}{0}{}
	\BrokenTopString{(-.4,0)}
	\String{(-.2,0)}
	\String{(0,0)}
	\String{(.2,0)}
\end{tikzpicture}
}
\,,\,
\underset{\Psi(\alpha_2\in\sR\sG(3,4))}{
\begin{tikzpicture}[baseline=-.1cm]
	\nbox{}{(0,0)}{.4}{.2}{0}{}
	\String{(-.4,0)}
	\BrokenTopString{(-.2,0)}
	\String{(0,0)}
	\String{(.2,0)}
\end{tikzpicture}
}
\,,\,
\underset{\Psi(\alpha_3\in\sR\sG(3,4))}{
\begin{tikzpicture}[baseline=-.1cm]
	\nbox{}{(0,0)}{.4}{.2}{0}{}
	\String{(-.4,0)}
	\String{(-.2,0)}
	\BrokenTopString{(0,0)}
	\String{(.2,0)}
\end{tikzpicture}
}
\,,\,
\underset{\Psi(\alpha_4\in\sR\sG(3,4))}{
\begin{tikzpicture}[baseline=-.1cm]
	\nbox{}{(0,0)}{.4}{.2}{0}{}
	\String{(-.4,0)}
	\String{(-.2,0)}
	\String{(0,0)}
	\BrokenTopString{(.2,0)}
\end{tikzpicture}
}
$$
\item
$\alpha_j^*\in\sR\sG(n,n-1)$ maps to $\Psi(\alpha_j)^*\in \sR\sP(n-1,n)$.
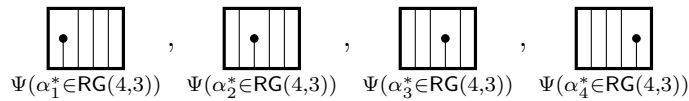
\begin{align*}
&
\underset{\Psi(\alpha_1^*\in\sR\sG(4,3))}{
\begin{tikzpicture}[baseline=-.1cm]
	\nbox{}{(0,0)}{.4}{.2}{0}{}
	\BrokenBottomString{(-.4,0)}
	\String{(-.2,0)}
	\String{(0,0)}
	\String{(.2,0)}
\end{tikzpicture}
}
\,,\,
\underset{\Psi(\alpha_2^*\in\sR\sG(4,3))}{
\begin{tikzpicture}[baseline=-.1cm]
	\nbox{}{(0,0)}{.4}{.2}{0}{}
	\String{(-.4,0)}
	\BrokenBottomString{(-.2,0)}
	\String{(0,0)}
	\String{(.2,0)}
\end{tikzpicture}
}
\,,\,
\underset{\Psi(\alpha_3^*\in\sR\sG(4,3))}{
\begin{tikzpicture}[baseline=-.1cm]
	\nbox{}{(0,0)}{.4}{.2}{0}{}
	\String{(-.4,0)}
	\String{(-.2,0)}
	\BrokenBottomString{(0,0)}
	\String{(.2,0)}
\end{tikzpicture}
}
\,,\,
\underset{\Psi(\alpha_4^*\in\sR\sG(4,3))}{
\begin{tikzpicture}[baseline=-.1cm]
	\nbox{}{(0,0)}{.4}{.2}{0}{}
	\String{(-.4,0)}
	\String{(-.2,0)}
	\String{(0,0)}
	\BrokenBottomString{(.2,0)}
\end{tikzpicture}
}
\qedhere
\end{align*}
\end{itemize}
\end{proof}

\begin{rem}
We can now pull back the tensor structure on $\sR\sP$ to get a tensor structure on $\sR\sG$.
The tensor product of objects is $[m]\otimes [n]=[m+n]$, and the tensor product of morphisms in standard form
\begin{align*}
\varphi&=
\alpha_{i_k}\cdots \alpha_{i_1}\alpha_{j_1}^*\cdots \alpha_{j_\ell}^*
\in \sR\sG(m_1,n_1)
\text{ and }
\\
\psi&=
\alpha_{i_{k'}'}\cdots a_{i_{1}'}\alpha_{j_{1}'}^*\cdots \alpha_{j_{\ell'}'}^*
\in \sR\sG(m_2,n_2)
\end{align*}
is given by
$$
\varphi\otimes \psi=
\alpha_{i_{k'}'+n_1}\cdots \alpha_{i_{1}'+n_1}\alpha_{i_k}\cdots \alpha_{i_1}
\alpha_{j_1}^*\cdots \alpha_{j_\ell}^*\alpha_{j_{1}'+m_1}^*\cdots \alpha_{j_{\ell'}'+m_1}^*
\in \sR\sG(m_1+m_2,n_1+n_2).
$$
With this tensor structure, the functor $\Psi$ in Theorem \ref{thm:RectangularEquivalence} is a $*,\otimes$-functor.
\end{rem}

%%%%%%%%%%%%%%%%%%%%%%%%%%%%%%%%%%%%%%%%%%%%%%%%%%%%%%%%%%%%
%%%%%%%%%%%%%%%%%%%%%%%%%%%%%%%%%%%%%%%%%%%%%%%%%%%%%%%%%%%%
%%%%%%%%%%%%%%%%%%%%%%%%%%%%%%%%%%%%%%%%%%%%%%%%%%%%%%%%%%%%
\section{Representation theory of the GICAR categories}\label{sec:Representations}

We now compute the representation theory of the GICAR categories $\sR\sG\cong \sR\sP$ and $\sA\sG\cong \sA\sP$ in the spirit of Graham and Lehrer's theory of cellular algebras \cite{MR1376244,MR1659204}.

%%%%%%%%%%%%%%%%%%%%%%%%%%%%%%%%%%%%%%%%%%%%%%%%%%
\subsection{A diagrammatic representation of the GICAR algebra}\label{sec:DigrammaticFermions}

We first give a diagrammatic description of the GICAR algebra acting on fermionic Fock space using the diagrams from $\sR\sP$ so that we may use Bigelow's dotted strand (Definition \ref{defn:DottedStrand}).
These diagrams implicitly appear in \cite{MR990778}, while our diagrams for fermionic Fock space arise from the cellular structure in the spirit of \cite{MR1376244,MR1659204}.

Our diagrammatic representation of $\GICAR(\cH)$ relies on choosing an orthonormal basis of $\cH$. 
This should neither surprise nor worry the reader for the following reason. 
Recall from Fact \ref{fact:ChooseBasis} that we must choose an orthonormal basis to show that $\CAR(\cH)\cong \bigotimes^\I M_2(\C)$ and to show that the Bratteli diagram for the tower of algebras $(G_n=A_n^{U(1)})_{n\geq 0}$ is given by Pascal's Triangle.
Since our diagrams in $\sR\sP$ are equivalent to those for the $\C P_n$'s, we are relying on the AF structure of $\GICAR(\cH)$, which relies on the choice of basis.

Suppose $\cH$ is a separable, infinite dimensional Hilbert space with a fixed choice of orthonormal basis $(\xi_i)_{i\geq 1}$.
Define $\cH_n = \spann\{\xi_1,\dots, \xi_n\}$.
We use the abbreviations  $a_i=a(\xi_i)$ and $a_j^*=a(\xi_j)^*$.
Recall the following facts about fermionic Fock space and the GICAR algebra. 

\begin{facts}
\mbox{}
\be
\item
An orthonormal basis of $\cF(\cH)$ is given by symbols of the form $\xi_{i_1}\wedge \cdots \wedge \xi_{i_n}$ for $i_1<i_2<\cdots < i_n$ together with the vacuum vector $\Omega$.
\item
By \cite[Lemma 2.2]{MR990778}, $\GICAR(\cH_n)$ has a presentation as a $*$-algebra with generators $f_i$ for $i=1,\dots, n$ and $u_i$ for $i=1,\dots, n-1$ and relations:
\begin{align}
&\text{$f_i=f_i^*=f_i^2$}
\tag{G1}
\label{rel:G1}
\\
&\text{$[f_i,f_j]=0$ if $j\neq i$}
\tag{G2}
\label{rel:G2}
\\
&\text{$[u_i,f_j]=0$ if $j\neq i,i+1$}
\tag{G3}
\label{rel:G3}
\\
&\text{$[u_i,u_j]=[u_i,u_j^*]=0$ if $|i-j|\geq 2$}
\tag{G4}
\label{rel:G4}
\\
&\text{$u_i^*u_i=f_{i+1}(1-f_i)$ and $u_iu_i^*=f_i(1-f_{i+1})$}
\tag{G5}
\label{rel:G5}
\end{align}
The isomorphism is given by $f_i\mapsto a_i^*a_i$ and $u_i\mapsto a_i^*a_{i+1}$.
\ee
\end{facts}

We now construct a diagrammatic Hilbert space $\cD_n$ on which we represent $\sR\sP_n$. 
We then give a spatial isomorphism $\Theta_n: \cF(\cH_n)\to \cD_n$ and a $*$-isomorphism of algebras $\theta_n: G_n\to \sR\sP_n$ which intertwines the actions, i.e., for all $\eta,\zeta\in \cF(\cH_n)$ and all $x,y\in G_n$, we have 
\begin{align}
\langle \eta,\zeta\rangle_{\cF(\cH_n)} 
&=
\langle  \Theta_n(\eta),\Theta_n(\zeta)\rangle_{\cD_n},
\label{rel:D1}
\tag{D1}
\\
\theta_n(xy^*)
&=
\theta_n(x)\theta_n(y)^*,
\text{ and}
\label{rel:D2}
\tag{D2}
\\
\Theta_n(x\eta)
&=
\theta_n(x)\Theta_n(\eta)
\label{rel:D3}
\tag{D3}
\end{align}

\begin{defn}
For $n\geq 0$, let $\cD_n$ be the complex span of diagrams in $\sR\sP$ with $n$ top boundary points, at most $n$ bottom boundary points, and no caps, such that all through strings are dotted.
Define an inner product on $\cD_n$ by declaring the diagrammatic basis of $\sR\sP$ to be orthonormal.
Let $\sR\sP_n$ act on $\cD_n$ by the usual composition of maps in $\sR\sP$.
\end{defn}

\begin{defn}
We define $\Theta_n: \cF(\cH_n)\to \cD_n$ as follows.
Let $\Theta_n(\xi_{i_1}\wedge\cdots \wedge \xi_{i_k})$ where $i_1<\cdots <i_k$ and $k\leq n$ be the diagram with $n$ upper boundary points, $n-k$ lower boundary points, cups in the $i_\ell$-th positions for all $\ell=1,\dots, k$, and all other strings are dotted through strings.
For example, when $k\leq 2$, we have
\begin{align*}
\Omega&\longmapsto
\begin{tikzpicture}[baseline=-.1cm]
	\nbox{unshaded}{(0,0)}{.4}{.2}{.2}{}
	\DottedString{(-.4,0)}
	\node at (.05,0) {$\cdots$};	
	\DottedString{(.4,0)}
	\node at (.4,.6) {\scriptsize{$n$}};
\end{tikzpicture}\,,
\\
\xi_i&\longmapsto
\begin{tikzpicture}[baseline=-.1cm]
	\nbox{unshaded}{(0,0)}{.4}{.2}{1.4}{}
	\DottedString{(-.4,0)}
	\node at (.05,0) {$\cdots$};	
	\DottedString{(.4,0)}
	\coordinate (a) at (.6,0);
	\BrokenTopString{(a)};
	\node at ($(a)+(0,.6)$)  {\scriptsize{$i$}};
	\DottedString{($(a)+(.2,0)$)}
	\node at ($(a)+(.65,0)$) {$\cdots$};
	\DottedString{($(a)+(1,0)$)}
	\node at ($(a)+(1,.6)$) {\scriptsize{$n$}};
\end{tikzpicture}\,,\text{ and}
\\
\xi_i\wedge\xi_j &\longmapsto
\begin{tikzpicture}[baseline=-.1cm]
	\nbox{unshaded}{(0,0)}{.4}{.2}{2.6}{}
	\DottedString{(-.4,0)}
	\node at (.05,0) {$\cdots$};	
	\DottedString{(.4,0)}
	\coordinate (a) at (.6,0);
	\BrokenTopString{(a)};
	\node at ($(a)+(0,.6)$)  {\scriptsize{$i$}};
	\DottedString{($(a)+(.2,0)$)}
	\node at ($(a)+(.65,0)$) {$\cdots$};
	\DottedString{($(a)+(1,0)$)}
	\coordinate (b) at ($(a)+(1.2,0)$);
	\BrokenTopString{(b)};
	\node at ($(b)+(0,.6)$)  {\scriptsize{$j$}};
	\DottedString{($(b)+(.2,0)$)}
	\node at ($(b)+(.65,0)$) {$\cdots$};
	\DottedString{($(b)+(1,0)$)}
	\node at ($(b)+(1,.6)$) {\scriptsize{$n$}};
\end{tikzpicture}
\text{ for }i<j.
\end{align*}
\end{defn}

\begin{thm}\label{thm:IsoOfTowers}
Define the map $\theta_n : G_n \to \sR\sP_n$ by 
$$
a_i^*a_i \mapsto
\begin{tikzpicture}[baseline=-.1cm]
	\nbox{unshaded}{(0,0)}{.4}{.3}{1.1}{}
	\String{(-.5,0)};
	\node at (-.15,0) {$\cdots$};
	\coordinate (a) at (.4,0);
	\String{($(a)+(-.2,0)$)};
	\DottedString{(a)};
	\node at ($(a)+(0,.6)$) {\scriptsize{$i$}};
	\String{($(a)+(.2,0)$)};
	\node at ($(a)+(.55,0)$) {$\cdots$};
	\String{($(a)+(.9,0)$)};
	\node at ($(a)+(.9,.6)$) {\scriptsize{$n$}};
\end{tikzpicture}\,,
\text{ and }
a_i^*a_{i+1} \mapsto
\begin{tikzpicture}[baseline=-.1cm]
	\nbox{unshaded}{(0,0)}{.4}{.3}{1.3}{}
	\String{(-.5,0)};
	\node at (-.15,0) {$\cdots$};
	\coordinate (a) at (.4,0);
	\String{($(a)+(-.2,0)$)};
	\PartialDottedIsoStar{(a)}
	\node at ($(a)+(0,.6)$) {\scriptsize{$i$}};
	\String{($(a)+(.4,0)$)};
	\node at ($(a)+(.75,0)$) {$\cdots$};
	\String{($(a)+(1.1,0)$)};
	\node at ($(a)+(1.1,.6)$) {\scriptsize{$n$}};
\end{tikzpicture}\,.
$$
Then $\Theta_n$ and $\theta_n$ satisfy Equations \eqref{rel:D1}-\eqref{rel:D3}.
\end{thm}
\begin{proof}
First, Equation \eqref{rel:D1} holds since $\Theta_n$ is a spatial isomorphism which maps an orthonormal basis to an orthonormal basis.
Second, by verifying that Relations \eqref{rel:G1}-\eqref{rel:G5} hold for $f_i=\theta_n(a_i^*a_i)$ and $u_i=\theta_n(a_i^*a_{i+1})$, we see that $\theta_n$ is an injective $*$-algebra homomorphism, since $\dim(G_n)=\dim(\sR\sP_n)$ by Proposition \ref{prop:Pascal}.

It remains to show Equation \eqref{rel:D3}. 
Since Relations \eqref{rel:D1}-\eqref{rel:D2} hold, it suffices to verify Equation \eqref{rel:D3} when $x$ is one of $a_i^*a_i,a_i^*a_{i+1}$, and $\eta$ is of the form $\xi_{i_1}\wedge\cdots \wedge \xi_{i_k}$ where $i_1<\cdots <i_k$.
Clearly
\begin{align*}
a_i^*a_i (\xi_{i_1}\wedge\cdots \wedge \xi_{i_k})
&=
\begin{cases}
0 & \text{if }i\in \{i_1,\dots, i_k\}\\
\xi_{i_1}\wedge\cdots \wedge \xi_{i_k} & \text{otherwise, and}
\end{cases}
\\
a_i^*a_{i+1} (\xi_{i_1}\wedge\cdots \wedge \xi_{i_k})
&=
\begin{cases}
0 & \text{if }i\notin \{i_1,\dots, i_k\}\text{ and }\\ &i+1\in \{i_1,\dots, i_k\}\\
\xi_{i_1}\wedge\cdots\wedge \widehat{\xi_i}\wedge \xi_{i+1}\wedge\cdots \wedge \xi_{i_k} & \text{otherwise.}
\end{cases}
\end{align*}
The rest is straightforward using Relations \eqref{rel:RP1}-\eqref{rel:RP3}.
\end{proof}

\begin{rem}
There is an easy graphical description of the inner product. 
If $\eta,\zeta$ are single diagrams in $\cD_n$, we look at the composite $\zeta^*\eta$ in $\sR\sP$, which is well-defined since $\eta,\zeta$ both have $n$ top boundary points. 
We then use Relations \eqref{rel:RP1}-\eqref{rel:RP3}. 
If we get a non-zero composite, then $\zeta^*\eta$ consists of only dotted through strings. 
Thus it must be the case that $\eta=\zeta$, since the cap positions must agree, and $\langle \eta,\zeta\rangle = 1$.
We leave it to the reader to extend this discussion to a formal definition of the graphical inner product.
\end{rem}

\begin{rem}
The following diagram commutes where the maps $G_n\to G_{n+1}$ and $\cH_n\to \cH_{n+1}$ are the usual inclusions, the map $\sR\sP_n\to \sR\sP_{n+1}$ is the right inclusion, and the map $\cD_n\to \cD_{n+1}$ is adding a dotted through string on the right.
\[
\xymatrix{
G_n\ar[rr]\ar[dd]^{\theta_n} \ar[dr]&& G_{n+1}\ar[dd]^(.6){\theta_{n+1}}\ar[dr]
\\
& B(\cF(\cH_n))\ar[dd]^(.6){\Ad(\Theta_{n})}\ar[rr] && B(\cF(\cH_{n+1}))\ar[dd]^{\Ad(\Theta_{n+1})}
\\
\sR\sP_n\ar[rr]\ar[dr] && \sR\sP_{n+1}\ar[dr]
\\
& B(\cD_n)\ar[rr] && B(\cD_{n+1})
}
\]
\end{rem}

\begin{ex}
We have the following diagrammatic representation of $G_3$ on $\cF(\cH_3)$.
\begin{align*}
\sR\sP(3,3)&=C^*\left\{
\underset{a_1a_1^*}{
\begin{tikzpicture}[baseline=-.1cm]
	\nbox{unshaded}{(0,0)}{.4}{0}{0}{}
	\BrokenString{(-.2,0)}
	\String{(0,0)}
	\String{(.2,0)}
\end{tikzpicture}
}
\,,\,
\underset{a_1a_2^*}{
\begin{tikzpicture}[baseline=-.1cm]
	\nbox{unshaded}{(0,0)}{.4}{0}{0}{}
	\PartialDottedIso{(-.2,0)}
	\String{(.2,0)}
\end{tikzpicture}
}
\,,\,
\underset{a_2a_1^*}{
\begin{tikzpicture}[baseline=-.1cm,]
	\nbox{unshaded}{(0,0)}{.4}{0}{0}{}
	\PartialDottedIsoStar{(-.2,0)}
	\String{(.2,0)}
\end{tikzpicture}
}
\,,\,
\underset{a_2a_2^*}{
\begin{tikzpicture}[baseline=-.1cm]
	\nbox{unshaded}{(0,0)}{.4}{0}{0}{}
	\String{(-.2,0)}
	\BrokenString{(0,0)}
	\String{(.2,0)}
\end{tikzpicture}
}
\,,\,
\underset{a_2a_3^*}{
\begin{tikzpicture}[baseline=-.1cm]
	\nbox{unshaded}{(0,0)}{.4}{0}{0}{}
	\PartialDottedIso{(0,0)}
	\String{(-.2,0)}
\end{tikzpicture}
}
\,,\,
\underset{a_3a_2^*}{
\begin{tikzpicture}[baseline=-.1cm,]
	\nbox{unshaded}{(0,0)}{.4}{0}{0}{}
	\PartialDottedIsoStar{(0,0)}
	\String{(-.2,0)}
\end{tikzpicture}
}
\,,\,
\underset{a_3a_3^*}{
\begin{tikzpicture}[baseline=-.1cm]
	\nbox{unshaded}{(0,0)}{.4}{0}{0}{}
	\String{(-.2,0)}
	\String{(0,0)}
	\BrokenString{(.2,0)}
\end{tikzpicture}
}
\right\}
\text{ and }
\\
\cD_3&=
\spann
\left\{
\underset{\Omega}{
\begin{tikzpicture}[baseline=-.1cm]
	\nbox{unshaded}{(0,0)}{.4}{0}{0}{}
	\DottedString{(-.2,0)};
	\DottedString{(0,0)};
	\DottedString{(.2,0)};	
\end{tikzpicture}
}
\,,\,
\underset{\xi_1}{
\begin{tikzpicture}[baseline=-.1cm]
	\nbox{unshaded}{(0,0)}{.4}{0}{0}{}
	\BrokenTopString{(-.2,0)};
	\DottedString{(0,0)};
	\DottedString{(.2,0)};	
\end{tikzpicture}
}
\,,\,
\underset{\xi_2}{
\begin{tikzpicture}[baseline=-.1cm]
	\nbox{unshaded}{(0,0)}{.4}{0}{0}{}
	\DottedString{(-.2,0)};
	\BrokenTopString{(0,0)};
	\DottedString{(.2,0)};	
\end{tikzpicture}
}
\,,\,
\underset{\xi_3}{
\begin{tikzpicture}[baseline=-.1cm]
	\nbox{unshaded}{(0,0)}{.4}{0}{0}{}
	\DottedString{(-.2,0)};
	\DottedString{(0,0)};	
	\BrokenTopString{(.2,0)};
\end{tikzpicture}
}
\,,\,
\underset{\xi_1\wedge \xi_2}{
\begin{tikzpicture}[baseline=-.1cm]
	\nbox{unshaded}{(0,0)}{.4}{0}{0}{}	
	\BrokenTopString{(-.2,0)};
	\BrokenTopString{(0,0)};
	\DottedString{(.2,0)};
\end{tikzpicture}
}
\,,\,
\underset{\xi_1\wedge \xi_3}{
\begin{tikzpicture}[baseline=-.1cm]
	\nbox{unshaded}{(0,0)}{.4}{0}{0}{}
	\BrokenTopString{(-.2,0)};
	\DottedString{(0,0)};
	\BrokenTopString{(.2,0)};
\end{tikzpicture}
}
\,,\,
\underset{\xi_2\wedge \xi_3}{
\begin{tikzpicture}[baseline=-.1cm]
	\nbox{unshaded}{(0,0)}{.4}{0}{0}{}
	\DottedString{(-.2,0)};	
	\BrokenTopString{(0,0)};
	\BrokenTopString{(.2,0)};
\end{tikzpicture}
}
\,,\,
\underset{\xi_1\wedge \xi_2\wedge \xi_3}{
\begin{tikzpicture}[baseline=-.1cm]
	\nbox{unshaded}{(0,0)}{.4}{0}{0}{}
	\BrokenTopString{(-.2,0)};
	\BrokenTopString{(0,0)};
	\BrokenTopString{(.2,0)};
\end{tikzpicture}
}
\right\}.
\end{align*}
However, note that we need a linear combination of diagrams to represent $a_1a_3^*$ and $a_3a_1^*$.
Using Relations \eqref{rel:CAR2}-\eqref{rel:CAR3}, we have
\begin{align*}
a_1a_3^*
&=a_1a_2^*a_2a_3^*+a_1a_2a_2^*a_3^*
=a_1a_2^*a_2a_3^*+a_2a_1a_3^*a_2^*
=a_1a_2^*a_2a_3^*-a_2a_3^*a_1a_2^*
\\
&\longmapsto
\begin{tikzpicture}[baseline=-.1cm]
	\draw[very thick] (-.4,-.8) rectangle (.4,.8);
	\draw[very thick] (-.4,0) -- (.4,0);
	\PartialDottedIso{(-.2,.4)}
	\String{(.2,.4)}
	\PartialDottedIso{(0,-.4)}
	\String{(-.2,-.4)}	
\end{tikzpicture}
-
\begin{tikzpicture}[baseline=-.1cm]
	\draw[very thick] (-.4,-.8) rectangle (.4,.8);
	\draw[very thick] (-.4,0) -- (.4,0);
	\PartialDottedIso{(-.2,-.4)}
	\String{(.2,-.4)}
	\PartialDottedIso{(0,.4)}
	\String{(-.2,.4)}	
\end{tikzpicture}
=
\begin{tikzpicture}[baseline=-.1cm]
	\nbox{unshaded}{(0,0)}{.4}{.2}{.2}{}
	\draw (-.4,-.4) -- (0,.4);
	\draw (0,-.4) -- (.4,.4);
	\draw (-.4,.4) -- (-.4,.15);
	\draw (.4,-.4) -- (.4,-.15);
	\filldraw (-.4,.15) circle (.05cm);
	\filldraw (.4,-.15) circle (.05cm);
	\filldraw (-.2,0) circle (.05cm);
	\filldraw (.2,0) circle (.05cm);
\end{tikzpicture}
-
\begin{tikzpicture}[baseline=-.1cm]
	\nbox{unshaded}{(0,0)}{.4}{.2}{.2}{}
	\draw (-.4,-.4) -- (.4,.4);
	\draw (-.4,.4) -- (-.4,.15);
	\draw (.4,-.4) -- (.4,-.15);
	\draw (.1,-.4) -- (.1,-.15);
	\draw (-.1,.4) -- (-.1,.15);
	\filldraw (-.4,.15) circle (.05cm);
	\filldraw (.4,-.15) circle (.05cm);
	\filldraw (.1,-.15) circle (.05cm);
	\filldraw (-.1,.15) circle (.05cm);
	\filldraw (0,0) circle (.05cm);
\end{tikzpicture}
\,.
\end{align*}
\end{ex}

\begin{rem}
At this point, we do not know if it is possible to use these diagrams or similar ones to represent $\CAR(\cH)$ on $\cF(\cH)$. 
One might be tempted to define $a_1$ by a cup in the first position on the top.
In order for $a_1$ to kill $\xi_1$, we must connect a dotted string to the first lower point. 
However, if this dotted string connected to upper point $i$, the image of $a_1$ would never contain an antisymmetric tensor containing a $\xi_i$, which is absurd.
\end{rem}

\begin{rem}\label{rem:DiagrammaticGICARRep}
We now get a nice diagrammatic description of the representations given in Theorem \ref{thm:GICARRepresentations} part (2). 
There are exactly $n\choose k$ minimal projections in $\sR\sP_n$ with exactly $k$ through strings, all of which are dotted. Each of these minimal projections $p$ generates a copy of $\Lambda^{n-k}\cH_n$ as $\sR\sP_np$, where we ignore the bottom of the broken strands. For example, if we have the minimal projection with $k$ broken strings on the left and $n-k$ doted through strings on the right,
$$
\sR\sP_n
\left(
\begin{tikzpicture}[baseline=-.1cm]
	\nbox{}{(0,0)}{.4}{.6}{.8}{};
	\BrokenString{(-.8,0)};
	\node at (-.35,0) {$\dots$};
	\BrokenString{(0,0)};	
	\DottedString{(.2,0)};	
	\node at (.65,0) {$\dots$};	
	\DottedString{(1,0)};	
\end{tikzpicture}
\right)
\cong
\sR\sP_n
\left(
\begin{tikzpicture}[baseline=-.1cm]
	\nbox{}{(0,0)}{.4}{.6}{.8}{};
	\BrokenTopString{(-.8,0)};
	\node at (-.35,0) {$\dots$};
	\BrokenTopString{(0,0)};	
	\DottedString{(.2,0)};	
	\node at (.65,0) {$\dots$};	
	\DottedString{(1,0)};	
\end{tikzpicture}
\right)
\cong
G_n(
\xi_1\wedge \cdots \wedge \xi_k)
=
\Lambda^k \cH_n.
$$
It would also be interesting to describe these representations using a unitary version of Graham and Lehrer's theory of cellular algebras \cite{MR1376244,MR1659204}.
\end{rem}

%%%%%%%%%%%%%%%%%%%%%%%%%%%%%%%%%%%%%%%%%%%%%%%%%%
\subsection{Representations of small involutive categories}

We now discuss the representation theory of small involutive categories, where for simplicity, we work with finite dimensional complex Hilbert spaces.
Our treatment is along the lines of \cite[Sections 2-3]{MR1929335}. We provide proofs for completeness.

\begin{defn}
Suppose $\sC$ is a small involutive category whose hom spaces are finite dimensional complex vector spaces.
A \underline{Hilbert $\sC$-module} is a $*$-functor $V: \sC\to {\sf Hilb}$, the category of finite dimensional Hilbert spaces and linear maps. 
We denote $V(n)$ by $V_n$ for $n\in\sC$, and we just use the notation $c$ for $V(c)\in B(V_m,V_n)$ when $c\in\sC(m,n)$.
This means for all $\xi\in V_m$ and $\eta\in V_n$, we have 
$$
\langle c \xi, \eta\rangle_{V_n}=\langle \xi, c^* \eta\rangle_{V_m}.
$$ 
\end{defn}

\begin{rem}
Sometimes one defines a $\sC$-module as a functor originating in $\sC^{\sf op}$, e.g., simplicial sets are functors ${\sf \Delta^{op}}\to {\sf Set}$.
Since $\sC$ is involutive, $\sC\cong \sC^{\sf op}$ via the involution, so we just use covariant functors.
\end{rem}

\begin{defn}
We call a Hilbert $\sC$-module 
\begin{itemize}
\item
\underline{indecomposable} if it cannot be written as the direct sum of two non-trivial orthogonal submodules, or
\item
\underline{irreducible} if it has no proper submodules, i.e., any non-zero element in $V_m$ for any $m$ generates all of $V$.
\end{itemize}
\end{defn}

\begin{lem}\label{lem:IrreducibleAndIndecomposable}
Suppose $V$ is a Hilbert $\sC$-module.
Then $V$ is irreducible if and only if $V$ is indecomposable. 
\end{lem}
\begin{proof}
Suppose $V$ is indecomposable, and let $\xi\in V_n$ for some $n\in\sC$. 
Define Hilbert $\sC$-modules $W,W^\perp$ by $W_m=\sC(m,n)\xi$ and $W_m^\perp$ is the usual orthogonal complement for all $m\geq 0$.
Then $V=W\oplus W^\perp$, so $W^\perp$ must be the zero module, i.e., $W_m=V_m$ for all $m$.
Thus $V$ is irreducible.

Suppose now that $V$ is irreducible, and suppose $V=W\oplus X$ for orthogonal Hilbert $\sC$-modules $W$ and $X$. 
Suppose $\xi\in W_m$ is non-zero for some $m\in\sC$.
Let $\eta\in X_n$. 
Then for all $c\in\sC(m,n)$, we have $\langle c\xi,\eta \rangle_{V_n} = 0$, but $\sC(m,n)\xi=V_n$, so $\eta=0$. 
Hence $X$ is the zero module, and $V$ is indecomposable.
\end{proof}

\begin{lem}\label{lem:Orthogonal}
Suppose $V$ is a Hilbert $\sC$-module.
Suppose $W_m$ and $X_m$ are orthogonal $\sC(m,m)$-invariant subspaces of $V_m$. Then $\sC(W)$ is orthogonal to $\sC(X)$.
\end{lem}
\begin{proof}
If $\xi=c_1\xi_0$ for some $c_1\in\sC(m,n)$ and $\xi_0\in W_m$, and $\eta=c_2 \eta_0$ for some $c_2\in \sC(m,n)$ and $\eta_0\in X_m$, then
$$
\langle \xi,\eta \rangle_{V_n} = \langle c_1\xi_0, c_2\eta_0\rangle_{V_n} = \langle c_2^*c_1 \xi_0, \eta_0\rangle_{V_m}=0
$$
since $c_2^*c_1\in \sC(m,m)$ and $W_m$ is $\sC(m,m)$-invariant.
\end{proof}

\begin{lem}\label{lem:IrreducibleSubmodule}
Suppose $W\subset V_m$ is an irreducible $\sC(m,m)$-module for some $m$. Then $W_n=\sC(W)_n$ is irreducible for all $n$. 
\end{lem}
\begin{proof}
Suppose $\xi\in W_n$ is nonzero, and let $\eta\in W_n$ be another vector. 
Write  $\xi=c_1 \xi_0$ and $\eta= c_2 \eta_0$ for $\eta_0,\xi_0\in W$ and $c_1,c_2\in\sC(m,n)$.
Then $c_1^*\xi = c_1^*c_1 \xi_0\in W$ is non-zero, so there is a $c_3\in \sC(m,m)$ with $c_3c_1^*\xi=\eta_0$. 
Then $c_2c_3c_1^*\xi=\eta$, and $W_n$ is irreducible.
\end{proof}

\begin{assume}\label{assume:Generating}
We now assume that our Hilbert $\sC$-module $V$ satisfies the following generating property:
\begin{itemize}
\item
For any two objects $m,n\in\sC$ such that $V_m,V_n\neq (0)$, the image of $\sC(m,n)$ in $B(V_m,V_n)$ contains a non-zero map.
\end{itemize}
\end{assume}

\begin{lem}\label{lem:AllIrreducible}
The following are equivalent.
\be
\item
$V_m$ is an irreducible $\sC(m,m)$-module for all $m$,
and 
\item
$V$ is irreducible.
\ee
\end{lem}
\begin{proof}
\itm{(1)\Rightarrow (2)}
Suppose $\xi\in V_m$ and $\eta\in V_n$ are nonzero. 
There is a nonzero map $c_2\in\sC(m,n)$ by Assumption \ref{assume:Generating}, so there is an $\zeta\in V_m$ such that $c_2 \zeta\in V_n\setminus\{0\}$. 
Since $V_m$ is irreducible, there is a $c_1\in \sC(m,m)$ such that $c_1\xi=\zeta$.
Since $V_n$ is irreducible, there is a $c_3\in \sC(n,n)$ such that $\eta=c_3c_2\zeta=c_3c_2c_1\xi$.

\itm{(2)\Rightarrow (1)}
We prove the contrapositive.
Suppose $V_m$ has a non-zero proper $\sC(m,m)$-module $W_m$ for some $m$.
Since $\sC(m,m)$ acts as a $*$-algebra of bounded operators on $V_m$, $W_m^\perp$ is also a non-zero proper submodule of $V_m$.
Applying $\sC$ to $W_m$ and $W_m^\perp$ yields two proper $\sC$-submodules of $V$ which are orthogonal by Lemma \ref{lem:Orthogonal}.
\end{proof}

\begin{assume}\label{assume:ObjectsAreN}
We now assume that the objects of $\sC$ are the symbols $[n]$ for $n\geq 0$, which come with the usual total order on $\N\cup\{0\}$.
Moreover, we assume that for $m\leq n$, there is a monomorphism in $\sC(m,n)$, so that Assumption \ref{assume:Generating} is satisfied.
\end{assume}

\begin{ex}
The (annular) Temperley-Lieb and the (annular) GICAR categories satisfy Assumption \ref{assume:ObjectsAreN}.
\end{ex}

\begin{defn}
The \underline{weight} $\wt(V)$ of a $\sC$-module $V$ is the least number $n$ such that $V_n$ is non-zero.
Elements of $V_{\wt(V)}$ are called \underline{lowest weight vectors}.
\end{defn}

\begin{prop}\label{prop:IrreducibleDecomposition}
Every Hilbert $\sC$-module has a canonical decomposition as an orthogonal direct sum of irreducible Hilbert $\sC$-modules.
\end{prop}
\begin{proof}
First decompose $V_{\wt(V)}$ into an orthogonal direct sum of irreducible $\sC_{\wt(V)}$-modules.
The direct summands generate irreducible $\sC$-modules by Lemmas \ref{lem:IrreducibleSubmodule} and \ref{lem:AllIrreducible}, all of which are mutually orthogonal by \ref{lem:Orthogonal}.
The orthogonal complement of these $\sC$-modules have higher weight, so we are finished by an induction argument.
\end{proof}

\begin{lem}\label{lem:ExtendMorphism}
If $V,W$ are two Hilbert $\sC$-modules with $V$ irreducible, and $\theta : V_m \to W_m$ is a non-zero homomorphism of $\sC_m$-modules, then $\theta$ extends uniquely to an injective homomorphism $\Theta$ of Hilbert $\sC$-modules.
\end{lem}
\begin{proof}
Suppose $\xi\in V_n$. 
Then $\xi=a \eta$ for some $a\in \sC(m,n)$ and $\eta\in V_m$ since $V$ is irreducible. 
We claim $\Theta(\xi)=c\theta(\eta)$ gives a well-defined map of $\sC$-modules. 
If $\xi=b\eta$ for another $b\in \sC(m,n)$, then for any $c\in \sC(m,n)$ and $\zeta\in V_m$, we have
$$
\langle a\theta(\eta), c\zeta\rangle_{V_n}
= \langle c^*a \theta(\eta),\zeta\rangle 
= \langle \theta(c^*b \eta),\zeta\rangle 
= \langle c^*b \theta(\eta),\zeta\rangle 
= \langle b \theta(\eta),c \zeta\rangle,
$$
so $a\theta(\eta)=b\theta(\eta)$, so $\Theta$ is well-defined.
By construction $\Theta$ is a $\sC$-module map, and it is injective since $V$ is irreducible.
The uniqueness of the extension $\Theta$ of $\theta$ is obvious.
\end{proof}

%%%%%%%%%%%%%%%%%%%%%%%%%%%%%%%%%%%%%%%%%%%%%%%%%%
\subsection{The representation theory of $\sA\sG\cong \sA\sP$}

Since the annular GICAR category $\sA\sG\cong \sA\sP$ satisfies Assumption \ref{assume:Generating}, the lemmas from the last subsection apply, and we easily obtain the complete classification of the representations of $\sA\sG\cong \sA\sP$.
We work with $\sA\sP$ so we can work modulo the ideal of diagrams without the maximal number of through strings.
We give two equivalent constructions of the irreducible modules; the first follows the technique of \cite{MR1929335}, and the second uses the algebra decomposition of $\sA\sP_k$ given in Proposition \ref{prop:AnnularAlgebraStructure}.

\begin{nota}
We identify the maps $\alpha_i,\alpha_j^*,\tau\in \sA\sG$ with their images in $\sA\sP$ under the equivalence $\Psi$ given in Theorem \ref{thm:Equivalence}.
\end{nota}

\begin{defn}
Let $\sA\sI(k,m)$ be the space spanned by tangles in $\sA\sP(k,m)$ with fewer than $k$ through strings, so $\sA\sI(k,m)=(0)$ if $m<k$.
We will use the usual abbreviation $\sA\sI_k = \sA\sI(k,k)$.
Note that $\sA\sI_k$ has codimension $k$ in $\sA\sP_k$, and $\sA\sP_k/\sA\sI_k \cong \C[\tau]\cong \C[\Z/k]$.
\end{defn}

\begin{prop}\label{prop:AnnularDescend}
Let $V$ be a Hilbert $\sA\sP$-module, and let $W_k$ be the $\sA\sP_k$-submodule of $V_k$ generated by all the $\sA\sP$-submodules of $V$ with weight less than $k$. Then 
$$
W_k^\perp = \bigcap_{a\in \sA\sI_k}\ker(a).
$$
\end{prop}
\begin{proof}
Suppose $\xi\in W_k^\perp$, and let $a\in \sA\sI_k$. 
Then $a$ is a linear combination of elements of the form $b^*c$ where $b,c\in \sA\sP(k,m)$ with $m<k$. 
For any $\eta\in V_k$, we have $b\eta\in V_m$ with $m<k$, $c^*(b\eta)\in W_k$. 
Hence $\langle a \xi, \eta\rangle = \langle \xi ,c^* b\eta\rangle=0$, and thus $a\xi=0$.

Suppose $\xi\in \ker(a)$ for all $a\in \sA\sI_k$, and let $\eta\in W_k$. Then $\eta$ is a linear combination of elements of the form $b \zeta$ where $\zeta\in V_m$ and $b$ a single diagram in $\sA\sP(m,k)$ with $m<k$. 
\begin{claim} 
There is a diagram $c\in\sA\sI_k$ such that $b=cb$. 
\end{claim}
\begin{proof}
Let $c$ be the projection in $\sA\sP_k$ corresponding to a projection in $\sR\sP_k$ under the cutting operation in Figure \ref{fig:CuttingAndGluing} such that $c$ has only non-dotted through strings in the same positions as the (dotted) through strings of $b$.
Clearly $b=cb$, and since $b\in\sA\sP(m,k)$, $c$ has at most $m<k$ through strings, and thus $c\in \sA\sI_k$. 
\end{proof}
By the claim, $\langle \xi, \eta\rangle = \langle \xi,cb\zeta\rangle = \langle c^*\xi,b\zeta\rangle=0$, and $\xi\in W_k^\perp$.
\end{proof}

\begin{cor}\label{cor:AnnularDecend}
If $V$ is irreducible of weight $k$, then $V_k$ descends to an irreducible $\sA\sP_k/\sA\sI_k\cong \C[\tau]\cong \C[\Z/k]$-module, all of which are one-dimensional.
\end{cor}

\begin{defn}
Since the action of $\sA\sP(m,n)$ on $\sA\sP(k,m)$ can only decrease the number of through strings, the action maps $\sA\sI(k,m)$ to $\sA\sI(k,n)$, and thus the action descends to an action of $\sA\sP(m,n)$ on $\sA\sP(k,m)/\sA\sI(k,m)$.
Note that the diagrams of $\sA\sP(k,m)$ with exactly $k$ through strings descend to a basis of $\sA\sP(k,m)/\sA\sI(k,m)$ under the canonical surjection.
Hence we may think of the action of $\sA\sP(m,n)$ on $\sA\sP(k,m)/\sA\sI(k,m)$ as the usual action in $\sA\sP$, except with the rule that if a composite has fewer than $k$ through strings, then we get zero.

Now $\langle \tau\rangle \cong \Z/k$ acts on $\sA\sP(k,m)/\sA\sI(k,m)$ by internal rotation, freely permuting the diagrammatic basis, and the internal rotation action commutes with the external $\sA\sP(m,n)$ action.
This means $\sA\sP(k,m)/\sA\sI(k,m)$ splits into the eigenspaces of the rotation $\tau$.
Let $V^{k,\omega}_m$ be the eigenspace of $\sA\sP(k,m)/\sA\sI(k,m)$ associated to the rotational eigenvalue $\omega$. 
\end{defn}

\begin{prop}\label{prop:DimV}
$\D\dim(V^{k,\omega}_m)=
\begin{cases}
0 &\text{if }m<k\\
\D{n\choose k} &\text{if }m\geq k.
\end{cases}
$
\end{prop}
\begin{proof}
For $k\leq m$, the action of $\Z/k$ on $\sA\sP(k,m)/\sA\sI(k,m)$ is free, and
$$
\dim(\sA\sP(k,m)/\sA\sI(k,m))=N(k,m;k)=
\begin{cases}
\D k{m\choose k}{k-1 \choose k-1}=k{m\choose k} & \text{if }k>0\\
1 & \text{if }k=0
\end{cases}
$$ 
by Remark \ref{rem:CountTangles}  and Lemma \ref{lem:CountTangles}.
\end{proof}

\begin{defn}\label{defn:VBasis}
Let
$$
\xi_k^\omega=
\frac{1}{Z}
\sum_{j=1}^k \omega^{-j}\tau^j
\in\sA\sP_k.
$$
where $Z$ is a normalization constant to be determined later.
Note that $\tau \xi_k^\omega=\omega \xi^\omega_k$, and $\alpha_i^*\xi^\omega_k=0$ for all $i=1,\dots, k$, since it is in $V^{k,\omega}_{k-1}=(0)$.

By Proposition \ref{prop:DimV}, a basis for $V^{k,\omega}_m$ is given by $\set{\alpha_{i_{m-k}}\cdots\alpha_{i_{1}}\xi^\omega_k}{i_1<\cdots < i_{m-k}}$.
For example, $V^{2,\omega}_4$ has the following diagrammatic basis:
$$
\left\{
\underset{\alpha_2\alpha_1\xi^\omega_2}{
\begin{tikzpicture}[annular]
	\clip (0,0) circle (1.6cm);
	\AnnularBrokenTopString{135}
	\AnnularBrokenTopString{45}
	\AnnularString{-45}
	\AnnularString{-135}
	\node at (0:0) {$\xi^\omega_2$};
	\node at (180:.75cm) {$\star$};
	\node at (180:1.25cm) {$\star$};
	\draw [very thick] (0,0) circle (.5cm);
	\draw [very thick] (0,0) circle (1.5cm);
\end{tikzpicture}
}
\,,\,
\underset{\alpha_3\alpha_1\xi^\omega_2}{
\begin{tikzpicture}[annular]
	\clip (0,0) circle (1.6cm);
	\AnnularBrokenTopString{135}
	\AnnularString{45}
	\AnnularBrokenTopString{-45}
	\AnnularString{-135}
	\node at (0:0) {$\xi^\omega_2$};
	\node at (180:.75cm) {$\star$};
	\node at (180:1.25cm) {$\star$};
	\draw [very thick] (0,0) circle (.5cm);
	\draw [very thick] (0,0) circle (1.5cm);
\end{tikzpicture}
}
\,,\,
\underset{\alpha_4\alpha_1\xi^\omega_2}{
\begin{tikzpicture}[annular]
	\clip (0,0) circle (1.6cm);
	\AnnularBrokenTopString{135}
	\AnnularString{45}
	\AnnularString{-45}
	\AnnularBrokenTopString{-135}
	\node at (0:0) {$\xi^\omega_2$};
	\node at (180:.75cm) {$\star$};
	\node at (180:1.25cm) {$\star$};
	\draw [very thick] (0,0) circle (.5cm);
	\draw [very thick] (0,0) circle (1.5cm);
\end{tikzpicture}
}
\,,\,
\underset{\alpha_3\alpha_2\xi^\omega_2}{
\begin{tikzpicture}[annular]
	\clip (0,0) circle (1.6cm);
	\AnnularString{135}
	\AnnularBrokenTopString{45}
	\AnnularBrokenTopString{-45}
	\AnnularString{-135}
	\node at (0:0) {$\xi^\omega_2$};
	\node at (180:.75cm) {$\star$};
	\node at (180:1.25cm) {$\star$};
	\draw [very thick] (0,0) circle (.5cm);
	\draw [very thick] (0,0) circle (1.5cm);
\end{tikzpicture}
}
\,,\,
\underset{\alpha_4\alpha_2\xi^\omega_2}{
\begin{tikzpicture}[annular]
	\clip (0,0) circle (1.6cm);
	\AnnularString{135}
	\AnnularBrokenTopString{45}
	\AnnularString{-45}
	\AnnularBrokenTopString{-135}
	\node at (0:0) {$\xi^\omega_2$};
	\node at (180:.75cm) {$\star$};
	\node at (180:1.25cm) {$\star$};
	\draw [very thick] (0,0) circle (.5cm);
	\draw [very thick] (0,0) circle (1.5cm);
\end{tikzpicture}
}
\,,\,
\underset{\alpha_4\alpha_3\xi^\omega_2}{
\begin{tikzpicture}[annular]
	\clip (0,0) circle (1.6cm);
	\AnnularString{135}
	\AnnularString{45}
	\AnnularBrokenTopString{-45}
	\AnnularBrokenTopString{-135}
	\node at (0:0) {$\xi^\omega_2$};
	\node at (180:.75cm) {$\star$};
	\node at (180:1.25cm) {$\star$};
	\draw [very thick] (0,0) circle (.5cm);
	\draw [very thick] (0,0) circle (1.5cm);
\end{tikzpicture}
}
\right\}.
$$

\end{defn}

\begin{defn}\label{defn:AnnularInnerProduct}
We define an inner product on $V^{k,\omega}$ as follows. Let $\tr$ be a faithful trace on $\sA\sP_k/\sA\sI_k\cong \C[\tau]\cong \Z/n$, and extend $\tr$ to $\sA\sP_k$ via the quotient map.
Given $a,b\in\sA\sP(k,m)$, $b^*a\in\sA\sP_k$, so we define an inner product on $\sA\sP_m$ by $\langle a,b\rangle_m = \tr(b^*a)$. 
It is clear that $\langle \cdot,\cdot\rangle$ satisfies $\langle ab,c\rangle = \langle b,a^*c\rangle$, and the rotation $\tau$ is clearly unitary, so the decomposition into the $V^{k,\omega}$ is orthogonal.

We can now choose the $Z$ in Definition \ref{defn:VBasis} to be any scalar for which $\|\xi_k^\omega\|^2=\langle \xi^\omega_k ,\xi^\omega_k\rangle = 1$.
We immediately get that the that basis of $V^{k,\omega}_m$ given in Definition \ref{defn:VBasis} is orthonormal.
\end{defn}

\begin{prop}\label{prop:DetermineInnerProduct}
All inner products in $V^{k,\omega}$ are determined by:
\be
\item $\langle \xi^\omega_k ,\xi^\omega_k\rangle = 1$,
\item $\alpha_i^*\xi_k^\omega=0$ for all $i=1,\dots, k$, i.e., $\xi^\omega_k$ is \underline{uncappable}, and
\item $\tau \xi_k^\omega = \omega \xi^\omega_k$.
\ee
\end{prop}
\begin{proof}
$V^{k,\omega}_m$ is spanned by elements of the form $a \xi^\omega_k$ with $a$ a single diagram in $\sA\sP(k,m)$. 
Since $\xi^\omega_k$ is uncappable, we see that $\langle a \xi^\omega_k , b\xi^\omega_k\rangle$ is zero unless all through strings connected to one $\xi^\omega_k$ attach to the other. 
When $\xi^\omega_k$ is not capped off, we use the rotational eigenvector property to end up with some power of $\omega$ times $\langle \xi^\omega_k, \xi^\omega_k\rangle$.
\end{proof}

\begin{prop}\label{prop:UniqueAnnularModule}
Under this inner product, $V^{k,\omega}$ is an irreducible Hilbert $\sA\sP$-module of weight $k$.
Moreover, any irreducible Hilbert $\sA\sP$-module of weight $k$ is isomorphic to some $V^{k,\omega}$.
\end{prop}
\begin{proof}
We know $V^{k,\omega}$ is a Hilbert $\sA\sP$-module, since we have exhibited an orthonormal basis. Irreducibility now follows by Lemmas \ref{lem:IrreducibleSubmodule} and \ref{lem:AllIrreducible} since $V^{k,\omega}=\sA\sP(\xi^\omega_k)$.

If $W$ is another irreducible $\sA\sP$-module of weight $k$, then it is generated by a lowest weight rotational eigenvector vector $\eta\in W_k$ by Lemma \ref{lem:AllIrreducible} and Corollary \ref{cor:AnnularDecend}. 
Let $\omega$ be the rotational eigenvalue, and without loss of generality, assume $\|\eta\|_{W_k}=1$. 
Define $\theta : V^{k,\omega}_k \to W_k$ by $\theta(\xi^\omega_k) = \eta$, which is a non-zero homomorphism of $\sA\sP_k$-modules. 
By Lemma \ref{lem:ExtendMorphism}, $\theta$ extends uniquely to an injective homomorphism $\Theta: V^{k,\omega}\to W$
which preserves the inner product by Proposition \ref{prop:DetermineInnerProduct}.
It is clear $\Theta$ is an isomorphism, as we can construct its inverse similarly.
\end{proof}

\begin{rem}\label{rem:EquivalentDefinition}
We get the following equivalent characterization of $V^{k,\omega}$.
Let $\dot{\tau}$ be the dotted rotation operator, e.g.,
$$
\begin{tikzpicture}[annular]
	\clip (0,0) circle (1.6cm);
	\draw (216:0.5cm) .. controls ++(216:.6cm) and ++(324:.4cm) .. (144:1.5cm);
	\draw (144:0.5cm) .. controls ++(144:.6cm) and ++(252:.4cm) .. (72:1.5cm);
	\draw (72:0.5cm) .. controls ++(72:.6cm) and ++(180:.4cm) .. (0:1.5cm);
	\draw (0:0.5cm) .. controls ++(0:.6cm) and ++(108:.4cm) .. (-72:1.5cm);
	\draw (-72:0.5cm) .. controls ++(-72:.6cm) and ++(36:.4cm) .. (-144:1.5cm);
	\filldraw (18:1cm) circle (.05cm);
	\filldraw (90:1cm) circle (.05cm);
	\filldraw (162:1cm) circle (.05cm);
	\filldraw (-126:1cm) circle (.05cm);
	\filldraw (-54:1cm) circle (.05cm);
	\node at (158:.75cm) {$\star$};
	\node at (180:1.25cm) {$\star$};
	\draw [very thick] (0,0) circle (.5cm);
	\draw [very thick] (0,0) circle (1.5cm);
\end{tikzpicture}
=\dot{\tau}\in \sA\sP_5.
$$
Recall from Proposition \ref{prop:AnnularAlgebraStructure} that
$$
\sA\sP_k \cong \C \oplus \bigoplus_{j=1}^k j M_{k\choose j}(\C),
$$
where the $j$ matrix algebras of size $k\choose j$ correspond to the annuli with $j$ dotted through strings. 
This means that the $k$ powers of $\dot{\tau}$ correspond to $k$ copies of $\C$.

Define a non-faithful trace $\tr$ on $\sA\sP_k$ by mapping the minimal projections
$$
p_k^\omega=
\frac{1}{k}
\sum_{j=1}^k \omega^{-j}\dot{\tau}^j
\in\sA\sP_k
$$
to 1 and mapping all other minimal projections to zero.
Then under the usual GNS sesquilinear form $\langle a,b\rangle = \tr(b^*a)$, the minimal projections $p_k^\omega$ are orthonormal, and all other minimal projections have length zero.
Note that $\tau p_k^\omega=\omega \xi^\omega_k$, and $\alpha_i^*p_k^\omega=0$ for all $i=1,\dots, k$.

Now we can extend this sesquilinear form to $\sA\sP(k,m)$ as in Definition \ref{defn:AnnularInnerProduct}, since for any $a,b\in\sA\sP(k,m)$, $b^*a\in\sA\sP_k$.
Let $V^{k}_m$ be the quotient of $\sA\sP(k,m)$ by the radical of the sesquilinear form, so that $V^k_m$ is a Hilbert space, which naturally carries an action of $\sA\sP(m,n)$.

The decomposition of $V^k_k=\spann\set{p_k^\omega}{\omega^k=1}$ into orthogonal eigenspaces for the rotation is easy.
Set $V^{k,\omega}_k=\spann\{p_k^\omega\}$, and let $V^{k,\omega}$ be the irreducible Hilbert $\sA\sP$-submodule generated by $p_k^\omega$.
Proposition \ref{prop:UniqueAnnularModule} implies this definition is equivalent to the previous definition via the isomorphism which maps $p_k^\omega$ to $\xi^\omega_k$.
\end{rem}

%%%%%%%%%%%%%%%%%%%%%%%%%%%%%%%%%%%%%%%%%%%%%%%%%%
\subsection{The representation theory of $\sR\sG\cong \sR\sP$}
We now do the same for $\sR\sG\cong \sR\sP$. 
The proofs of the propositions in the subsection are similar to the proofs from the last subsection, and they will be omitted.

\begin{nota}
We now identify $\alpha_i,\alpha_j^*\in\sR\sG$ with their image in $\sR\sP$ under the restriction of the equivalence $\Psi$ as in Theorem \ref{thm:RectangularEquivalence}.
\end{nota}

\begin{defn}
Let $\sR \sI_n\subset \sR\sP_n$ denote the ideal generated by the diagrams with fewer than $n$ through strings.  
Note that $\sR\sI_n$ has codimension one in $\sR\sP_n$.
\end{defn}

\begin{prop}
Let $V$ be a Hilbert $\sR\sP$-module, and let $W_k$ be the $\sR\sP_k$-submodule of $V_k$ generated by all the $\sR\sP$-submodules of $V$ with weight less than $k$. Then 
$$
W_k^\perp = \bigcap_{a\in \sR\sI_k}\ker(a).
$$
\end{prop}

\begin{cor}\label{cor:RectangularDecend}
If $V$ is irreducible of weight $k$, then $V_k$ descends to an irreducible $\sR\sP_k/\sR\sI_k\cong \C$-module, which is one dimensional.
\end{cor}

\begin{defn}
We define the irreducible Hilbert $\sR\sP$-modules $V^k$ using the technique of Remark \ref{rem:EquivalentDefinition}.
For $k\geq 0$, let 
$$
p_k=
\begin{tikzpicture}[baseline=-.1cm]
	\nbox{}{(0,0)}{.4}{.2}{.2}{};
	\DottedString{(-.4,0)};
	\node at (.05,0) {$\dots$};
	\node at (.4,.6) {\scriptsize{$k$}};	
	\DottedString{(.4,0)};	
\end{tikzpicture}
\in\sR\sP_k,
$$
and note that $\alpha_i^*p_k=0$ for all $i=1,\dots, k$.
Define a sesquilinear form on $\sR\sP_k$ by declaring the minimal projection $p_k$ to be a unit vector, and all other minimal projections have length zero.
Extend the sesquilinear form to $\sR\sP(k,m)$ as before, and let $V^k_m$ be the quotient of $\sR\sP(k,m)$ by the radical of the form.
Then $V^k$ is an irreducible Hilbert $\sR\sP$-module of weight $k$. 
\end{defn}

\begin{prop}
$
\D
\dim(V^k_m) =
\begin{cases}
0 &\text{if }m<k\\
\D{n\choose k} &\text{if }m\geq k.
\end{cases}
$
\end{prop}

\begin{prop}\label{prop:DetermineRectangularInnerProduct}
All inner products in $V^{k}$ are determined by:
\be
\item $\langle p_k ,p_k\rangle = 1$, and
\item $\alpha_i^*p_k=0$ for all $i=1,\dots, k$, i.e., $p_k$ is uncappable.
\ee
\end{prop}

\begin{prop}
$V^k$ is the the unique irreducible Hilbert $\sR\sP$-module of weight $k$ up to isomorphism.
\end{prop}

\begin{rem}
It would be interesting to fully compute the decomposition of the Temperley-Lieb algebras $TL_n(\delta)$ as irreducible $\sA\sG$ and $\sR\sG$-modules.
For example, while there is only one Temperley-Lieb module for Temperley-Lieb, there are many GICAR modules.
It is well known that the $n$-th Catalan number counts the number of non-crossing partitions of $\{1,\dots,n\}$.
Using this fact, we expect to find for each $n\in \bbN$ a new low-weight generator corresponding to the partition which includes all $\{1,\dots, n\}$ and to use these elements to decompose the $TL_n(\delta)$ into irreducible modules by applying rotational symmetries.
We leave this for another time.
\end{rem}

%%%%%%%%%%%%%%%%%%%%%%%%%%%%%%%%%%%%%%%%%%%%%%%%%%
%%%%%%%%%%%%%%%%%%%%%%%%%%%%%%%%%%%%%%%%%%%%%%%%%%
%%%%%%%%%%%%%%%%%%%%%%%%%%%%%%%%%%%%%%%%%%%%%%%%%%
\section{Hilbert bimodules and {\rm II}$_1$-subfactors}\label{sec:InfiniteBackground}

We now have a brief interlude to introduce the background necessary to construct our {\rm II}$_1$-factor bimodule and subfactor representations of the annular and rectangular GICAR categories.

%%%%%%%%%%%%%%%%%%%%%%%%%%%%%%%%%%%%%%%%%%%%%%%%%%
\subsection{Hilbert bimodules}\label{sec:BimoduleBackground}

We refer to \cite[Section 2]{MR3040370} for the background on Hilbert bimodules. We rapidly introduce our notation and conventions.

\begin{nota}\mbox{}
\begin{itemize}
\item
$A$ is a {\rm II}$_1$-factor.
\item
$H$ is an $A-A$ Hilbert bimodule.
\item
$D({\sb{A}H})$ is the set of left $A$-bounded vectors.
%%%%%%%% \sb{A}H
\begin{itemize}
\item
For each $\eta\in D({\sb{A}H})$ the right multiplication operator $R(\eta): L^2(A)\to H$ is the unique extension of $\widehat{a}\mapsto a\eta$.
\item
On $D({\sb{A}H})$, the $A$-valued inner product ${\sb{A}\langle} \eta_1,\eta_2\rangle = JR(\eta_1)^*R(\eta_2)J\in A$ is $A$-linear on the \underline{left}.
\item
An $\sb{A}H$ basis (which exists by \cite{MR561983}) is a set of vectors $\{\alpha\}\subset D({\sb{A}H})$ such that 
$$
\sum\limits_\alpha R(\alpha)R(\alpha)^* = 1_H \Longleftrightarrow \sum_\alpha {\sb{A}\langle}\eta ,\alpha\rangle\alpha = \eta\text{ for all }\eta\in D({\sb{A}H}).
$$
\item
The canonical normal, faithful, semifinite (n.f.s.) trace on $A'\cap B(H)$ is given by $\Tr_{A'\cap B(H)}(x)= \sum_\alpha \langle x \alpha,\alpha\rangle$ where $\{\alpha\}$ is any $\sb{A}H$ basis.
\end{itemize}
\item
$D(H_A)$ is the set of right $A$-bounded vectors.
\begin{itemize}
%%%%%%%% H_A 
\item
For each $\xi\in D(H_A)$ the left multiplication operator $L(\xi): L^2(A)\to H$ is the unique extension of $\widehat{a}\mapsto \xi a$.
\item
On $D(H_A)$, the $A$-valued inner product $\langle \xi_1|\xi_2\rangle_A = L(\xi_1)^*L(\xi_2)\in A$ is $A$-linear on the \underline{right}.
\item
An $H_A$ basis (which exists by \cite{MR561983}) is a set of vectors $\{\beta\}\subset D(H_A)$ such that 
$$
\sum\limits_\beta L(\beta)L(\beta)^* = 1_H\Longleftrightarrow \sum_\beta \beta \langle \beta|\xi\rangle_A = \xi \text{ for all } \xi\in D(H_A).
$$
\item
The canonical n.f.s. trace on on $(A\op)'\cap B(H)$ is given by  $\Tr_{(A\op)'\cap B(H)}(x)= \sum_\beta \langle x \beta,\beta\rangle$ where $\{\beta\}$ is any $H_A$ basis.
\end{itemize}
%%%%%%%% The rest of the definitions
\item
$H^n = \bigotimes^n_A H$, and we use the convention $H^0=L^2(A)$.
\begin{itemize}
\item 
For $\eta\in D(H^k_A)$ and $\xi\in D({\sb{A}H^n})$, we denote their relative tensor product in $H^{k+n}$ by $\eta\otimes \xi$.
\item
For each $\eta\in D({\sb{A}H})$, the right creation operator $R_{\eta} : H^n \to H^{n+1}$ is the unique extension of $\zeta\mapsto \zeta\otimes \eta$ for $\zeta\in D(H^n_A)$. 
Its adjoint is given by $R_{\eta}^* (\zeta\otimes \xi)=\zeta {\sb{A}\langle } \xi,\eta\rangle$ for $\zeta,\xi$ appropriate bounded vectors.
\item
For each $\xi_1\in D(H_A)$, the right creation operator $L_{\xi_1} : H^n \to H^{n+1}$ is the unique extension of $\xi_2\mapsto \xi_1\otimes \xi_2$ for $\xi_w\in D({\sb{A}H^n})$.
 Its adjoint is given by $L_{\xi}^* (\eta\otimes \zeta)=\langle \xi| \eta\rangle_A \zeta$ for $\eta,\zeta$ appropriate bounded vectors.
\item
If $x\in (A\op)'\cap B(H^k)$ and $y\in A'\cap B(H^n)$, the operator $x\otimes_A y \in B(H^{k+n})$ given by the unique extension of $\xi\otimes \eta\mapsto (x\eta)\otimes(y\xi)$ where $\eta\in D(H^k_A)$ and $\xi\in D({\sb{A}H^n})$ is well-defined and bounded, and $\|x\otimes_A y\|_\I\leq \|x\|_\I\|y\|_\I$.
\end{itemize}
\item
$B^n=D(\sb{A}H^n)\cap D(H^n_A)$, the simultaneously left and right-bounded vectors, which are dense in $H^n$ \cite[Lemma 1.2.2]{correspondences}. 
Also, we use the convention that $B=B^1$, and we note $B^0=A$.

{\textbf{Note:}} In the case that $H$ is obtained from a {\rm II}$_1$-superfactor $A_1$ of $A=A_0$, $B$ will have a different meaning. 
However, all statements we make about $B$ will still hold for either definition of $B$.
See Remark \ref{rem:NewDefinitionForB} in Subsection \ref{sec:SubfactorBackground}.

\item
Fix $\{\alpha\}\subset B$ an $\sb{A}H$ basis (see Lemma \ref{lem:BasesInB} below). 
We have 
$$\{\alpha^n\}=\set{\alpha_{1}\otimes\cdots\otimes\alpha_{n} }{\alpha_{i}\in\{\alpha\}\text{ for all } i=1,\dots,n}\subset B^n$$ 
is the corresponding $\sb{A}H^n$ basis, since $R_{\alpha_{1}\otimes\cdots\otimes \alpha_{n}} =R_{\alpha_{1}}\cdots R_{\alpha_{n}}$. 
Similarly, we let $\{\beta\}\subset B$ be an $H_A$ basis, and we have the corresponding $H_A^n$ basis $\{{\beta}^n\}\subset B^n$.
\item
$C_n = (A^{\OP})'\cap B(H^n)$, the commutant of the right $A$-action on $H^n$, which has a canonical trace $\Tr_n=\sum_{{\beta}^n} \langle \, \cdot \, {\beta}^n,{\beta}^n\rangle$.
\begin{itemize}
\item
The inclusion $C_n\hookrightarrow C_{n+1}$ is given by $x\mapsto x\otimes_A \id_H$.
\item
The unique trace preserving n.f.s. operator valued weight\\ 
$T_{n+1} : (C_{n+1}^+,\Tr_{n+1})\to (\widehat{C_{n}^+},\Tr_{n})$ is given by $x\mapsto \sum_\beta R_\beta^* x R_\beta$.
\end{itemize}
\item
$C_n^{\OP} = A'\cap B(H^n)$, which has a canonical trace $\Tr_n^{\OP}=\sum_{{\alpha}^n} \langle \, \cdot \, {\alpha}^n,{\alpha}^n\rangle$.
\begin{itemize}
\item
The inclusion $C_n\op\hookrightarrow C_{n+1}\op$ is given by $y\mapsto \id_H\otimes_A y$.
\item
The unique trace preserving n.f.s. operator valued weight\\
$T_{n+1}\op: ((C_{n+1}\op)^+,\Tr_{n+1}\op)\to (\widehat{(C_{n}\op)^+},\Tr_{n}\op)$ is given by $y\mapsto \sum_\alpha L_\alpha^* y L_\alpha$.
\end{itemize}
\item 
The \underline{standard invariant} of $H$ is the sequences of centralizer algebras $\cQ_n=C_n\cap C_n\op$ and central $L^2$-vectors $\cP_n = A'\cap H^n=\set{\zeta\in H^n}{a\zeta=\zeta a \text{ for all }a\in A}$. 
\begin{itemize}
\item 
The planar calculus of \cite{MR3040370} acts on the $\cQ_n$ and $\cP_n$. 
Note that the $\cQ_n$ naturally act on the $\cP_n$.
\item
Note that $\cP_0=A'\cap L^2(A) = \C\widehat{1}$.
\end{itemize}
\end{itemize}
\end{nota}

The next lemma was used without proof in \cite{MR3040370}; due to its importance, we provide a proof for the convenience of the reader.
The proof uses ideas similar to \cite[Lemma 1.2.2]{correspondences} and \cite[Proposition 3.2.19]{1111.1362}.
The latter proves this result in the subfactor case (but $B$ has a different meaning - see Remark \ref{rem:NewDefinitionForB}).

\begin{lem}\label{lem:BasesInB}
There exist $\sb{A}H$ and $H_A$-bases which are subsets of $B$.
\end{lem}
\begin{proof}
We show there is a $H_A$-basis $\{\gamma\}\subset B$, and the other case is similar.

First, let  $\{\beta_i\}$ be an orthogonal $H_A$-basis, so the $L(\beta_i)L(\beta_i)^*$ are projections which sum to $1_H$, and $L(\beta_i)^*L(\beta_j)=\langle \beta_i|\beta_j\rangle_A=0$ if $i\neq j$.
This exists by \cite[Proposition 2.2]{MR1387518}.
Moreover, $C_1=\set{L(\eta)L(\xi)^*}{\eta,\xi\in D(H_A)}''$ by \cite{MR561983}.

For each $i$, $T_1(L(\beta_i)L(\beta_i)^*)$ has finite trace, and  thus has a spectral resolution
$$
T_1(L(\beta_i)L(\beta_i)^*)=\int_0^\infty \lambda \, de_\lambda^i.
$$
Mimicking \cite[Proposition 3.2.19]{1111.1362}, set $\gamma_{i,1}=e_1^i \beta$ and $\gamma_{i,n} = (e_{n}^i-e_{n-1}^i)\beta$ for $n\geq 2$.
Thus $T_1(L(\gamma_{i,n})L(\gamma_{i,n})^*)$ has norm at most $n$, and $L(\beta_i)L^2(A) = \bigoplus_{n\in \bbN} L(\gamma_{i,n})L^2(A)$. 
Moreover, the $\gamma_{i,n}$ are orthogonal, since $L(\gamma_{i,m})^*L(\gamma_{j,n})=L(\beta_i)^*e_me_nL(\beta_j)$, which is zero unless $i=j$ and $m=n$.
Thus $\{\gamma_{i,n}\}$ is an $H_A$-basis.

Finally, to show each $\gamma_{i,n}\in D({\sb{A}H})$, for each $a\in A$, we have
\begin{align*}
\|a\gamma_{i,n}\|_H^2 
&= 
\| aL(\gamma_{i,n}) \widehat{1}\|_H^2 
= 
\tr_A(L(\gamma_{i,n})^*a^*aL(\gamma_{i,n}))
=
\Tr_{1}(aL(\gamma_{i,n})L(\gamma_{i,n})^*a^*)\\
&=
\tr_A(a T_1(L(\gamma_{i,n})L(\gamma_{i,n})^*)a^*)
\leq
n\tr_A(aa^*)
=
n\|\widehat{a}\|_{L^2(A)}^2.
\qedhere
\end{align*}
\end{proof}

\begin{defn}\label{defn:symmetric}
$H$ is called \underline{symmetric} if there is a conjugate-linear isomorphism $J: H\to H$ such that $J(a\xi b)= b^* (J\xi) a^*$ for all $a,b\in A$ and $\xi\in H$ and $J^2=\id_H$. 
\end{defn}

\begin{rem}\label{rem:symmetric}
If $H$ is symmetric, then for $n\geq 1$, $H^n$ is symmetric with conjugate-linear isomorphism $J_n: H^n\to H^n$ given by the extension of
$$
J_n( \xi_1\otimes \cdots \otimes \xi_n)=(J\xi_1)\otimes\cdots \otimes (J\xi_n).
$$
for $\xi_i\in B$ for all $i$. Note that $J_nAJ_n=A^{\OP}$, $J_n C_nJ_n=C_n^{\OP}$, and $J_n B^n=B^n$. On $B(H^n)$, we define $j_n$ by $j_n(x)=J_nx^*J_n$. Note that $j_n^2=\id$ and $\Tr_n=\Tr_n^{\OP}\circ j_n$. 

If $H$ is not symmetric, then in general, $C_n^{\OP}$ is not the opposite algebra of $C_n$, e.g. $\sb{R\otimes 1} L^2(R\otimes R)_{R\otimes R}$ where $R$ is the hyperfinite {\rm II}$_1$-factor.
\end{rem}

%%%%%%%%%%%%%%%%%%%%%%%%%%%%%%%%%%%%%%%%%%%%%%%%%%
\subsection{Arbitrary index {\rm II}$_1$-subfactors}\label{sec:SubfactorBackground}

Suppose $A=A_0$ is contained in a {\rm II}$_1$-factor $A_1$.
The $A-A$ bimodule $H=L^2(A_1)$ is the motivating example for this paper.

\begin{rem}\label{rem:NewDefinitionForB}
In the case that $H=L^2(A_1)$, we no longer use the notation $B$ for $D({\sb{A}L^2(A_1)})\cap D(L^2(A_1)_A)$.
Rather, we set $B=A_1$, which is obviously still dense in $L^2(A_1)$.
Note that the bi-bounded vectors $D({\sb{A}L^2(A_1)})\cap D(L^2(A_1)_A)$ do not agree with the image of $B$ in $L^2(B)$ when $[B:A]=\I$.
However, all statements we made concerning the $B$ in the previous subsection still apply for $B=A_1$.
\end{rem}

We can perform the Jones basic construction to get another type {\rm II} factor $\langle B, e_A\rangle$, which is type {\rm II}$_1$ if and only if $[B: A]<\I$ \cite{MR696688}. 
When $[B:A]<\I$, we can form the Jones tower, and the higher relative commutants form a planar algebra \cite{math/9909027}, which always includes a Temperley-Lieb planar subalgebra.

When $[B:A]=\I$, $\langle B, e_A\rangle$ is a {\rm II}$_\I$-factor, and we must be more careful. 
Detailed analysis of this situation was started in \cite{MR1387518,1111.1362}, 
and planar structure was given for the centralizer algebras and central $L^2$-vectors in \cite{MR3040370}.
We rapidly recall the necessary background from \cite{MR1387518,1111.1362}.

\begin{defn}
Suppose $M$ is a semifinite von Neumann algebra with n.f.s. trace $\Tr_M$. Define
\begin{align*}
\n_{\Tr_M}&=\set{x\in M}{\Tr_M(x^*x)<\I}\text{ and}\\
\m_{\Tr_M}&= \n_{\Tr_M}^*\n_{\Tr_M} = \spann\set{x^* y}{x,y\in \n_{\Tr_M}}.
\end{align*}
Suppose $N$ is a von Neumann subalgebra of $M$ with n.f.s. trace $\Tr_N$. Then by \cite{MR534673}, there is a unique trace-preserving n.f.s. operator valued weight $T: M^+\to \widehat{N^+}$, the extended positive cone of $N$. We define
\begin{align*}
\n_{T}&=\set{x\in M}{T(x^*x)\in N^+}\text{ and}\\
\m_{T}&= \n_{T}^*\n_{T} = \spann\set{x^* y}{x,y\in \n_{T}}.
\end{align*}
\end{defn}

Suppose $A_0=A\subset B=A_1$ is an infinite index {\rm II}$_1$-subfactor.
First, recall that $A_0,A_1$ have normal, faithful, finite normalized traces $\tr_0,\tr_1$ respectively, and $\tr_1|_{A_0}=\tr_0$. 
We have $T_1=E_1 :A_1 \to A_0$ is the normal, faithful conditional expectation, which is implemented by the Jones projection $e_1\in A_0'\cap B(L^2(A_1,\tr_1))$.

The basic construction of $A_0\subset A_1$ is $A_2$, a type {\rm II}$_\I$-factor, and the canonical n.f.s. trace $\Tr_2$ on $A_2$ satisfies 
$$
\Tr_2(xe_1 y^*)=\Tr_2(L(\widehat{x})L(\widehat{y})^*)=\tr_1(xy^*)
$$ 
for all $x,y\in A_1$. (Note that in the notation of the previous subsection, $A_2=C_1$, and $\Tr_2=\Tr_1$.)
We form the $L^2$-space in the usual way as the closure of $\n_{\Tr_2}$.

Now the unique trace-preserving operator valued weight $T_2: A_2^+\to \widehat{A_1^+}$ is not a conditional expectation. For $x\in \n_{T_2}\subset A_2$, we define left creation operators $\Lambda_{T_2}(x) : L^2(A_1)\to L^2(A_2)$ which commute with the right $A_1$-action by $\widehat{y}\mapsto \widehat{xy}$ for $y\in \n_{\tr_1}=A_1$ which is well-defined and bounded since $xy\in \n_{T_2}\cap \n_{\Tr_2}$:
$$
\|xy\|_2^2=\Tr_2((yy^*)^{1/2} x^*x(yy^*)^{1/2})=\tr_1((yy^*)^{1/2} T(x^*x)(yy^*)^{1/2})\leq \|T(x^*x)\|_\I \|y\|_2^2.
$$
Moreover, the maps $\Lambda_{T_2}(x)$  satisfy
\begin{itemize}
\item $\Lambda_{T_2}(x)^*\widehat{y}=\widehat{T_2(x^*y)}$ for all $x\in\n_{T_2}$ and $y\in \n_{\Tr_2}$, and
\item $\Lambda_{T_2}(x)^*\Lambda_{T_2}(y) = E_1(x^*y)$ for all $x,y\in\n_{T_2}$.
%\item $T_2(\Lambda_{T_2}(x)\Lambda_{T_2}(y)^*)=xy^*$ and $\Tr_{2}(\Lambda_{T_2}(x)\Lambda_{T_2}(y)^*)=\tr_{1}(xy^*)$ for $x,y\in \n_{T_2}$.
\end{itemize}

To iterate the basic construction, note that the modular conjugation $J_2$ on $\n_{\Tr_2}$ extends to an anti-linear unitary. Hence we may define the basic construction by $A_3=J_2 (A_1'\cap B(L^2(A_2,\Tr_2)))J_2$.
It follows from \cite{MR1387518} that
$$
A_3= \set{\Lambda_{T_2}(x)\Lambda_{T_2}(y)^*}{x,y\in\n_{T_{2}}}''.
$$
The canonical n.f.s. trace $\Tr_3$ on $A_3$ is given by 
$$
\Tr_3(\Lambda_{T_2}(x)\Lambda_{T_2}(y)^*)=\tr_1(\Lambda_{T_2}(y)^*\Lambda_{T_2}(x))=\tr_1(y^*x),
$$
and the unique trace-preserving n.f.s. operator valued weight satisfies
$$
T_3(\Lambda_{T_2}(x)\Lambda_{T_2}(x)^*)=xx^*.
$$
By \cite[Section 3.2.1]{1111.1362}, $\Tr_3|_{A_2^+}=\Tr_2$, so $T_3=E_3$ is a conditional expectation, which is implemented by a Jones projection $e_3\in A_2'\cap B(L^2(A_3,\Tr_3))$.

One continues this process as in \cite{MR1387518,1111.1362} to get a tower of type {\rm II} factors $(A_n,\Tr_n)_{n\geq 0}$ together with 
\begin{itemize}
\item
the conjugate-linear unitary $J_n$ on $L^2(A_n,\Tr_n)$ extending the adjoint on $\n_{\Tr_n}$,
\item
the basic construction 
$$
A_{n+1}
=J_n (A_{n-1}'\cap B(L^2(A_n,\Tr_n))J_n
= \set{\Lambda_{T_2}(x)\Lambda_{T_2}(y)^*}{x,y\in\n_{T_{n}}}'',
$$
\item
the operator valued weights $T_{n+1}: A_{n+1}^+\to \widehat{A_{n}^+}$,
\item
the left creation operators $\Lambda_{T_n}(x): L^2(A_{n-1},\Tr_{n-1})\to L^2(A_n,\Tr_n)$ for $x\in \n_{T_n}$ which commute with the right $A_{n-1}$ action,
\item
the n.f.s. traces $\Tr_{n+1}$ satisfying
\begin{align*}
%\Tr_{n+1}(L(\xi)L(\xi)^*)&=\Tr_{n-1}(L(\xi)^*L(\xi))=\|\xi\|^2\text{ for all }\xi\in D(L^2(A_n))_{A_{n-1}},\text{ or}\\
\Tr_{n+1}(\Lambda_{T_n}(x)\Lambda_{T_n}(y)^*)&=\Tr_{n-1}(\Lambda_{T_n}(y)^*\Lambda_{T_n}(x))=\Tr_{n-1}(y^*x)\text{ for all }x,y\in\n_{T_n},
\end{align*}
\end{itemize}

We have the following facts due to \cite{MR1387518,1111.1362}.

\begin{facts}
\mbox{}
\be
\item
$L^2(A_n)\cong \bigotimes_A^n L^2(B)$ via isomorphisms $\theta_n$ (see Subsection \ref{sec:RealizeOddJones}).
\item 
Writing $B^n=\bigotimes^n_A\widehat{B}$ (again, this notation differs from Subsection \ref{sec:BimoduleBackground}), $B^n$ is dense in $\bigotimes_A^n L^2(B)$.
\item (Multistep basic construction \cite{MR1387518}) For $0\leq k\leq n$, the inclusions $A_{n-k}\subseteq A_n\subseteq A_{n+k}$ are standard, i.e.  
$$
(\id_k\otimes_A J_{n-k} A_{n-k}J_{n-k})'=J_n(A_{n-k}\otimes_A \id_k)'J_n \cong A_{n+k}.
$$
\item (Shifts \cite{MR1387518})For $0\leq k\leq n$, we let $j_n(x)=J_n x^* J_n$ for $x\in B(L^2(A_n,\Tr_n))$. Then $j_n$ is an anti-isomorphism of $A_{n-k}$ onto $A_{n+k}'$. Hence if we compose $j_n$ and $j_{n+1}$, we get an isomorphism:
$$
j_{n+1}j_n (A_{n-k}'\cap A_n)\underset{anti}{\cong} j_{n+1} (A_{n}'\cap A_{n+k} )\underset{anti}{\cong} A_{n-k+2}'\cap A_{n+2}.
$$
\item 
(Odd Jones projections \cite[Section 3.2.1]{1111.1362}) For all $n\in\N$, $\Tr_{2n}|_{A_{2n-1}^+}=\I$ and $\Tr_{2n+1}|_{A_{2n}^+}=\Tr_{2n}$. Therefore, $T_{2n+1}: A_{2n+1}\to A_{2n}$ is a conditional expectation, which gives rise to the odd Jones projection $e_{2n+1}$. 

When we realize $A_{2n}$ acting on $L^2(A_{n},\Tr_{n})$ from the multistep basic construction, 
$e_{2n-1}=J_{n} e_1 J_{n}$.
\ee
\end{facts}

\begin{rem}
For $x_1,\dots x_n\in B$, we write $\widehat{x_1}\otimes\cdots \otimes \widehat{x_n} \in B^n$ omitting the subscript $A$ to distinguish between operators and vectors, such as
$x\otimes_A \id_1$ and $\widehat{x}\otimes \widehat{1}$ for $x\in B$. One is left multiplication by $x\in B$ on $L^2(B)\otimes_A L^2(B)$, and the other is $\theta_2^{-1}(\widehat{xe_1})$.
\end{rem}

%%%%%%%%%%%%%%%%%%%%%%%%%%%%%%%%%%%%%%%%%%%%%%%%%%
\subsection{Identifying the Jones projections}\label{sec:RealizeOddJones}

We now identify the Jones projections acting on $\bigotimes_A^n L^2(B)$ via $\theta_n$.
We recall Burns' definition of the isomorphisms $\theta_n : \bigotimes^n_A L^2(B) \to L^2(A_n)$ \cite[Section 3.2.2]{1111.1362}.
Note that our numbering differs from Burns' numbering in that we start with $A_0\subset A_1$.
Also, Burns' definition is more general in that he works with arbitrary type {\rm II} factors, and he defines a more general set of isomorphisms.
We provide a simplified definition for the reader's convenience.

\begin{defn}
The isomorphisms $\theta_n : \bigotimes^n_A L^2(B) \to L^2(A_n)$ are composites of other known isomorphisms. We define:
\begin{itemize}
\item
$v_{k+1} : L^2(A_{k})\underset{A_{k-1}}{\otimes} L^2(A_{k})\to L^2(A_{k+1})$ by $\eta\otimes J_{k}\xi \mapsto L(\eta)L(\xi)^*$ for $\eta,\xi\in D(L^2(A_k)_{A_{k-1}})$,
\item
$\iota_{k} : L^2(A_{k})\underset{A_k}{\otimes}L^2(A_{k})\to L^2(A_{k})$ by $\widehat{x}\otimes \widehat{y} \mapsto \widehat{xy}$ for $x,y\in\n_{\Tr_n}$,
\item
$\D
\psi_{k,n}=\left(\bigotimes_{A_k}^{n-1} v_{k+1} \right) \circ \left(\id_k \underset{A_{k-1}}{\otimes} \left(\bigotimes_{A_{k-1}}^{n-2} \iota_k^* \right) \underset{A_{k-1}}{\otimes} \id_k\right) 
$
which doubles, regroups, and contracts, i.e., $\psi_{k,n}$ is the composite map
$$
\hspace{-1.2cm}
\bigotimes_{A_{k-1}}^{n} L^2(A_k) 
\to
\bigotimes_{A_{k-1}}^{n} \left(L^2(A_k)\underset{A_k}{\otimes} L^2(A_k)\right) 
\cong
\bigotimes_{A_{k}}^{n-1} \left(L^2(A_k)\underset{A_{k-1}}{\otimes} L^2(A_k)\right) 
\to 
\bigotimes_{A_k}^{n-1} L^2(A_{k+1}),
$$
and
\item
for $n\geq 2$, 
$\theta_n : \bigotimes^n_{A} L^2(B) \to L^2(A_{n})$ by $\psi_{n-1,2}\circ \psi_{n-3,3}\circ\cdots \circ \psi_{2,n-1}\circ\psi_{1,n}$.
\end{itemize}
Note that $\theta_n$ is compatible with $J : \bigotimes^n_A L^2(B)\to \bigotimes^n_A L^2(B)$ by $\xi_1\otimes \cdots\otimes \xi_n \mapsto J_B\xi_n \otimes\cdots \otimes J_B\xi_1$ since the $v_{k+1}$, $\iota_k$, and $\psi_{k,j}$ are also.
\end{defn}

\begin{lem}\label{lem:IdentifyOddJonesProjections}
When we use $\theta_n$ to transport the action of $A_{2n}$ to $\bigotimes^n_{A} L^2(B)$, the Jones projection $e_1$ maps to $e_1\otimes_A \id_{n-1}$ and the Jones projection $e_{2n-1}$ maps to $\id_{n-1} \otimes_A e_1$.
\end{lem}
\begin{proof}
The result follows from \cite[Lemma 3.3.20]{1111.1362} and the compatibility of $\theta_n$ and $J$. 
\end{proof}

\begin{prop}\label{prop:IdentifyOddJonesProjections}
When we use $\theta_n$ to transport the action of $A_{2n}$ to $\bigotimes^n_{A} L^2(B)$, then for $1\leq i\leq n$, we may identify the Jones projection $e_{2i-1}$ with $\id_{i-1} \otimes_A e_1\otimes_A \id_{n-i}$ .
\end{prop}
\begin{proof}
We use strong induction on $n$. The case $n=1$ is trivial. Suppose the result holds for all $0\leq i \leq n$. Now use $\theta_{n+1}$ to realize the action of $A_{2n+2}$ on $\bigotimes^{n+1}_A L^2(B)$. 
In the notation of Subsection \ref{sec:BimoduleBackground}, the inclusion $A_{2n}\hookrightarrow A_{2n+2}$ transports to the inclusion $C_n\hookrightarrow C_{n+1}$ which is given by $x\mapsto x\otimes_A \id_{1}$, so the result is true for all $1\leq i < n+1$ by the associativity of the relative tensor product of $A-A$ bilinear operators. The result for $i=n+1$ now follows by Lemma \ref{lem:IdentifyOddJonesProjections}. 
\end{proof}

%%%%%%%%%%%%%%%%%%%%%%%%%%%%%%%%%%%%%%%%%%%%%%%%%%%%%%%%%%%%
%%%%%%%%%%%%%%%%%%%%%%%%%%%%%%%%%%%%%%%%%%%%%%%%%%%%%%%%%%%%
%%%%%%%%%%%%%%%%%%%%%%%%%%%%%%%%%%%%%%%%%%%%%%%%%%%%%%%%%%%%
\section{Representations via subfactors and bimodules}\label{sec:BimoduleRepresentations}

The rectangular GICAR category $\sR\sG$ acts on tensor powers of a {\rm II}$_1$-factor bimodule which contains a copy of the trivial bimodule.
The action of the tensor category $\sR\sG$ is compatible with the tensor structure of the tensor category of bimodules.
One can imagine that the annular GICAR category $\sA\sG$ is obtained from $\sR\sG$ by tensoring the morphisms with themselves around the outside, i.e., gluing the rectangles into annuli (the opposite of Figure \ref{fig:CuttingAndGluing}).
This no longer leaves us with a tensor category, and thus $\sA\sG$ must act on the spaces obtained from the bimodules by tensoring themselves on the outside, i.e., the invariant vectors of the bimodules.

%%%%%%%%%%%%%%%%%%%%%%%%%%%%%%%%%%%%%%%%%%%%%%%%%%%%%%%%%%%%
\subsection{Rectangular GICAR representations}\label{sec:RectangularGICARReps}

Let $A$ be a {\rm II}$_1$-factor and let $H$ be a Hilbert $A-A$ bimodule.
We assume the following.

\begin{assume}\label{assume:ContainsTrivial}
Suppose $H$ is not the trivial bimodule, but $H$ contains a distinguished copy of the trivial bimodule, i.e., $H\cong L^2(A)\oplus K$ for a non-zero $A-A$ Hilbert bimodule $K$.
\end{assume}

\begin{rem}\label{rem:ACentral}
Containing a distinguished copy of the trivial bimodule is equivalent to the existence of a distinguished central $L^2$-vector $\zeta\in \cP_1$ with $\langle \zeta|\zeta\rangle_A = 1_A$.
Note that all central $L^2$-vectors are automatically $A-A$ bounded. 
See \cite[Sections 3.3-3.4]{MR3040370} for more details.
For all $A-A$ bounded $\kappa\in K \neq (0)$, we have $\langle \kappa |\zeta\rangle_A = 0$.
\end{rem}

Below is the main theorem of this subsection, which is implied by Theorem \ref{thm:GICARinRH}.

\begin{thm}\label{thm:GICARinQ}
There is a faithful $*,\otimes$-representation of the rectangular GICAR category $\sR\sG$ as $A-A$ bimodule maps between the $H^n$, which is independent of the left and right von Neumann dimension of $H$.
\end{thm}

Of particular importance is the following corollary, which tells us some basic structure of the centralizer algebras $\cQ_n$.

\begin{cor}
The $G_n$ embed faithfully in the centralizer algebras $\cQ_n$, so $\cQ_n$ is nonabelian for $n\geq 2$.
\end{cor}

%\begin{rem}
%In the case that $H=L^2(B)$ for an arbitrary index {\rm II}$_1$-subfactor $A\subset B$, $\cQ_n=A_0'\cap A_{2n}$.
%Hence the rectangular GICAR category embeds into the tensor category given by the principal even half of the subfactor $A\subset B$ (this tensor category is not necessarily rigid. See Section \ref{sec:NonRigid} \nn{}).
%In particular, we have an injection $G_n\hookrightarrow A_0'\cap A_{2n}$, which is nonabelian for $n\geq 2$.
%\end{rem}

Using the binomial theorem, it is easy to see how the algebras $G_n$ should arise as intertwiners among the $H^n$. For $n\geq 0$, let $K^n=\bigotimes_A^n K$, where as usual, $K^0=L^2(A)$. Then we have
$$
H^n\cong \left(L^2(A)\oplus K\right)^{\otimes n}\cong \bigoplus_{j=0}^n {n \choose j} K^j,
$$
and we get a canonical inclusion $G_n\hookrightarrow \End_{A-A}(H^n)$. 
If $K^j$ is irreducible and distinct for all $j\in\N$, then $G_n\cong \End_{A-A}(H^n)$ for all $n\geq 0$.

\begin{ex}\label{ex:OuterZAction}
Let $\sigma : \Z\to \Out(R)$ be an outer action, where $R$ is the hyperfinite {\rm II}$_1$-factor. 
We denote $\sigma(n)$ by $\sigma^n$.
Let $K=L^2(A)_\sigma$, where the action is given by $a\cdot \widehat{b}\cdot c = (ab\sigma(c))^{\widehat{\,\,}}$.
Recall that $K^j\cong L^2(A)_{\sigma^j}$ for all $j\geq 0$.
Hence each $K^j$ is irreducible and distinct, and 
$\End_{R-R}\left(\bigotimes^n_R(L^2(R)\oplus K)\right)\cong G_n$ for all $n\geq 0$.
\end{ex}

\begin{quests}
Is there such a $K$...
\begin{itemize}
\item 
which is symmetric?
\item
which is of the form $L^2(B)\ominus L^2(A)$ for a (necessarily infinite index) {\rm II}$_1$-subfactor $A\subset B$?
\end{itemize}
\end{quests}

With more care, we obtain a faithful representation of the entire rectangular GICAR category $\sR\sG$ as $A-A$ bimodule maps among the $H^n$'s.

\begin{rem}
Recall that $H\cong L^2(A)\oplus K$, where the copy of $L^2(A)$ corresponds to the distinguished central $L^2$-vector $\zeta\in\cP_1$.
Note that $L(\zeta): L^2(A)\to H$ can be viewed as the inclusion, and $L(\zeta)^* : H\to L^2(A)$ behaves like the Jones projection for a {\rm II}$_1$-subfactor. 
More precisely, if $\xi\in B=D({\sb{A}H})\cap D(H_A)$, then $L(\zeta)^*\xi=\langle \zeta|\xi\rangle_A$ defines an element of $A$.
\end{rem}

\begin{nota}
We write $e_A=L(\zeta)^*$ and $e_A^*=L(\zeta)$.
Note that $e_A,e_A^*$ are $A-A$ bilinear since $\zeta\in\cP_1$.
\end{nota}

\begin{defn}
Given an $A-A$ bimodule $H$, we define the rectangular bimodule category $\sR(H)$ as the following small involutive tensor category:
\itt{Objects} $H^n$ for $n\geq 0$.
\itt{Tensoring objects} 
Connes relative tensor product. Note that $H^m\otimes_A H^n \cong H^{m+n}$.
The associators are the unique extensions of the obvious associators on the subspaces of bounded vectors $B^n$.
\itt{Morphisms} 
For $1\leq i\leq n$, define the maps $\a_i^*: H^n \to H^{n-1}$ by the following commutative diagrams:
\[
\xymatrix{
H^n \ar@{<->}[rr]^(.3)\cong\ar[d]_{\a_i^*} &&H^{i-1} \otimes_A H \otimes_A H^{n-i} \ar[d]^{\id_{i-1}\otimes_A e_A \otimes_A \id_{n-i}}\\
H^{n-1}\ar@{<->}[rr]^(.3)\cong &&H^{i-1}\otimes_A L^2(A) \otimes_A H^{n-i}.
}
\]
The horizontal arrows are the associator isomorphisms.
For $1\leq i\leq n$, we define the maps $\a_j: H^n \to H^{n-1}$ similarly, but replacing $e_A$ with $e_A^*$.
The morphisms of $\sR(H)$ are $\C$-linear combinations of all composites of the $\a_i,\a_j^*$.
Note that these morphisms are all $A-A$ bimodule maps.
\itt{Composition} composition of operators.
\itt{Adjoint} adjoint of operators.
\end{defn}

We have the following explicit characterization of the maps $\a_i,\a_j^*$.

\begin{prop}\label{prop:RelativeTensorMaps}
The maps $a_i$, $a_i^*$ are given by the unique extensions of
\begin{align*}
\a_i(\xi_1\otimes\cdots \otimes \xi_n) & = \xi_1\otimes\cdots \otimes \xi_{i-1} \otimes \zeta \otimes \xi_i\otimes\cdots\otimes \xi_n\tag{creation}\\
\a_i^*(\xi_1\otimes\cdots \otimes \xi_n) & =\xi_1\otimes\cdots\otimes \xi_{i-1} \otimes e_A(\xi_i)\xi_{i+1}\otimes\cdots\otimes \xi_n \tag{annihilation}
\end{align*}
where $\xi_j\in B$ for all $j$. 
\end{prop}
\begin{proof}
Since $\zeta\in \cP_1$, the right hand side of the first formula is well-defined.
Since $\xi_j$ is $A$-bounded, $e_A(\xi_j)=\langle \zeta|\xi_j\rangle_A$ defines an element of $A$.
Since $\zeta\in\cP_1$ is $A$-central, $e_A$ is $A-A$ bilinear, and the right hand side of the second formula is well-defined.
The rest is a straightforward calculation.
\end{proof}

Compare Relations \eqref{rel:AG1}-\eqref{rel:AG2} and Proposition \ref{prop:Standard} with Proposition \ref{prop:Relations}.

\begin{prop}\label{prop:Relations}
\mbox{}
\be
\item The words on $\a_i,\a_j^*$ satisfy the following relations:
\begin{enumerate}[(i)]
\item $\a_i\a_{j-1} = \a_{j}\a_i$ and $\a_i^* \a_{j}^* = \a_{j-1}^* \a_{i}^*$ for all $i<j$,
\item $\D \a_{i}^*\a_{j}=
\begin{cases}
\a_{j+1}\a_{i}^* &\text{if }i<j\\
\id_{n} & \text{if } i=j\\
\a_j \a_{i+1}^* & \text{if } i>j 
\end{cases}$ \hspace{.2in} and
\hspace{.2in}
\item $\a_i\a_i^* =\id_{i-1}\underset{A}{\otimes}\, e_A^*e_A \underset{A}{\otimes} \id_{n-i}$ for all $i\leq n$.
\ee
\item Each word in the $a_i,a_j^*$ has a unique standard form
$$
\a_{i_k}\cdots \a_{i_1}\a_{j_1}^*\cdots \a_{j_\ell}^*
$$
where $i_1<\cdots <i_{k}$ and $j_1<\cdots j_\ell$.
\ee
\end{prop}
\begin{proof}
Straightforward from Proposition \ref{prop:RelativeTensorMaps}.
\end{proof}

Comparing $(1.iii)$ in Proposition \ref{prop:Relations} with Proposition \ref{prop:IdentifyOddJonesProjections}, we make the following definition.

\begin{defn}[Odd Jones projections]\label{def:OddJones}
For $1\leq i\leq n$, define $e_{2i-1}=\a_i\a_i^*$.
\end{defn}

\begin{cor}
For $1\leq i,j \leq n$, the $\a_i\a_j^*\in \cQ_n$ witness the von Neumann equivalence of the projections $e_{2i-1},e_{2j-1}\in \cQ_n$. 
Thus once we know $e_{2i-1}\neq e_{2j-1}$ (which follows from Corollary \ref{cor:LinearlyIndependent}), $\cQ_n$ is not abelian for $n\geq 2$.
\end{cor}
\begin{proof}
By Proposition \ref{prop:Relations},
\begin{align*}
(\a_i\a_j^*)(\a_j\a_i^*) 
&= \a_i\a_i^* = e_{i-1}\otimes_A e_1\otimes_A \id_{n-i} =e_{2i-1}\text{ and}\\
(\a_j\a_i^*)(\a_i\a_j^*) 
&= \a_j\a_j^* = e_{j-1}\otimes_A e_1\otimes_A \id_{n-j} =e_{2j-1}.
\end{align*}
\end{proof}

\begin{rem}
Suppose $H=L^2(B)$ where $A\subset B$ is a {\rm II}$_1$-subfactor.
In this case, since $H^n\cong L^2(A_n)$, we have $\cQ_n\cong A_0'\cap A_{2n}$, and the odd Jones projections in Definition \ref{def:OddJones} agree with Burns' odd Jones projections via Proposition \ref{prop:IdentifyOddJonesProjections}. 
Thus $A_0'\cap A_{2n}$ is not abelian for $n\geq 2$.
\end{rem}

\begin{lem}\label{lem:Kappa}
There is a $\kappa\in K$ with $\|\kappa\|_K=1$ such that 
\be
\item 
$\kappa\in D(H_A)$ and $\langle \kappa |\kappa\rangle_A = 1_A$, or 
\item
$\kappa\in D({\sb{A}K})$ and ${\sb{A}\langle} \kappa,\kappa\rangle = 1_A$.
\ee
Hence the simple relative tensors consisting of only $\kappa$'s and $\zeta$'s are well-defined vectors in $H^n$.
\end{lem}
\begin{proof}
Since $K$ is a non-zero $A-A$ bimodule, $\dim_{A-}({\sb{A}K})\dim_{-A}(K_A)\geq 1$.
Suppose $\dim_{-A}(K_A)\geq 1$, and choose a submodule $M_A\subset K_A$ such that $\dim_{-A}(M_A)=1$.
Then there is a spatial isomorphism $\phi : L^2(A)\to M$ which intertwines the right $A$-actions. Let $\kappa=\phi(\widehat{1})$. 
Then $\kappa\in D(M_A)\subset D(K_A)$, $\phi=L(\kappa)$, and $L(\kappa)^*L(\kappa)=\langle \kappa|\kappa\rangle_A=1_A$.
If $\dim_{A-}({\sb{A}K})\geq 1$, then a similar argument finds a $\kappa$ such that ${\sb{A}\langle} \kappa,\kappa\rangle = 1_A$.

The last assertion follows from the fact that for $n\geq 2$, $H^n$ is the completion of the algebraic tensor product $D(H_A)\odot_A H^{n-1}$ with the inner product $\langle \eta_1 \odot \xi_1 , \eta_2\odot \xi_2\rangle = \langle \langle \eta_2|\eta_1\rangle_A \xi_1, \xi_2\rangle_{H^{n-1}}$, and similarly for left modules.
\end{proof}

\begin{prop}\label{prop:injective}
Suppose
$$
x=\a_{i_k}\cdots \a_{i_{1}}\a_{j_{1}}^*\cdots \a_{j_\ell}^*\in \sR(H)(n,n-\ell+k)
$$
is in the standard form of Proposition \ref{prop:Relations}. Then there are $\xi\in B^n$ and $\eta\in B^{n-\ell+k}$ such that $\langle x\xi,\eta\rangle = 1$ and $\langle y\xi,\eta\rangle=0$ for all words $y\in\sR(H)(n,n-\ell+k)$ on the $\a_i,\a_j^*$ whose standard form has length at least $\ell+k$. 
\end{prop}
\begin{proof}
Choose $\kappa$ as in Lemma \ref{lem:Kappa}.
Let
\begin{itemize}
\item $\xi\in B^n$ be the simple relative tensor with $\zeta$'s in positions $j_1<\dots<j_\ell$ and $\kappa$'s in the other positions, and
\item $\eta\in B^{n-\ell+k}$ be the simple realtive tensor with $\zeta$'s in positions $i_1<\cdots <i_k$ and $\kappa$'s in the other positions.
\end{itemize}
Then by Lemma \ref{lem:Kappa},
$$
\langle x\xi,\eta\rangle=\|\underbrace{\kappa \otimes\cdots \otimes\kappa}_{n-\ell\text{ vectors}} \|_{H^{n-\ell}}^2=1.
$$ 
Suppose $y\in \sR(H)(n,n-\ell+k)$ is a word on the $\a_i,\a_j^*$ with $\langle y\xi,\eta\rangle\neq 0$, and write $y$ in standard form
$$
y= \a_{i'_{k'}}\cdots \a_{i'_{1}}\a_{j'_{1}}^*\cdots \a_{j'_{\ell'}}^*.
$$
Since $e_A(\kappa)=0$, we must have $i'_1,\cdots i'_{k'}\in\{i_1,\dots,i_k\}$ and $j'_1,\cdots j'_{\ell'}\in \{j_1,\dots, j_\ell\}$, so $k'\leq k$ and $\ell'\leq \ell$. Moreover, if $k'=k$ and $\ell'=\ell$, then $y=x$.
\end{proof}

\begin{cor}\label{cor:LinearlyIndependent}
The words on $\a_i,\a_j^*$ in standard form in $\sR(H)(m,n)$ are a basis.
\end{cor}
\begin{proof}
We already know such words span by Proposition \ref{prop:Relations}.
Suppose 
$$
0=\sum_{i=1}^k \lambda_i w_i\in \sR(H)(m,n)
$$ 
where $w_i\in\sR(H)(m,n)$ are distinct words on the $\a_i,\a_j^*$ in standard form, ordered by increasing word length. 
We show by induction on $k$ that all the $\lambda_i$'s are zero. 
If $k=1$, this is trivial, since $w\neq 0$ for all words $w$ by Proposition \ref{prop:injective} (there is a linear functional which separates $w$ from $0$). 
Suppose now that $k>1$. Since the standard form word length of $w_1$ is minimal, by Proposition \ref{prop:injective}, there are $\xi\in B^m$ and $\eta\in B^n$ such that $\langle w_i \xi,\eta\rangle = \delta_{1,i}$. 
This means 
$$
\lambda_1 =\sum_{i=1}^k \lambda_i \langle w_i \xi,\eta\rangle = \left \langle \sum_{i=1}^k \lambda_i w_i \xi,\eta\right\rangle =0.
$$
We are finished by the induction hypothesis.
\end{proof}

\begin{thm}\label{thm:GICARinRH}
The $*,\otimes$-functor $\Phi: \sR\sG \to \sR(H)$ given by $[n]\mapsto H^n$ for $n\geq 0$ and 
$$
\sR\sG(n,n+1)\ni\alpha_i\longmapsto \a_i\in \sR(H)(n,n+1) \text{ for }1\leq i\leq n+1
$$
defines an equivalence of involutive tensor categories.
\end{thm}
\begin{proof}
By Proposition \ref{prop:Relations}, the relations of $\sR\sG$ are satisfied in $\sR(H)$, so $\Phi$ is well-defined.
By definition $\Phi$ preserves the adjoint, and it is easy to check that $\Phi$ preserves $\otimes$.
Since the words on $\alpha_i,\alpha_j^*$ in $\sR\sG(m,n)$ in standard form are a basis for $\sR\sG(m,n)$ by Proposition \ref{prop:Standard}, Corollary \ref{cor:LinearlyIndependent} shows that $\Phi$ is fully faithful and essentially surjective.
\end{proof}

\begin{rem}
The involutive tensor category $\sR\sG\cong\sR\sP$ is positive, i.e., if $x\in\sR\sG(m,n)$ and $x^*x=0\in\sR\sG(m,m)$, then $x=0$.
This can be shown using the standard form in Proposition \ref{prop:Standard}, or using positivity of $\sR(H)$ which comes for free.
If $x\in \sR\sG(m,n)$ with $x^*x=0$, then $\Phi(x^*x)=0$, so $\Phi(x)=0$ as $\sR(H)$ is positive. Hence $x=0$ as $\Phi$ is injective on hom spaces.
\end{rem}

\begin{rem}\label{rem:PlanarCalculus}
The planar calculus of \cite{MR3040370} is compatible with diagrams in $\sR\sP$. The value of a free-floating strand is
$$
\confetti = 
\begin{cases}
\cwBrokenLoop &= \Tr_1(e_1) = \dim_{-A}(L^2(A))=1\\
\ccwBrokenLoop &= \Tr_1\op(e_1) = \dim_{A-}(L^2(A))=1,
\end{cases}
$$
and the value of the dotted closed oriented loops are
\begin{align*}
\cwDottedLoop &= \Tr_1(1-e_1)=\dim_{-A}(K) \text{ and }
\\ 
\ccwDottedLoop &=\Tr_1\op(1-e_1)=\dim_{A-}(K).
\end{align*}
Thus if $0\leq k\leq n$ and $q\in\sR\sP_n$ is a minimal projection with exactly $j$ dotted through strings, then
\begin{align*}
\Tr_n(\Phi\circ\Psi^{-1}(q))
&=
\dim_{-A}(K)^{j} 
\\
\Tr_n\op(\Phi\circ\Psi^{-1}(q))
&=
\dim_{A-}(K)^{j} .
\end{align*}
\end{rem}

%%%%%%%%%%%%%%%%%%%%%%%%%%%%%%%%%%%%%%%%%%%%%%%%%%%%%%%%%%%%
\subsection{Annular GICAR representations}

Let $H$ be as in Assumption \ref{assume:ContainsTrivial}.
Then $\cP_n\neq (0)$ for all $n\geq 0$, since it contains the vector $\zeta\otimes \cdots \otimes \zeta$.

\begin{defn}[{\cite[Section 4]{MR3040370}}]
A Hilbert $A-A$ bimodule $H$ is called \underline{extremal} if $\Tr_1=\Tr_1\op$ on $\cQ_1$.

A \underline{Burns rotation} is an operator $\rho:\cP_n\to\cP_n$ such that for all $\zeta\in\cP_n$ and $b_1,\dots, b_n\in B$, we have
$$
\langle \rho(\zeta) , b_1\otimes\cdots \otimes b_n\rangle = \langle \zeta, b_2\otimes\cdots \otimes b_n\otimes b_1\rangle.
$$
An \underline{opposite Burns rotation} is defined similarly:
$$
\langle \rho\op(\zeta) , b_1\otimes\cdots \otimes b_n\rangle = \langle \zeta, b_n\otimes b_1\otimes\cdots \otimes b_{n-1}\rangle.
$$
Note that if such a $\rho$ exists, then it is unique, and $\rho^n=\id_{\cP_n}$. In this case, $\rho^{-1}=\rho\op$.
\end{defn}

Recall the following theorems.

\begin{thm}[{\cite[Theorem 4.7]{MR3040370}}]
The following are equivalent:
\be
\item
$H$ is extremal.
\item
$H^n$ is extremal for all $n\geq 1$.
\item
$H^n$ is extremal for some $n\geq 1$.
\ee
\end{thm}

\begin{thm}[{\cite[Theorems 4.11, 4.20, and 4.28]{MR3040370}}]
If $H$ is extremal, then the Burns rotation $\rho=\sum_\beta L_\beta R_\beta^*$ converges strongly on $\cP_n$ for all $n\geq 2$.
Moreover, $\rho^{-1}=\rho^*$ is given by the strongly convergent sum $\sum_\alpha R_\alpha L_\alpha^*$.

Conversely, if a unitary Burns rotation $\rho$ exists on $\cP_{2n}$ and $H$ is symmetric, then $H^n$ is extremal.
\end{thm}

We now impose the following assumption.

\begin{assume}\label{assume:Extremal}
Suppose  $H$ is extremal, so that the Burns rotation $\rho=\sum_\beta L_\beta R_\beta^*$ converges strongly on $\cP_n$ for all $n\geq 2$.
\end{assume}

\begin{defn}
Given an $A-A$ bimodule $H$, we define the annular bimodule category $\sA(H)$ as the following small involutive category:
\itt{Objects} $\cP_n$ for $n\geq 0$.
\itt{Morphisms} 
For $1\leq i\leq n+1$, the maps $\a_i: H^n\to H^{n+1}$ descend to maps $\cP_n \to \cP_{n+1}$ since they are $A-A$ bilinear, i.e.,
for all $x\in A$ and $\xi\in \cP_n$,
$$
x(\a_i(\xi))=\a_i(x\xi)= \a_i(\xi x)=(\a_i(\xi))x.
$$
A similar statement holds for the $\a_j^*$.
For $n=0$, let $\rho=\id_{L^2(A)}$, and for $n\geq 1$, let $\rho$ be the Burns rotation, which preserves $\cP_n$.
The morphisms of $\sA(H)$ are $\C$-linear combinations of all composites of the $\a_i,\a_j^*,\rho$.
\itt{Composition} composition of operators.
\itt{Adjoint} adjoint of operators.
\end{defn}

The main theorem of this subsection is as follows.

\begin{thm}\label{thm:GICARinP}
The $*$-functor $\Phi_{\sA}: \sA\sG \to \sA(H)$ given by $\alpha_i\mapsto \a_i$ and $\tau\mapsto \rho$ defines a $*$-representation.
\end{thm}
\begin{proof}
Note that Relations \eqref{rel:AG1} and \eqref{rel:AG2} automatically hold in $\sA(H)$ by Proposition \ref{prop:Relations}.
It remains to show Relations \eqref{rel:AG3} and \eqref{rel:AG4} hold. Since $\rho$ is periodic and unitary, Relation \eqref{rel:AG3} follows for $\Phi_{\sA}(\tau)=\rho$ immediately. Suppose $2\leq i\leq n$. Then for all $\xi\in \cP_n$ and $b_1,\dots, b_{n-1}\in B$, by the definition of the Burns rotation, we have
\begin{align*}
\langle \Phi_{\sA}(\alpha_i^*\tau)(\xi),b_1\otimes \cdots \otimes b_{n-1}\rangle
&=
\langle \Phi_{\sA}(\alpha_i)^*\Phi_{\sA}(\tau)(\xi),b_1\otimes \cdots \otimes b_{n-1}\rangle
\\
&=
\langle \rho(\xi),\a_i(b_1\otimes \cdots \otimes b_{n-1})\rangle
\\
&=
\langle \rho(\xi),b_1\otimes \cdots \otimes b_{i-1}\otimes \zeta \otimes b_i \otimes \cdots \otimes b_{n-1} \rangle
\\
&=
\langle \xi,b_2 \cdots \otimes b_{i-1}\otimes \zeta \otimes b_i \otimes \cdots \otimes b_{n-1}\otimes b_1 \rangle
\\
&=
\langle \xi,\a_{i-1}(b_2 \otimes\cdots \otimes b_{n-1}\otimes b_1) \rangle
\\
&=
\langle\Phi_{\sA}(\alpha_{i-1})^*( \xi),b_2 \otimes\cdots \otimes b_{n-1}\otimes b_1 \rangle
\\
&=
\langle\rho(\Phi_{\sA}(\alpha_{i-1}^*)(\xi)),b_1 \otimes\cdots \otimes b_{n-1}\rangle
\\
&=
\langle\Phi_{\sA}(\tau\alpha_{i-1}^*)( \xi),b_1 \otimes\cdots \otimes b_{n-1}\rangle.
\end{align*}
The relation $\alpha_i\tau=\tau\alpha_{i-1}$ is similar, and Relation \eqref{rel:AG4} holds.
\end{proof}

\begin{rem}
The representation of Theorem \ref{thm:GICARinP} is not necessarily faithful as we will see in Examples \ref{ex:SInfinity} and \ref{ex:OuterZAction2}.
\end{rem}

Just as every subfactor planar algebra decomposes as an orthogonal direct sum of irreducible annular Temperley-Lieb modules, so does the sequence of central $L^2$-vectors $(\cP_n)_{n\geq 0}$ under our $\sA\sG$-action afforded by Theorem \ref{thm:GICARinP}.

Just as the empty diagram generates an annular Temperley-Lieb module for a subfactor planar algebra \cite{MR1929335}, 
the vector $\widehat{1}\in \cP_0=A'\cap L^2(A)$ always generates an annular GICAR module.
However, this $\sA\sG$-module is trivial, since the only $\sA\sG$-consequence of $\widehat{1}\in\cP_0$ in $\cP_n$ is the $n$-fold tensor product of $\zeta$.
In stark comparison with finite index subfactors, it may be the case that $(\cP_n)_{n\geq 0}$ only consists of the trivial $\sA\sG$-module when $A_0\subset A_1$ is a non-trivial subfactor!

\begin{ex}[{\cite[Section 5]{MR3040370}}]\label{ex:SInfinity}
Consider $A_0=R\rtimes \Stab(1)\subset R\rtimes S_\I=A_1$ where $\Stab(1)$ is the stabilizer of 1 under the action of $S_\I$ on $\N$.
Then $\dim_\C(\cP_n)=1$ for all $n\geq 0$.
More precisely, for $n\geq 1$,
$$
\cP_n=A_0'\cap \bigotimes^n_{A_0} L^2(A_1)
=\spann\left\{
\underbrace{\widehat{1}\otimes \cdots \otimes \widehat{1}}_{n\text{ vectors}}
\right\}
\cong \C \,
\begin{tikzpicture}[baseline=-.1cm]
	\nbox{}{(0,0)}{.4}{.2}{.2}{};
	\BrokenTopString{(-.4,0)};
	\node at (.05,0) {$\dots$};
	\BrokenTopString{(.4,0)};	
	\node at (.4,.6) {\scriptsize{$n$}};	
\end{tikzpicture}\,.
$$
However, note that although $\dim_\C(\cQ_n)<\I$ for all $n$, the dimension grows superfactorially, which is much faster than $\dim(G_n)=\sum_{k=0}^n {n\choose k}^2$. Thus $A_0\subset A_1$ does not have trivial standard invariant, i.e., it is not the infinite index analog of the Temperley-Lieb subfactors. 
\end{ex}

\begin{ex}\label{ex:OuterZAction2}
Recall Example \ref{ex:OuterZAction}, i.e. $H=L^2(R)\oplus L^2(R)_{\sigma}$ for an outer action $\sigma:\Z\to \Out(R)$, where we denote $\sigma^n=\sigma(n)$.
In this case, when $n\geq 1$, $K^n=L^2(R)_{\sigma^n}$ has no central vectors, so
$$
\cP_n=A'\cap H^n=\spann\left\{
\underbrace{\zeta\otimes \cdots \otimes \zeta}_{n\text{ vectors}}
\right\}
\cong \C \,
\begin{tikzpicture}[baseline=-.1cm]
	\nbox{}{(0,0)}{.4}{.2}{.2}{};
	\BrokenTopString{(-.4,0)};
	\node at (.05,0) {$\dots$};
	\BrokenTopString{(.4,0)};	
	\node at (.4,.6) {\scriptsize{$n$}};	
\end{tikzpicture}\,,
$$
where $\zeta$ is the image of $\widehat{1}\in L^2(R)$ inside $H$.
This bimodule has trivial standard invariant by Example \ref{ex:OuterZAction}, but it does not come from a {\rm II}$_1$-subfactor.
\end{ex}

%%%%%%%%%%%%%%%%%%%%%%%%%%%%%%%%%%%%%%%%%%%%%%%%%
\bibliographystyle{amsalpha}
{\footnotesize
\bibliography{../../bibliography}
}
\end{document}